\documentclass[11pt, oneside]{article}   	
\usepackage{geometry}                		
\usepackage{graphicx}				
								
\usepackage[all]{xy}
\CompileMatrices

\usepackage{amssymb}
\usepackage{amsmath}
\usepackage{stmaryrd}
\usepackage{amsthm}
\usepackage{tikz-cd}
\usepackage{extarrows}
\usepackage[mathscr]{euscript}
\usepackage{mathrsfs} 
\usepackage{verbatim}
\usepackage[noinputenc,nocaptions]{vietnam}
\catcode`\@=11
\def\tabb{\@tabacckludge}
\catcode`\@=12

\usepackage[OT2,T1]{fontenc}
\DeclareSymbolFont{cyrletters}{OT2}{wncyr}{m}{n}
\DeclareMathSymbol{\Sha}{\mathalpha}{cyrletters}{"58}
\DeclareMathSymbol{\Che}{\mathalpha}{cyrletters}{"51}

\usepackage{calligra}
\usepackage{mathrsfs}
\newcommand{\calHom}{\mathscr{H}\mathit{om}}

\newcommand{\Ga}{{\mathbf{G}}_{\rm{a}}}
\newcommand{\Gm}{{\mathbf{G}}_{\rm{m}}}
\DeclareMathOperator{\red}{red}
\DeclareMathOperator{\perf}{perf}
\DeclareMathOperator{\prim}{prim}
\DeclareMathOperator{\Gal}{Gal}

\DeclareMathOperator{\Ext}{Ext}
\DeclareMathOperator{\Hom}{Hom}
\DeclareMathOperator{\End}{End}

\DeclareMathOperator{\GL}{GL}

\DeclareMathOperator{\R}{R}

\DeclareMathOperator{\im}{im}

\newcommand*{\Z}{\ensuremath{\mathbf{Z}}}                        
\newcommand*{\F}{\ensuremath{\mathbf{F}}}                        
\newcommand*{\A}{\ensuremath{\mathbf{A}}}                        
\newcommand*{\N}{\ensuremath{\mathbf{N}}}                        
\renewcommand*{\P}{\ensuremath{\mathbf{P}}}                        
\newcommand*{\calO}{\mathcal{O}}                                  
\newcommand*{\address}{Einstein Institute of Mathematics, The Hebrew University of Jerusalem, Edmond J. Safra Campus, 91904, Jerusalem, Israel}
\newcommand*{\email}{zev.rosengarten@mail.huji.ac.il}

\usepackage{rotating}

\usepackage{bm}

\numberwithin{equation}{section}

\newtheorem{theorem}{Theorem}[section]

\newtheorem{lemma}[theorem]{Lemma}
\newtheorem{proposition}[theorem]{Proposition}
\newtheorem{corollary}[theorem]{Corollary}

\theoremstyle{definition}
  \newtheorem{definition}[theorem]{Definition}
  
\theoremstyle{question}
  \newtheorem{question}[theorem]{Question}
  
  \theoremstyle{remark}
  \newtheorem{example}[theorem]{Example}

\theoremstyle{remark}
  \newtheorem{remark}[theorem]{Remark}

\usepackage[OT2,T1]{fontenc}

\tikzset{commutative diagrams/.cd,
mysymbol/.style={start anchor=center,end anchor=center,draw=none}
}

\usepackage{stmaryrd}
\usepackage{hyperref}

\title{\textbf{PERMAWOUND UNIPOTENT GROUPS}}
\author{Zev Rosengarten \thanks{
While completing this work, the author was supported by Israel Science Foundation Grant No.\,2083/24.}}

\date{}
\begin{document}
\maketitle

\begin{abstract}
We introduce the class of permawound unipotent groups, and show that they simultaneously satisfy certain ``ubiquity'' and ``rigidity'' properties that in combination render them very useful in the study of general wound unipotent groups. As an illustration of their utility, we present two applications: We prove that nonsplit smooth unipotent groups over (infinite) fields finitely-generated over $\F_p$ have infinite first cohomology; and we show that every commutative $p$-torsion wound unipotent group over a field of degree of imperfection $1$ is the maximal unipotent quotient of a commutative pseudo-reductive group, thus partially answering a question of Totaro.
\end{abstract}

\setcounter{tocdepth}{1}
\tableofcontents{}

\section{Introduction}

Unipotent groups play a fundamental role in the theory of linear algebraic groups. To illustrate this, recall that, given a connected linear algebraic group $G$ over a field $k$, the $k$-unipotent radical $\mathscr{R}_{U,k}(G)$ of $G$ is defined to be its maximal smooth connected normal unipotent $k$-subgroup. That there exists a unique such maximal $k$-subgroup is a well-known fact in the theory of linear algebraic groups; see \cite[Ch.\,IV, \S11.21]{borel}. We say that $G$ is {\em $k$-pseudo-reductive}, or just pseudo-reductive when $k$ is clear from context, when $\mathscr{R}_{U,k}(G) = 1$ -- that is, $G$ contains no nontrivial smooth connected normal unipotent subgroups over $k$. Then the quotient $P := G/\mathscr{R}_{u,k}G)$ is pseudo-reductive, and in this manner one may write $G$ as an extension of a pseudo-reductive group by a smooth connected unipotent group. In \cite{cgp}, a classification of pseudo-reductive groups is given, up to the commutative case. That is, all such groups are described explicitly (the classification is somewhat less explicit in characteristic $2$ when $[k: k^2] > 2$), up to a commutative pseudo-reductive group. This leaves one to study unipotent groups. Furthermore, one can say little in general about commutative pseudo-reductive groups, and unipotent groups arise naturally in the study of such groups as well, as the quotients of such groups by their maximal tori. (More about this later.)

Another manner in which unipotent groups arise naturally in the study of general (even pseudo-reductive) groups is via the so-called open cell decomposition associated to a cocharacter. To define these open cells, we need to review some ``dynamic constructions''. Fix a cocharacter $\beta: \Gm \rightarrow G$ (possibly trivial). Then we define a subfunctor $U_G(\beta) \subset G$ by the formula
\[
U_G(\beta)(R) := \{g \in G(R) \mid \lim_{t\to 0} \beta(t)g\beta(t)^{-1} = 1\}
\]
for $k$-algebras $R$. Let us explain what we mean by this. The map $(\Gm)_R \rightarrow G_R$ defined by $t \mapsto \beta(t)g\beta(t)^{-1}$ may or may not extend to a map $c_{\beta}: \A^1_R \rightarrow G_R$. If it does, then it extends uniquely, since $G_R$ is separated over $R$. Then we define $\lim_{t \to 0} \beta(t)g\beta(t)^{-1} := c_{\beta}(0)$.

The subfunctor $U_G(\beta) \subset G$ is represented by a split (see the next paragraph for the definition of splitness) smooth connected unipotent $k$-subgroup scheme of $G$ (which we also denote by $U_G(\beta)$) \cite[Lemma 2.1.5, Prop.\,2.1.10]{cgp}. Let $Z_G(\beta)$ denote the (scheme-theoretic) centralizer in $G$ of $\beta(\Gm)$.
Then the multiplication map
\begin{equation}\label{opencellH}
m_G: U_G(\beta^{-1}) \times Z_G(\beta) \times U_G(\beta) \rightarrow G
\end{equation}
is an open immersion \cite[Prop.\,2.1.8(3)]{cgp}. The image of this map is called the {\em open cell} in $G$ associated to $\beta$, or the {\em open cell} of the pair $(G, \beta)$. For a generically chosen $\beta$, $Z_G(\beta) = Z_G(T)$ for a maximal $k$-torus $T$. This open cell is a very useful tool for studying the group $G$. The mysterious part of the open cell is the quotient $Z_G(T)/T$, which is a unipotent group.

Among unipotent groups, there is a collection of particularly simple ones: namely, those that are split. Recall that a smooth connected unipotent group $U$ is said to be {\em $k$-split} (or just split when $k$ is clear from the context) if it admits a $k$-group scheme filtration {\em over $k$}
\[
1 = U_0 \trianglelefteq U_1 \trianglelefteq \dots \trianglelefteq U_n = U
\]
with $U_{i+1}/U_i \simeq \Ga$ for $0 \leq i < n$. 

In general, for a smooth connected unipotent $k$-group $U$, there is a maximal split unipotent $k$-subgroup $U_{\rm{split}} \subset U$, and this subgroup is normal in $U$ \cite[Th.\,B.3.4]{cgp}. The quotient then does not contain a copy of $\Ga$. These are the more mysterious unipotent groups, and they have a name. As introduced by Jacques Tits \cite{tits} (of whose work \cite[App.\,B]{cgp} is an exposition), a smooth connected unipotent $k$-group $U$ is said to be {\em $k$-wound} (or just wound when $k$ is clear from the context) if any $k$-morphism $\A^1_k \rightarrow U$ from the affine line to $U$ is the constant morphism to a $k$-point of $U$ \cite[Def.\,B.2.1]{cgp}. This is (non-obviously) equivalent to requiring that there be no non-constant $k$-homomorphism $\Ga \rightarrow U$. Indeed, this property follows from woundness because $\Ga \simeq \A^1$ as $k$-schemes, and conversely, it also implies woundness by \cite[Prop.\,B.3.2]{cgp}. Woundness is also equivalent to requiring that $U$ not contain a $k$-subgroup scheme $k$-isomorphic to $\Ga$, as once again follows from \cite[Prop.\,B.3.2]{cgp}.

Woundness, which may be viewed as the analogue for unipotent groups of anisotropicity for tori (while splitness for unipotent groups is the analogue of splitness for tori), nevertheless behaves in the opposite manner to anisotropicity with regard to field extensions, as we now explain. A $k$-torus becomes split over some finite separable extension (equivalently, every torus over a separably closed field is split), while a smooth connected unipotent group becomes split over a finite {\em purely inseparable} extension (equivalently, every smooth connected unipotent group over a perfect field is split) \cite[Cor.\,B.2.7]{cgp}. Similarly, for a finite purely inseparable extension $k'/k$, a $k$-torus is anisotropic over $k$ if and only if it is so over $k'$, while for a (not necessarily algebraic) separable extension $k''/k$, a smooth connected unipotent $k$-group is wound over $k$ if and only if it is so over $k''$ \cite[Prop.\,B.3.2]{cgp}. Because the only wound unipotent group over a perfect field is the trivial group, woundness is only an interesting notion over imperfect fields (and in particular is uninteresting in characteristic $0$). On the other hand, wound unipotent groups exist over every imperfect field, as the following example shows.

\begin{example}(\cite[Ex.\,B.1.1]{cgp})
\label{exofwoundgp}
This example is originally due to Rosenlicht \cite[p.\,46]{rosenlicht}. Let $k$ be an imperfect field of characteristic $p$, let $a \in k- k^p$, and consider the subgroup $W_a$ of $\Ga^2 = \Ga \times \Ga$ defined by the following equation:
\[
X + X^p + aY^p = 0.
\]
One may check that the maps $(X, Y) \rightarrow X + a^{1/p}Y$, $T \mapsto (-T^p, a^{-1/p}(T + T^p))$ define mutually inverse isomorphisms between $W_a$ and $\Ga$ over the purely inseparable extension $k(a^{1/p})$. In particular, $W_a$ is smooth connected unipotent. We claim that it is wound. One way to see this is to note that the projective closure $\overline{W}_a$ of $W_a$, defined in $\P^2_k$ by the equation
\[
XZ^{p-1} + X^p + aY^p = 0,
\]
is regular, hence is the regular completion of the affine curve $W_a$. If $W_a$ were not wound, then it would contain, hence be $k$-isomorphic to, $\Ga$. But the regular completion of $\Ga$ is $\P^1_k$, which has rational point at infinity, while the unique point of $\overline{W}_a \backslash W_a$ only becomes rational over $k(a^{1/p})$. (Another more general approach to seeing woundness is to use \cite[Lem.\,B.1.7,$(1)\Longrightarrow(2)$]{cgp}.)
\end{example}

One may also construct non-commutative examples over every imperfect field. See \cite[Ex.\,2.10]{conradsolvable}.

Woundness enjoys certain permanence properties. It is clearly inherited by smooth connected subgroups, and it is also inherited by extensions: an extension of wound unipotent groups is still wound unipotent, as follows, for example, from the formulation in terms of non-existence of nonzero $k$-homomorphisms from $\Ga$. It is {\em not}, however, inherited by quotients. In fact, every smooth connected unipotent $k$-group $U$ admits an isogeny onto a split group as follows. The group $U_{k^{1/p^n}}$ is split for some $n \geq 0$ (because $U$ splits over a perfect closure of $k$). Let $U^{(p^n)}$ denote the $n$-fold Frobenius twist of $U$ over $k$. We have the following commutative diagram:
\[
\begin{tikzcd}
U^{(p^n)} \arrow{d} \arrow{r}{\sim} \arrow[dr, phantom, "\square"] & U_{k^{1/p^n}} \arrow[dr, phantom, "\square"] \arrow{d} \arrow{r} & U \arrow{d} \\
{\rm{Spec}}(k) \arrow{r}{\sim} & {\rm{Spec}}(k^{1/p^n}) \arrow{r} & {\rm{Spec}}(k) 
\end{tikzcd}
\]
The vertical arrows are the structure maps defining the $k$-scheme structures on $U$ and $U^{(p^n)}$ and the $k^{1/p^n}$-structure on $U_{k^{1/p^n}}$, the first arrow in the bottom row is that induced by the isomorphism $k^{1/p^n} \xrightarrow{\sim} k$ which sends $x$ to $x^{p^n}$, and the second arrow is induced by the inclusion $k \hookrightarrow k^{1/p^n}$. The second square is Cartesian by definition, as is the large outer square. Therefore, so too is the first square. This diagram shows that the $k^{1/p^n}$-splitness of $U_{k^{1/p^n}}$ implies the $k$-splitness of $U^{(p^n)}$. Therefore, the $n$-fold Frobenius map $U \rightarrow U^{(p^n)}$ defines a $k$-isogeny from $U$ onto a split unipotent group. While this example involves forming the quotient by an infinitesimal subgroup, even forming quotients by smooth connected subgroups does not preserve woundness in general, as we shall see soon.

While woundness is not inherited in general by quotients, this paper involves the study of groups with the property that woundness {\em is} inherited when one forms quotients by them (or more precisely, by any of {\em their} quotients). These are the permawound groups, to be defined shortly. First, we mention the notion of {\em semiwound} unipotent groups. These are defined by the same condition as wound groups, and satisfy many of the same properties, except that we do not require them to be smooth or connected. See appendix \ref{semiwoundsection} and in particular Definition \ref{semiwounddef}. 

The following simple definition underlies this entire paper.

\begin{definition}
\label{asymptotunivdef}
Let $k$ be a field. We say that a smooth unipotent $k$-group scheme $U$ is {\em $k$-permawound}, or just permawound when $k$ is clear from context, when the following condition holds: For every right-exact sequence of finite type $k$-group schemes
\[
U \longrightarrow E \xlongrightarrow{\pi} \Ga \longrightarrow 1,
\]
$E$ contains a $k$-subgroup scheme $k$-isomorphic to $\Ga$.
\end{definition}

Stated differently, if $\overline{U}$ is a quotient group of $U$, then semiwound groups are stable under the formation of quotients by $\overline{U}$.

\begin{example}
\label{asunivnonexample}
To get a feel for the meaning of the above definition, let us return to the group $W_a$ of Example \ref{exofwoundgp}. When $p > 2$, that group is not permawound. Indeed, in this case, we may choose $b \in k$ that is not of the form $x^p + ay^p$, because $[k: k^p] \geq p > 2$. Then we consider the $k$-group
\[
E := \{X + aX^p + Y^p - bZ^p = 0\} \subset \Ga^3.
\]
One may show that $E$ does not contain a copy of $\Ga$. Indeed, the leading part $aX^p+Y^p-bZ^p$ has no nontrivial zeroes, so this follows from \cite[Lem.\,B.1.7,$(1)\Longrightarrow(2)$]{cgp}. But one has the exact sequence
\[
W_a \longrightarrow E \longrightarrow \Ga \longrightarrow 0,
\]
where the second map is $(X, Y, Z) \mapsto Z$. Thus $W_a$ does not satisfy the condition of Definition \ref{asymptotunivdef} when $p >2$. 
\end{example}

The problem in Example \ref{asunivnonexample} was that the leading part (usually called the {\em principal part}) of the polynomial defining $W_a$ does not take every value in $k$. That is, it is not {\em universal}. We will show later (Proposition \ref{asunivcrit}) that a smooth group $U$ defined by a $p$-polynomial $F$ whose principal $P$ part has no nontrivial zeroes is permawound precisely when $P$ is universal. In particular, when $p=2$, $W_a$ is permawound precisely when $k$ has degree of imperfection $1$. 

Permawound groups satisfy two fundamental properties which make them extremely useful in the study of general wound unipotent groups. The first, given below, we call ubiquity.

\begin{theorem}$($``Ubiquity''$)$
\label{ubiquityintro}
Let $k$ be a field of finite degree of imperfection. Then for any smooth commutative $p$-torsion semiwound unipotent $k$-group $U$, there is an exact sequence
\[
0 \longrightarrow U \longrightarrow W \longrightarrow V \longrightarrow 0
\]
with $W$ wound, commutative, $p$-torsion, and permawound, and $V$ a vector group.
\end{theorem}

Theorem \ref{ubiquityintro} is not very surprising. We saw in Example \ref{asunivnonexample} how one may begin with a $p$-polynomial $F$ and simply add on a variable to obtain an extension of $\Ga$ by the group $\{F = 0\}$ described by $F$. The proof of Theorem \ref{ubiquityintro} essentially entails repeating this process until one obtains a $p$-polynomial with universal principal part. The hard part is showing that groups defined by such polynomials are indeed permawound (the subtlety being to ensure that there are no semiwound extensions of $\Ga$ by {\em quotients} of such groups).

The utility of the ubiquity property is that many properties that one might wish to prove about wound unipotent groups are inherited by subgroups, hence to prove them for general wound groups -- at least those that are commutative and $p$-torsion -- it suffices to treat the case of permawound groups. In fact, using \cite[Cor.\,B.3.3]{cgp}, if the property one wishes to establish is also inherited by extensions (as many such properties are), then it suffices to treat the $p$-torsion commutative case, hence the permawound case.

The reader may now wonder how this is helpful. Why should the study of permawound groups be any easier than that of general wound unipotent groups? The answer lies in the second fundamental property of such groups, which we call rigidity. In contrast to the ubiquity property, rigidity is rather surprising. To state it, we define for any separably closed field $k$ of finite degree of imperfection a wound unipotent $k$-group $\mathscr{V}$. The definition may be given as follows, though a more thorough discussion will be given in \S \ref{specgpsection}. Take any $p$-basis $\lambda_1, \dots, \lambda_r$ of $k$, let $\mathscr{V} \subset \Ga^{I}$ be the vanishing locus of the following polynomial in $k[X_f\mid f \in I]$:
\[
-X_0 + \sum_{f\in I} \left(\prod_{i=1}^r\lambda_i^{f(i)}\right)X_f^p,
\]
where the subscript $0$ denotes the constant function with value $0$. The rigidity property of permawound groups is the following statement.

\begin{theorem}$($``Rigidity''$)$
\label{rigidityintro}
Let $k$ be a separably closed field of finite degree of imperfection $r$, and let $U$ be a wound permawound unipotent $k$-group scheme. Then $U$ admits a filtration $1 = U_0 \trianglelefteq U_1 \trianglelefteq \dots \trianglelefteq U_m = U$ such that, for each $1 \leq i \leq m$, either $U_i/U_{i-1} \simeq \mathscr{V}$ or $U_i/U_{i-1} \simeq \R_{k^{1/p}/k}(\alpha_p)$.
\end{theorem}

In the above, $\alpha_p$ over an $\F_p$-scheme $S$ denotes as usual the base change to $S$ of the $\F_p$-group scheme $\F_p[X]/(X^p)$ with group law $X\cdot Y := X + Y$, while $\R_{k^{1/p}/k}$ denotes Weil restriction of scalars from $k^{1/p}$ to $k$. The $k$-groups $\mathscr{V}$ and $\R_{k^{1/p}/k}(\alpha_p)$ appearing in Theorem \ref{rigidityintro} are semiwound and satisfy a property known as weak permawoundness to be defined in \S\ref{qwasunivsection}. In fact, the former group is even wound and permawound.

Combining ubiquity and rigidity, one obtains for instance the following: given a separably closed field of finite degree of imperfection, and a property of groups that one wishes to prove for all wound unipotent $k$-groups, if the property is inherited by subgroups and by extensions, then it suffices to prove it for the two groups $\mathscr{V}$ and $\R_{k^{1/p}/k}(\alpha_p)$. In fact, permawound groups -- and in particular the ubiquity and rigidity properties -- are useful even for proving properties of wound groups over fields that may not be separably closed or of finite degree of imperfection. Indeed, the latter is hardly a condition because essentially any property of wound groups that one wishes to prove may be easily reduced to the finite degree of imperfection case by a simple spreading out argument. (For instance, any unipotent group is defined over a finitely-generated field, as are maps between finite type schemes over fields.) The separably closed condition is also not too restrictive for many applications, because one can often pull properties down from finite separable extensions.

As a first illustration of the utility of permawound groups -- in particular of their ubiquity and rigidity properties -- we give two applications to questions of independent interest. The first concerns the cohomology of smooth unipotent groups.

\begin{theorem}
\label{infofcohomintro}
Let $K$ be an infinite finitely generated extension field of $\F_p$, and let $U$ be a smooth unipotent $K$-group scheme that is not split. Then ${\rm{H}}^1(K, U)$ is infinite.
\end{theorem}

Note that this is a statement about unipotent groups over fields that are not separably closed. Nevertheless, the rigidity property plays a crucial role in the proof of Theorem \ref{infofcohomintro}.

The second application is to a question of Totaro concerning unipotent quotients of commutative pseudo-reductive groups. As stated earlier, in \cite{cgp}, a classification of pseudo-reductive groups is given, up to the commutative case. That is, all such groups are described explicitly, up to a commutative pseudo-reductive group, and it appears that one can in general say little about the nature of groups of the latter type. It is therefore of significant interest in the general theory of linear algebraic groups over imperfect fields to study commutative pseudo-reductive groups. 

Let $C$ be a commutative pseudo-reductive $k$-group, and let $T \subset C$ be its maximal torus. Then $U := C/T$ is a smooth connected commutative unipotent $k$-group. Furthermore, $U$ is necessarily wound. Indeed, because woundness and pseudo-reductivity are insensitive to passage to separable extensions, we may assume that $k$ is separably closed, and in that case the woundness of $U$ follows from the fact that $\Ext^1(\Ga, \Gm) = 0$ \cite[Ch.\,III, \S6, Cor.\,5.2]{demazuregabriel}. We refer to $U$ as the {\em maximal unipotent quotient} of $C$. A natural question is which commutative wound unipotent $k$-groups arise as maximal unipotent quotients of commutative pseudo-reductive $k$-groups. When $k = k_0(a, b)$ is a rational function field in two variables over an arbitrary field $k_0$ of characteristic $p$, Totaro constructed a commutative wound unipotent $k$-group $U$ that cannot be expressed as such a quotient \cite[Ex.\,9.7]{totaro}. In fact, one can show that his example works over any field of degree of imperfection $> 1$, with the roles of $a$ and $b$ played by any pair of $p$-independent elements of $k$. Nevertheless, for fields of degree of imperfection $1$, Totaro posed the following optimistic question.

\begin{question}[{\cite[Question 9.11]{totaro}}]
\label{totaroquestion}
If $k$ is a field of degree of imperfection $1$, is every commutative wound unipotent $k$-group the maximal unipotent quotient of a commutative pseudo-reductive $k$-group?
\end{question}

We will use the ubiquity and rigidity properties of permawound groups to prove the following affirmative answer to Totaro's question when the group in question is $p$-torsion.

\begin{theorem}
\label{totquesyesintro}
Let $k$ be a field of degree of imperfection $1$. Then every commutative $p$-torsion wound unipotent $k$-group is the maximal unipotent quotient of a commutative pseudo-reductive $k$-group.
\end{theorem}

This paper may be roughly broken up into three parts. The first part, consisting of \S\S\ref{subgpsvecsection}--\ref{asunivsubgpsvecgpssection}, is concerned with characterizing those subgroups of vector groups that are permawound, culminating in the characterization Theorem \ref{asunivsmoothcrit} of such groups, as well as the corollary rigidity principle (Theorem \ref{ubiquity}). The second part of the paper, consisting of \S\S\ref{specgpsection}--\ref{rigiditysection}, consists of the proof of the rigidity property of permawound groups (Theorem \ref{filtrationasuniv}). The third and final part -- \S\S\ref{cohunipgpsection}--\ref{unipquotcpredsec} -- consists of applications of permawound groups. The reader may wish to read this part as initial motivation before the rest of the paper.

\subsection{Acknowledgements}

Thanks to Michel Brion for his suggestions regarding the naming of the main objects of this paper. I also thank the anonymous referees for their careful and thorough reading of the manuscript and for their many helpful suggestions which have been incorporated into the text.

\subsection{Notation and Conventions}

Throughout the paper, $k$ denotes a field, and, when $p$ appears, it means that $k$ has positive characteristic $p$. The symbols $k_{\perf}$ and $k_s$ denote perfect and separable closures of $k$, respectively.
 
For a scheme $S$, we regard the category of $S$-group schemes as a fully faithful subcategory of the category of fppf group sheaves on $S$. In particular, when we say that a sequence of $S$-group schemes is exact, or that a map of $S$-group schemes is surjective, et cetera, we mean that the corresponding property of fppf sheaves holds. Furthermore, all cohomology is in the category of fppf sheaves unless stated otherwise.

\section{Subgroups of vector groups}
\label{subgpsvecsection}

In this section we analyze subgroups of vector groups, the main point being a result which yields a filtration on commutative semiwound unipotent groups such that each of the factor groups may be described as a hypersurface in a vector group (Proposition \ref{swdfilthyper}). We begin by recalling some basic notions pertaining to $p$-polynomials. Let $k$ be a field of characteristic $p > 0$. A {\em $p$-polynomial} is a polynomial $F \in k[X_1, \dots, X_n]$ of the form $F(X_1, \dots, X_n) := \sum_{i=1}^n \sum_{j = 0}^{N_i} c_{i,j}X_i^{p^j}$ with $c_{i,j} \in k$. We remark that part of the data of the $p$-polynomial is the polynomial ring in which it lies. So, for instance, we regard $X_1 \in k[X_1]$ and $X_1 \in k[X_1, X_2]$ as being different $p$-polynomials. The reason for this is that the first one describes the trivial group $\{X_1 = 0\} \subset \Ga$, while the latter describes the copy of $0 \times \Ga$ of $\Ga$ sitting inside $\Ga^2$. The {\em principal part} of $F$ is the $p$-polynomial $P(X_1, \dots, X_n) := \sum_{i \in I} c_{i, n_i}X_i^{p^{n_i}} \in k[X_1, \dots, X_n]$, where $I \subset \{1, \dots, n\}$ is the set of indices such that $X_i$ appears in $F$, and $n_i$ is the maximal $j$ such that $c_{i,j} \neq 0$. Note that we consider $P$ as lying in the same polynomial ring as $F$, regardless of how many variables actually appear in $F$. We say that $F$ is {\em separable} when it defines a smooth $k$-subgroup scheme of $\Ga^n$ -- equivalently, when $F$ has nonzero linear part.
All of this is previously defined terminology; see \cite[App.\,B]{cgp}. Now we introduce some new terminology.

\begin{definition}
\label{ppolymonred}
We say that a $p$-polynomial $P \in k[X_1, \dots, X_n]$ is {\em monogeneous} when each variable appears in at most one term of $P$. That is, we may write $P = \sum_{i=1}^n c_iX_i^{p^{n_i}}$, where we allow $c_i = 0$. (Equivalently, the part involving each variable is homogeneous, hence the name.) We say that a monogeneous $p$-polynomial $P$ is {\em reduced} when the equation $P(x_1, \dots, x_n) = 0$ has no solutions in $k$ other than $\vec{0}$. We say that an arbitrary $p$-polynomial $F$ is reduced when its principal part is.
\end{definition}

The first part of the following result follows from \cite[Lem.\,B.1.7,(2)$\Longrightarrow$(3)]{cgp}.

\begin{proposition}
\label{simplifyppoly}
Let $F \in k[X_1, \dots, X_n]$ be a nonzero $p$-polynomial. Then there exist a $k$-group scheme automorphism $\sigma$ of $\Ga^n$ and $1 \leq m \leq n$ such that $F \circ \sigma \in k[X_m, \dots, X_n]$ is a reduced $p$-polynomial. If $F$ is monogeneous, then one may choose $\sigma$ so that $F \circ \sigma$ is monogeneous.
\end{proposition}

\begin{proof}
If $n = 0$, then $F = 0$, so assume that $n > 0$. If $F$ can be transformed by an automorphism of $\Ga^n$ (respectively, an automorphism preserving monogeneity of $F$ in the monogeneous case) to lie in $k[X_i]_{i \in I}$ for some proper subset $I \subset \{1, \dots, n\}$, then we are done by induction. So assume that every automorphism of $\Ga^n$ (resp.\,every automorphism preserving monogeneity of $F$) transforms $F$ into a $p$-polynomial involving all of the variables $X_1, \dots, X_n$. We may assume that $F$ is such that $\sum_i {\rm{deg}}_{X_i}(F)$ is minimal among all transformations of $F$ via automorphisms of $\Ga^n$ (resp.\,automoprhisms preserving monogeneity). Then we claim that $F$ is reduced. 

Indeed, let $P$ be the principal part of $F$, and assume that $P$ admits a nontrivial zero $(r_1, \dots, r_n) \in k^n$. Let $i_0$ be such that $r_{i_0} \neq 0$ and such that
\[
p^{d_{i_0}} := {\rm{deg}}_{X_{i_0}}(F) = \max\{ p^{d_i} := {\rm{deg}}_{X_i}(F) \mid r_i \neq 0\}.
\]
Let $\sigma$ denote the following invertible change of variables on $\Ga^n$:
\[
X_i \mapsto 
\begin{cases}
X_i & r_i = 0\\
r_{i_0}X_{i_0} & i = i_0\\
X_i + r_iX_{i_0}^{p^{d_{i_0}-d_i}} & \mbox{else}.
\end{cases}
\]
Then $F\circ \sigma$ is still monogeneous if $F$ is, and has the same degree as $F$ in $X_i$ for all $i$ except for $i = i_0$, for which the degree of $F\circ \sigma$ is strictly lower. This violates minimality, so we conclude that $F$ was already reduced.
\end{proof}

The proof of the proposition below amounts to an algorithm, akin to Gaussian elimination, for, given a subgroup of a vector group described by a system of $p$-polynomials, describing it by a new system of equations of a relatively simple form.

\begin{proposition}
\label{gausselimunigps}
Let $H_1, \dots, H_r \in k[X_1, \dots, X_n]$ be $p$-polynomials, and let $U := \cap_i \{H_i = 0\} \subset \Ga^n$ be the $k$-group that they define. Then there exists a $k$-group scheme automorphism $\sigma$ of $\Ga^n$, and $p$-polynomials $F_1, \dots, F_m \in k[X_1, \dots, X_n]$ such that $m \leq r$ and
\begin{itemize}
\item[(i)] For $1 \leq i < m$, $F_i \in k[X_i, X_{i+1}, \dots, X_n] - k[X_{i+1}, \dots, X_n]$;
\item[(ii)] If $m > 0$, then for some $0 \leq j \leq n-m$, $F_m \in k[X_{m+j}, \dots, X_n]$ is a reduced $p$-polynomial $($as an element of $k[X_{m+j}, \dots, X_n]$$)$; and
\item[(iii)]  $\sigma(U)$ is the $k$-subgroup scheme of $\Ga^n$ described as the vanishing locus of the $F_i$.
\end{itemize}
\end{proposition}

\begin{proof}
We proceed by induction on the pair $(n, r)$ ordered lexicographically. The cases $n = 0$ and $r = 0$ are trivial, so assume that $n, r > 0$. We first treat the case $r = 1$. Then $U$ is described by a single $p$-polynomial $H$. If $H$ vanishes, then we are reduced to the case $r =0$, so assume that $H \neq 0$. Then Proposition \ref{simplifyppoly} implies that a change of variables transforms $H$ into a reduced $p$-polynomial $F_1$ on the last $t$ coordinates for some $t \leq n$, so taking $m = 1$, $F_1$ does the job.

Now consider the general case. If there is a family of $p$-polynomials describing $U$ such that none of the polynomials involve $X_1$ , then $U$ is just the product of the first $\Ga$ factor in $\Ga^n$ and a subgroup of $\Ga^{n-1}$, so we are done by induction. We therefore assume that every family of $p$-polynomials defining $U$ involves $X_1$ in at least one of the equations. Let $G_1, \dots, G_s$ be such a family chosen so that $p^{d_1} := {\rm{deg}}_{X_1}(G_1)$ is minimized among all $p$-polynomials involving $X_1$ which appear in a family of $\leq r$ $p$-polynomials defining $U$. If $s < r$, then we are done by induction, so assume that $s = r$. For each $i > 1$, one has $G_i \equiv R_i \pmod{G_1}$ for some $p$-polynomial $R_i$ such that ${\rm{deg}}_{X_1}(R_i) < p^{d_1}$ \cite[Lem.\,6.3]{rosmodulispaces}. It follows from the minimality of $p^{d_1}$ that $X_1$ does not appear in any of the $R_i$ for $i > 1$. Further, $G_1$ together with the $R_i$ defines the same group $U$ as the original $G_i$. If $R_2, \dots, R_r$ all vanish, then we are in the already-treated case $r = 1$. Otherwise, if the $R_i$ are not all $0$, then the family $R_2, \dots, R_r \in k[X_2, \dots, X_n]$ defines a subgroup which by induction may also be described after a change of variables $\sigma$ on the last $n - 1$ coordinates (so that $G_1$ still involves the variable $X_1$ even after the change of variables) by a family $F_1, \dots, F_t$ of $p$-polynomials as in the statement of the proposition (so in particular $t < r$) in which necessarily $t > 0$. Then $G_1$ together with the $F_i$ yields a defining family of $p$-polynomials for $\sigma(U)$ of the desired shape.
\end{proof}

We next show that vector subgroups of commutative groups with vanishing Verschiebung may be split off as direct factors.

\begin{proposition}
\label{splitvecsubgps}
Let $k$ be a field of characteristic $p$. Then any exact sequence of commutative $k$-group schemes with vanishing Verschiebung
\begin{equation}
\label{splitvecsubgpseqn1}
0 \longrightarrow \Ga^n \xlongrightarrow{i} U \longrightarrow W \longrightarrow 0
\end{equation}
splits.
\end{proposition}

\begin{proof}
This follows from \cite[Ch.\,IV, \S3, Corr.\,6.7]{demazuregabriel}, because $\Hom(\Ga^n, \Ga)$ is a free $k[F]$-module, where $F$ denotes the (relative) Frobenius.
\end{proof}

As a consequence of Propositions \ref{gausselimunigps} and \ref{splitvecsubgps}, we will now deduce that every subgroup of a vector group is itself the product of a vector group and a group which may be filtered by unipotent groups described as hypersurfaces in vector groups.

\begin{proposition}
\label{filterhypersurf}
Let $k$ be an infinite field of characteristic $p > 0$, and let $G$ be a commutative $k$-group scheme with vanishing Verschiebung. Then $G \simeq V \times U$, where $V$ is a vector group, and there is a filtration $0 = U_0 \subset U_1 \subset \dots \subset U_n = U$ such that, for each $1 \leq i \leq n$, $U_i/U_{i-1}$ is a $k$-group scheme of the form $\{F_i = 0\} \subset \Ga^{n_i}$ for some reduced $p$-polynomial $F_i \in k[X_1, \dots, X_{n_i}]$. In particular, each $U_i/U_{i-1}$ is semiwound.
\end{proposition}

\begin{proof}
The assertion about the semiwoundness of the quotients $U_i/U_{i-1}$ follows from \cite[Lem.\,B.1.7,$(1)\Longrightarrow(2)$]{cgp}. Let $V \subset G$ be the subgroup generated by all $k$-homomorphisms $\Ga \rightarrow G$. By Proposition \ref{splitvecsubgps}, this is a vector group such that $G/V$ is semiwound, and the extension $$0 \longrightarrow V \longrightarrow G \longrightarrow G/V \longrightarrow 0$$ splits, so we may assume that $G = U$ is semiwound. By \cite[Ch.\,IV, \S3, Th.\,6.6]{demazuregabriel}, $U$ is a $k$-subgroup of a vector group $\Ga^n$, and by \cite[Ch.\,IV, \S3, Cor.\,6.8]{demazuregabriel}, the quotient $\Ga^n/U$ is isomorphic to some $\Ga^t$. Since the homomorphisms $\Ga^n \rightarrow \Ga$ are exactly the $p$-polynomials in $n$ variables, it follows that $U$ may be described as the vanishing locus of $t$ $p$-polynomials inside $\Ga^n$. We prove that $U$ has a filtration of the desired form by induction on $t$, noting that $t = 0$ is impossible unless $n = 0$ and $U = 0$ because $U$ is semiwound. So assume that $U \neq 0$ and therefore cannot be described by $0$ equations.

Upon a change of variables, we may describe $U$ inside $\Ga^n$ by a system of equations $F_1, \dots, F_m$ as in Proposition \ref{gausselimunigps} with $0 < m \leq t$. If $m < t$, then we are done by induction, so assume that $m = t$. Then $F_t \in k[X_{t+r}, \dots, X_n]$ is a reduced $p$-polynomial in $\Ga^{n-t-r+1}$, and its vanishing locus $U_{F_t}$ is a semiwound $k$-subgroup scheme of $\Ga^{n-t-r+1}$ by \cite[Lem.\,B.1.7,$(1)\Longrightarrow(2)$]{cgp}. We claim that the obvious map $U \rightarrow U_{F_t}$ is fppf surjective with kernel described by $< t$ equations, in which case the proof is complete by induction.

First we note that the kernel of the map $U \rightarrow U_{F_t}$ is the subgroup scheme of $\Ga^{\{1, \dots, t+r-1\}}$ described as the vanishing locus of the $p$-polynomials $F_i(X_1, \dots, X_{t+r-1}, 0, \dots, 0)$ for $1 \leq i < t$. In particular, it is defined by $< t$ equations. It only remains, therefore, to check that the map $U \rightarrow U_{F_t}$ is fppf surjective. For this, consider for each $1 \leq i < t$ the $k$-subgroup $W_i \subset \Ga^{\{i, \dots, n\}}$ defined as the vanishing locus of $F_i, \dots, F_t$, and let $W_t := U_{F_t}$. Then we have maps $W_i \rightarrow W_{i+1}$ for $1 \leq i < t$ given by projection onto the last $n-i$ coordinates for $i < t-1$ and projection onto the last $n-t-r+1$ coordinates when $i=t-1$. We claim that each such map is fppf surjective, which will prove the result. To see this, we simply note that the fiber of $W_i$ above a point $(x_{f(i+1)}, \dots, x_n) \in W_{i+1}(R)$ for some $k$-algebra $R$, where
\[
f(i) := \begin{cases}
i, & i < t\\
t+r, & i = t
\end{cases}
\]
is the spectrum of the $R$-algebra $$F_i(X_i, \dots, X_{f(i+1)-1}, 0, \dots, 0) + F_i(0, \dots, 0, x_{f(i+1)}, \dots, x_n) = 0.$$ This is an fppf $R$-algebra because $F_i$ involves the variable $X_i$ (Proposition \ref{gausselimunigps}(i)).
\end{proof}

\begin{lemma}
\label{extGavertriv}
Suppose given a short exact sequence of commutative unipotent $k$-group schemes
\[
0 \longrightarrow U \longrightarrow E \longrightarrow \Ga^n \longrightarrow 0.
\]
If $U$ is semiwound with vanishing Verschiebung, then $E$ has vanishing Verschiebung.
\end{lemma}

\begin{proof}
We have the Frobenius twisted exact sequence
\[
0 \longrightarrow U^{(p)} \longrightarrow E^{(p)} \longrightarrow \Ga^n \longrightarrow 0
\]
(because $(\Ga^n)^{(p)} \simeq \Ga^n$). Because $U$ and $\Ga^n$ have vanishing Verschiebung, the Verschiebung of $E$ descends to a homomorphism $\Ga^n \simeq E^{(p)}/U^{(p)} \rightarrow U$. Because $U$ is semiwound, this homomorphism vanishes.
\end{proof}

Subgroups of vector groups play a fundamental role in the theory of unipotent groups because every unipotent group admits a filtration whose successive quotients are subgroups of vector groups. A somewhat subtler statement says that, for $U$ semiwound, we may ensure that the quotients inherit semiwoundness and thereby ensure that there is a filtration by semiwound hypersurface groups.

\begin{proposition}
\label{swdfilthyper}
Let $k$ be an infinite field, and let $U$ be a commutative semiwound unipotent $k$-group. Then there is a filtration $1 = U_0 \subset U_1 \subset \dots \subset U_m = U$ such that, for all $1 \leq i \leq m$, $U_i/U_{i-1} \simeq \{F_i = 0\} \subset \Ga^{n_i}$ for some reduced $p$-polynomial $F_i \in k[X_1, \dots, X_{n_i}]$.
\end{proposition}

\begin{proof}
In light of Proposition \ref{filterhypersurf}, it suffices to prove the existence of a filtration $1 = U_0 \subset U_1 \subset \dots \subset U_m = U$ such that, for all $1 \leq i \leq m$, $U_i/U_{i-1}$ is semiwound with vanishing Verschiebung. Because $U$ is unipotent, there is such a filtration of some length $m$, possibly sans the semiwoundness condition. We then prove that $U$ has a filtration of the required sort by induction on $m$. The case $m = 0$ is trivial, so assume that $m > 0$. Let $V \subset \overline{U} := U/U_1$ denote its maximal split unipotent $k$-subgroup, a vector group by \cite[Th.\,B.2.5]{cgp}. Let $\pi\colon U \rightarrow \overline{U}$ denote the quotient map, and let $U_1' := \pi^{-1}(V) \subset U$. By Lemma \ref{extGavertriv}, $U_1'$ has vanishing Verschiebung. Since $U/U_1'$ is semiwound and admits a filtration of length $< m$ by Verschiebung-trivial groups, we are done by induction.
\end{proof}

\section{Universal $p$-polynomials}

In this section we introduce universal $p$-polynomials, which are intimately connected to permawound groups, playing a similar role in the theory of permawound groups to that played by reduced $p$-polynomials in the theory of wound unipotent groups. (More precisely, those reduced $p$-polynomials with universal principal part play the analogous role to that played by reduced $p$-polynomials.) Here is the key definition.

\begin{definition}
\label{universalpolyn}
Let $P \in k[X_1, \dots, X_n]$ be a monogeneous $p$-polynomial. Then $P$ defines a group homomorphism $k^n \rightarrow k$, and we say that $P$ is {\em universal over $k$}, or just universal if $k$ is clear from context, when this homomorphism is surjective.
\end{definition}

\begin{example}
If $k$ is imperfect and $p^n > 2$, then the monogeneous $p$-polynomial $X^{p^n} + aY^{p^n}$ (with $a \in k$) is not universal. Indeed, if it is universal, then $[k: k^{p^n}] \leq 2$. But $$[k: k^{p^n}] = \prod_{i=0}^{n-1} [k^{p^i}: k^{p^{i+1}}] = \prod_{i=0}^{n-1} [k: k^p] = [k: k^p]^n$$ is a positive power of $p^n$ because $k$ is imperfect.
\end{example}

Recall that a set $S$ of elements of a characteristic $p$ field $k$ is called a $p$-basis for $k$ if any of the following equivalent conditions hold: 
\begin{itemize}
\item[(i)] $$k = \bigoplus_f \left(\prod_{\lambda \in S} \lambda^{f(\lambda)}\right)k^p,$$ where the sum runs over all maps $f\colon S \rightarrow \{0, 1, \dots, p-1\}$ which vanish at all but finitely many elements of $S$.
\item[(ii)] The elements $d\lambda$ form a basis for $\Omega^1_k$ as $\lambda$ runs over $S$.
\item[(iii)] The elements $\lambda^{1/p}$ with $\lambda \in S$ form a minimal set of generators of $k^{1/p}/k$.
\end{itemize}
Recall also that the degree of imperfection of $k$ is the size of any $p$-basis of $k$.

Now we introduce a piece of notation that we will use repeatedly in the sequel.

\begin{definition}
\label{Imdef}
For an integer $m \geq 0$, we denote by $I_m$ the set of functions $\{1, \dots, r\} \rightarrow \{0, 1, \dots, p^m-1\}$. When $m = 1$, we also denote this set by $I$.
\end{definition}

\begin{example}
\label{univbasicexample}
Suppose that $\{\lambda_1, \dots, \lambda_r\}$ is a $p$-basis for $k$. Let $$P(X_f\mid f \in I) := \sum_{f \in I} \left(\prod_{i=1}^r \lambda_i^{f(i)}\right)X_f^p.$$ Then $P$ is reduced and universal. Indeed, this is just a reformulation of condition (i) above in the definition of a $p$-basis.
\end{example}

The following proposition gives a numerical criterion for a reduced monogeneous $p$-polynomial to be universal.

\begin{proposition}
\label{monppolynuniv}
Let $k$ be a field of characteristic $p > 0$ and degree of imperfection $r \leq \infty$, and let $P \in k[X_1, \dots, X_n]$ be a monogeneous $p$-polynomial. 
\begin{itemize}
\item[(i)] If $P$ is reduced, then 
\begin{equation}
\label{monppolynuniveqn1}
\sum_{i=1}^n {\rm{deg}}_{X_i}(P)^{-r} \leq 1,
\end{equation}
where by convention, $1^{-\infty} = 1$ and $c^{-\infty} = 0$ for $c > 1$. Equality holds if and only if $P$ is universal.
\item[(ii)] Suppose that $r < \infty$. If
\begin{equation}
\label{monppolynuniveqn2}
\sum_{i=1}^n {\rm{deg}}_{X_i}(P)^{-r} = 1,
\end{equation}
then $P$ is reduced if and only if it is universal.
\end{itemize}
\end{proposition}

\begin{proof}
When $k$ has infinite degree of imperfection, this just says that a reduced monogeneous $p$-polynomial can contain at most one term of degree $1$ -- which is clear -- and that it is universal if and only if there is a term of degree $1$. Clearly, this is sufficient for universality. For necessity, note that if no variable appears with degree $1$, then $P(k, \dots, k) \subset k$ is contained in a finite-dimensional $k^p$-subspace of $k$, hence cannot be all of $k$. 

Now suppose that $k$ has degree of imperfection $r < \infty$, and choose a $p$-basis $\lambda_1, \dots, \lambda_r$ of $k$. For each $m \geq 0$, let $F_m \in k[X_f \mid f \in I_m]$ be the monogeneous $p$-polynomial $$F_m := \sum_{f \in I_m} \left(\prod_{i=1}^r \lambda_i^{f(i)} \right)X_f^{p^m}.$$ Then $F_m$ is reduced and universal.

Now let $P \in k[Y_1, \dots, Y_n]$ be a monogeneous $p$-polynomial, let $p^d_i := {\rm{deg}}_{Y_i}(P)$ for each $i$, and let $N := \max_i d_i$. For each variable $Y_i$, replace $Y_i$ in $P$ by $F_{m_i}(X_{i, f} \mid f \in I_{m_i})$, where $m_i := N - d_i$. This yields a new $p$-polynomial $\tilde{P}$ that is homogeneous. Furthermore, because each $F_m$ is reduced and universal, one has that $\tilde{P}$ is reduced, respectively universal, if and only if $P$ is reduced, respectively universal. Further, if we denote the left side of (\ref{monppolynuniveqn1}) by $\phi(P)$, then $\phi(P) = \phi(\tilde{P})$. Thus both parts (i) and (ii) are reduced to the case in which $P$ is homogeneous, say of degree $p^N$.

(i) This says that a reduced homogeneous $p$-polynomial of degree $p^N$ has $\leq p^{rN}$ terms, with equality if and only if it is universal. Both assertions follow from the fact that
\begin{equation}
\label{monppolynuniveqn3}
[k: k^{p^N}] = \prod_{i=0}^{N-1} [k^{p^i}: k^{p^{i+1}}] = \prod_{i=0}^{N-1} [k: k^p] = p^{rN}.
\end{equation}

(ii) The only if direction follows from part (i). For the if direction, suppose that $P$ is not reduced. The condition (\ref{monppolynuniveqn2}) says that $P$ contains exactly $p^{rN}$ terms. The nonreducedness of $P$ says that the coefficients of $P$ are linearly dependent over $k^{p^N}$. Equation (\ref{monppolynuniveqn3}) then implies that these coefficients do not span $k$ over $k^{p^N}$ -- that is, $P$ is not universal.
\end{proof}

\begin{proposition}
\label{universaloverextensions}
Let $k$ be a field of characteristic $p$ and degree of imperfection $r \leq \infty$, let $P \in k[X_1, \dots, X_n]$ be a monogeneous $p$-polynomial, and let $K/k$ be an extension field.
\begin{itemize}
\item[(i)] If $K/k$ is separable, then $P$ is reduced over $k$ if and only if it is reduced over $K$.
\item[(ii)] If $K/k$ is separable and $K$ also has degree of imperfection $r$ $($for instance, if $K/k$ is separable algebraic$)$, then $P$ is universal over $k$ if and only if it is universal over $K$.
\end{itemize}
\end{proposition}

\begin{proof}
Assertion (i) is \cite[Lem.\,2.8]{rosmodulispaces}. For (ii), Proposition \ref{simplifyppoly} allows us to make a change of variables and thereby assume that $P \in k[X_1, \dots, X_n]$ is reduced (possibly after replacing $n$ by a smaller value and renaming). Then $P$ is also reduced over $K$ by (i), and because $k$ and $K$ have the same degree of imperfection, the assertion follows from Proposition (\ref{monppolynuniv})(i).
\end{proof}

\section{Ext groups and universal $p$-polynomials}

In this section we relate universal $p$-polynomials to the groups $\Ext^i(\Ga, U)$ for $i = 1, 2$ with $U$ unipotent. This will play a crucial role in our characterization in \S \ref{asunivsubgpsvecgpssection} of those $p$-polynomials defining permawound groups.

We begin with the following simple lemma.

\begin{lemma}
\label{leadcoeff}
Let $F \in k[X_1, \dots, X_n]$ be a reduced $p$-polynomial with principal part $P$, and suppose given $Y_1, \dots, Y_n \in k[T]$ not all constant. Then the leading coefficient of $F(Y_1, \dots, Y_n) \in k[T]$ is of the form $P(c_1, \dots, c_n) \neq 0$ for some $c_i \in k$.
\end{lemma}

\begin{proof}
Let $d_i := {\rm{deg}}_{X_i}(F)$ and $D := \max_i d_i{\rm{deg}}(Y_i(T))$. Then $D > 0$, and $${\rm{deg}}(F(Y_1(T), \dots Y_n(T)) \leq D.$$ If we let $c_i$ denote the coefficient of $T^{D/d_i}$ in $Y_i(T)$, then some $c_j \neq 0$ (because the $Y_i$ are not all constant), and $T^D$ appears in $F(Y_1(T), \dots, Y_n(T))$ with coefficient $P(c_1, \dots, c_n) \neq 0$, where the nonvanishing is because $P$ is reduced.
\end{proof}

Let $\mathscr{F}\colon Sh_{k} \rightarrow {\rm{Ab}}$ be an additive functor from the category of sheaves on the (big) fppf site of ${\rm{Spec}}(k)$ to the category of abelian groups. Then the natural $k$-action on $\Ga$ via $\lambda \cdot X = \lambda X$ for $X \in \Ga$ induces a $k$-vector space structure on $\mathscr{F}(\Ga)$. In particular, we apply this when $\mathscr{F} := \Ext^1(\cdot, \mathscr{G})$ for $\mathscr{G}$ a sheaf. 

The following result allows us to give an intrinsic characterization of those hypersurface groups defined by a reduced $p$-polynomial with universal principal part. This will, among other things, play a crucial role in our proof of the criterion for permawoundness (Theorem \ref{asunivsmoothcrit}).

\begin{proposition}
\label{univiffH^1primfd}
Let $F \in k[X_1, \dots, X_n]$ $($with $n > 0$$)$ be a reduced $p$-polynomial with principal part $P$, and let $U := \{F = 0\} \subset \Ga^n$ be the $k$-group scheme it defines. Then $P$ is universal if and only if the $k$-vector space $\Ext^1(\Ga, U)$ is finite-dimensional.
\end{proposition}

\begin{proof}
The group $U$ is defined by the exact sequence
\[
0 \longrightarrow U \xlongrightarrow{i} \Ga^n \xlongrightarrow{F} \Ga \longrightarrow 0.
\]
By the effectivity of affine descent for relatively affine schemes, every element of $\Ext^1(\Ga, U)$ is represented by a scheme $E$. Furthermore, $E$ has vanishing Verschiebung by Lemma \ref{extGavertriv}. Therefore, by Proposition \ref{splitvecsubgps}, $i_*(E) = 0 \in \Ext^1(\Ga, \Ga^n)$. It follows that $\Ext^1(\Ga, U)$ may be identified with the cokernel of the map $F_*\colon \Hom(\Ga, \Ga^n) \rightarrow \Hom(\Ga, \Ga)$. Furthermore, identifying elements of $\Hom(\Ga, \Ga)$ with $p$-polynomials $f \in k[T]$, the $k$-linear action is given by $(\lambda \cdot f)(T) := f(\lambda T)$ for $\lambda \in k$. 

First suppose that $P$ is not universal, and we will show that $\Ext^1(\Ga, U)$ is infinite-dimensional. We are free to multiply $F$ by a nonzero constant and thereby assume that $1 \notin P(k, \dots, k)$. Let $p^N := \max_i {\rm{deg}}_{X_i}(F)$. Then we claim that the classes of the $p$-polynomials $T^{p^m}$ in $\Ext^1(\Ga, U)$ are linearly independent over $k$ as $m$ varies over the integers $\geq N$. We need to show that, for any $\lambda_N, \dots, \lambda_s \in k$ with $\lambda_s \neq 0$, the $p$-polynomial $\sum_{i=N}^s \lambda_i^{p^i}T^{p^i}$ does not lie in $F_*(\Hom(\Ga, \Ga^n))$. That is, we need to show that one cannot write $$\sum_{i=N}^s \lambda_i^{p^i}T^{p^i} = F(Y_1(T), \dots, Y_n(T))$$ with $Y_i \in k[T]$ $p$-polynomials. If one does have a solution to this equation, then Lemma \ref{leadcoeff} implies that equating leading coefficients yields an equation $$P(c_1, \dots, c_n) = \lambda_s^{p^s}.$$ If $p^{d_i} := {\rm{deg}}_{X_i}(P)$, then by assumption $d_i \leq N \leq s$, so dividing by $\lambda^{p^s}$, one obtains that $$P(c_1/\lambda_s^{p^{s-d_1}}, \dots, c_n/\lambda_s^{p^{s-d_n}}) = 1,$$ in violation of the assumption that $1 \notin P(k, \dots, k)$. This completes the proof of the if direction.

Conversely, suppose that $P$ is universal, and we will prove that $\Ext^1(\Ga, U)$ is finite-dimensional. If $k$ has infinite degree of imperfection, then $F \in k[X]$ has degree $1$ by Proposition \ref{monppolynuniv}, hence $U = 0$ and the assertion is immediate, so assume that $k$ has finite degree of imperfection. Let $p^N := \max_i {\rm{deg}}_{X_i}(F)$. Then we claim that every element of $\Ext^1(\Ga, U)$ is represented by a $p$-polynomial of degree $< p^N$. Assuming this, we are done because the $k$-linear action of $\lambda \in k$ sends $a_iT^{p^i}$ to $a_i\lambda^{p^i}T^{p^i}$, and $[k: k^{p^i}] < \infty$ because $k$ has finite degree of imperfection.

To prove that every element of $\Ext^1(\Ga, U)$ is indeed the image of a $p$-polynomial of degree $< p^N$, we proceed by induction on the degree, the degree $< p^N$ case being trivial. So suppose given $m \geq N$. It suffices to show that every monomial $aT^{p^m}$ with $a \in k$ is equivalent modulo $F_*(\Hom(\Ga, \Ga^n))$ to a $p$-polynomial of degree $< p^N$. Because $P$ is universal, we may choose $c_1, \dots, c_n \in k$ such that $P(c_1, \dots, c_n) = a$. Then consider $F(Y_1(T), \dots, Y_n(T))$, where $Y_i(T) := c_iT^{p^{m-d_i}}$. Then the leading term of $F(Y_1(T), \dots, Y_n(T))$ is $P(c_1, \dots, c_n)T^{p^m} = aT^{p^m}$, so that $aT^{p^m} - F(Y_1(T), \dots, Y_n(T))$ is a $p$-polynomial of degree $< p^m$, and therefore by induction is equivalent modulo $F_*(\Hom(\Ga, \Ga^n))$ to a $p$-polynomial of degree $< p^N$. It follows that the same is true of $aT^{p^m}$, so the proof is complete.
\end{proof}

We will require the following technical lemma which may be thought of as a sort of division algorithm.

\begin{lemma}
\label{divunivppoly}
Let $F \in k[X_1, \dots, X_n]$ be a reduced $p$-polynomial whose principal part $$P(X_1, \dots, X_n) = \sum_{i=1}^n a_iX_i^{p^{d_i}}$$ is universal. Let $R := k[Y_1, \dots, Y_m]$, and let $A \subset R$ denote the subgroup $F(R, \dots, R)$. Then for every $g \in R$ with vanishing constant term, there is a unique $h \in R$ with vanishing constant term such that $g \equiv h \pmod{A}$ and such that, for every $1 \leq i \leq n$ and every $m$-tuple $\vec{s}$ of nonnegative integers, the coefficient of $(\vec{Y}^{\vec{s}})^{p^{d_i}}$ in $h$ lies in $P(k, \dots, 0, \dots, k)$, where the $0$ is in the $i$th coordinate. Further, if $S_g$ denotes the set of ordered $m$-tuples $\vec{s} \in \N^m$ such that the monomial $\vec{Y}^{\vec{s}}$ appears with nonzero coefficient in $g$, then $S_h \subset p^{-\infty} S_g := \cup_{t \geq 0} p^{-t}S_g$.
\end{lemma}

\begin{proof}
Because $P$ is reduced and universal, every $\alpha \in k$ may be written in the form $P(c_1, \dots, c_n)$ for unique $c_i \in k$. Let $\phi \colon k \rightarrow \N$ denote the function which sends $\alpha$ to $\#\{i \mid c_i \neq 0\}$. For $f \in k[Y_1, \dots, Y_m]$, let $d(f)$ denote the minimal positive integer $d$ such that $f$ contains a monomial of degree $d$ which violates the condition required of $h$ in the lemma. That is, for some $1 \leq i \leq n$, $f$ contains a monomial of degree $d$ of the form $(\vec{Y}^{\vec{s}})^{p^{d_i}}$ whose coefficient $b$ satisfies $b = P(c_1, \dots, c_n)$ with $c_i \neq 0$. If no such monomial exists, we define $d(f) := 0$. Next, define $\phi(f)$ to be the sum over all monomials $M$ in $f$ of degree $d(f)$ of $\phi(c_M)$, where $c_M$ is the coefficient of $M$ in $f$.

To prove the existence of $h$ as in the lemma, choose $h \equiv g \pmod{A}$ that minimizes $(d(h), \phi(h))$, ordered lexicographically, among those polynomials $\equiv g \pmod{A}$ such that only monomials with multidegree lying in $p^{-\infty}S_g$ appear in $h$. (In particular, $h$ has vanishing constant term.) Then we claim that $h$ has the property required by the lemma. If not, then for some nonzero $m$-tuple $\vec{s}$ of nonnegative integers, and some $1 \leq i \leq n$, the coefficient $b$ of $(\vec{Y}^{\vec{s}})^{p^{d_i}}$ in $h$ satisfies $b = P(r_1, \dots, r_n)$ with all $r_j \in k$ and $r_i \neq 0$. Choose such a monomial of degree $d(h)$. Consider $\tilde{h} := h - F(0, \dots, r_i\vec{Y}^{\vec{s}}, \dots, 0)$, where all terms are $0$ except for the term $r_i\vec{Y}^{\vec{s}}$ appearing in the $i$th place. Then $\tilde{h}$ only modifies $h$ by replacing the coefficient $b = P(r_1, \dots, r_n)$ with $\tilde{b} := P(r_1, \dots, 0, \dots, r_n)$ (where the $0$ is in the $i$th spot) and by adding terms of multidegree lying in $\cup_{t > 0} p^{-t}S_h \subset \cup_{t > 0} p^{-t}S_g$. Note in particular that $\phi(\tilde{b}) < \phi(b)$. We therefore see that either $d(\tilde{h}) < d(h)$ or that $d(\tilde{h}) = d(h)$ but $\phi(\tilde{h}) < \phi(h)$. In either case, this violates the minimality of $h$. We conclude that $h$ does indeed satisfy the property required by the lemma.

Finally, it remains to prove uniqueness of $h \equiv g \pmod{A}$ as in the lemma. For this, suppose given $h_1, h_2$ equivalent to $g$ and which both satisfy the property required by the lemma. By considering $h_1 - h_2$, we may assume that $g = 0$, so we have $h \in A$ satisfying the requirement of the lemma and we wish to show that $h = 0$. Suppose not. Because $h$ has vanishing constant term, it follows that we may write $h = F(G_1, \dots, G_n)$ with $G_i \in R$ not all constant. Suppose that $G_{i_0}$ is such that $p^{d_i}{\rm{deg}}(G_i)$ is maximized when $i = i_0$, and let $b\vec{Y}^{\vec{s}}$ be a monomial of degree ${\rm{deg}}(G_{i_0})$ appearing in $G_{i_0}$. Then the coefficient of $(\vec{Y}^{\vec{s}})^{p^{d_{i_0}}}$ appearing in $h = F(G_1, \dots, G_n)$ is $P(r_1, \dots, r_n)$ for some $r_i \in k$ with $r_{i_0} = b \neq 0$. This violates the required condition on $h$, so this contradiction shows that in fact $h = 0$, as desired.
\end{proof}

\begin{lemma}
\label{H^2=0affX}
For a commutative unipotent $k$-group scheme $U$ over a field $k$, and an affine $k$-scheme $X$, one has ${\rm{H}}^2(X, U) = 0$.
\end{lemma}

\begin{proof}
The proof in the case $X = {\rm{Spec}}(k)$ is given in \cite[Lem.\,3.3]{tossici}. The same argument goes through with ${\rm{Spec}}(k)$ replaced by any affine $k$-scheme $X$, using the fact that ${\rm{H}}^i(X, \Ga) = 0$ for $i > 0$. This last assertion follows from the fact that $\Ga$ is a quasi-coherent sheaf, hence its Zariski and fppf cohomology agree.
\end{proof}

Next we will relate universal $p$-polynomials to second Ext groups. This will require an explicit description of such groups, which we obtain in the following lemma.

\begin{lemma}
\label{Ext^2explicit}
Let $U$ be a commutative semiwound unipotent $k$-group scheme. Then one has an isomorphism, functorial in $U$ and in endomorphisms of $\Ga$, between $\Ext^2(\Ga, U)$ and the middle cohomology of the complex
\[
{\rm{H}}^1(\Ga, U) \xlongrightarrow{d_1} {\rm{H}}^1(\Ga^2, U) \xlongrightarrow{d_2} {\rm{H}}^1(\Ga^2, U) \oplus {\rm{H}}^1(\Ga^3, U),
\]
where 
\[
d_1 := m^* - (\pi^2_1)^* - (\pi^2_2)^*
\]
with $m, \pi^2_i\colon \Ga^2 \rightarrow \Ga$ the addition and projection maps, and the components $f_1\colon {\rm{H}}^1(\Ga^2, U) \rightarrow {\rm{H}}^1(\Ga^2, U)$, and $f_2\colon {\rm{H}}^1(\Ga^2, U) \rightarrow {\rm{H}}^1(\Ga^3, U)$ of $d_2$ defined as follows:
\[
f_1 := Id - {\rm{\sigma}}^*
\]
\[
f_2 := (\pi^3_1, m_{23})^* + \pi_{23}^* - (m_{12}, \pi^3_3)^* - \pi_{12}^*,
\]
where $\sigma\colon \Ga^2 \rightarrow \Ga^2$ is the switching map $(X, Y) \mapsto (Y, X)$, $\pi^3_{i}\colon \Ga^3 \rightarrow \Ga$ denotes projection onto the $i$th factor, $\pi_{ij} \colon \Ga^3 \rightarrow \Ga^2$ denotes projection onto the $(i, j)$ factors, and $m_{ij}\colon \Ga^3 \rightarrow \Ga$ denotes addition of the $i$ and $j$ factors.
\end{lemma}

\begin{proof}
The proof makes use of certain spectral sequences due to Breen. For the reader's convenience, we recall the properties of these sequences in appendix \ref{breenseqapp}. We use the notation of that appendix here. To begin with, we have the sequence
\[
{}^{\prime}E_2^{i,j} = \Ext^i({\rm{H}}_j(A(\Ga)), U) \Longrightarrow \Ext^{i+j}(A(\Ga), U).
\]
(See (\ref{E_2spectralseqbreen21}.) Then ${}^{\prime}E_2^{i,1}$ and ${}^{\prime}E_2^{0,2}$ both vanish by (\ref{H_1(A)=0}) with $G = \Ga$, plus the fact that $U$ is semiwound (for the latter vanishing). It follows that $\Ext^2(A(\Ga), U) \simeq {}^{\prime}E_2^{2,0} = \Ext^2(\Ga, U)$. Thus we may replace the latter group by the former in proving the lemma.

To compute $\Ext^2(A(\Ga), U)$, we use the other Breen spectral sequence, which takes the form
\[
E_1^{i,j} = \widetilde{{{\rm{H}}}}^j(X_i, U) \Longrightarrow \Ext^{i+j}(A(\Ga), U).
\]
(See (\ref{E_1specseqbreen20}).) Because $U$ is semiwound, $E_1^{i,0} = 0$, and $E_1^{0,2} = 0$ by Lemma \ref{H^2=0affX}. It follows that $\Ext^2(A(\Ga), U) \simeq E_2^{1,1}$. But by definition of the Breen sequence, $E_2^{1,1}$ is exactly the middle cohomology of the sequence in the lemma. The identification of the lemma is functorial because of the functoriality of the Breen spectral sequences.
\end{proof}

\begin{lemma}
\label{C^1(F)findim}
Let $F \in k[X_1, \dots, X_n]$ be a reduced $p$-polynomial whose principal part $P$ is universal, and let $U := \{F = 0\} \subset \Ga^n$. Then $\Ext^2(\Ga, U)$ is a finite-dimensional $k$-vector space.
\end{lemma}

\begin{proof}
If $k$ has infinite degree of imperfection, then $n = 1$ and $U = 0$ by Proposition \ref{monppolynuniv}, so the assertion is trivial. Assume, therefore, that $k$ has finite degree of imperfection. Because $\Ga$ is quasi-coherent, its fppf and Zariski cohomology agree. Thus ${\rm{H}}^i(T, \Ga) = 0$ for $i > 0$ and $T$ affine. We thus identify ${\rm{H}}^1(T, U)$ with the quotient of ${\rm{H}}^0(T, \Ga)$ by $F({\rm{H}}^0(T, \Ga), \dots, {\rm{H}}^0(T, \Ga))$ for affine $k$-schemes $T$. In particular, we will do this with $T$ equal to various powers of $\Ga$. We use Lemma \ref{Ext^2explicit} to identify $\Ext^2(\Ga, U)$ with the middle cohomology group $C$ of the sequence of that lemma.

So suppose given $\alpha \in \Ext^2(\Ga, U)$, represented by a class in ${\rm{H}}^1(\Ga^2, U)$, and let this class in turn be represented via the above identification by an element $g \in k[X, Y]$. Modifying $g$ by $f(X+Y) - f(X) - f(Y)$ for $f \in k[T]$ does not change its class in ${\rm{H}}^1$. In particular, we may take $f$ to be a suitable constant and thereby assume that $g$ has vanishing constant term. By Lemma \ref{divunivppoly}, we may assume that $g$ is such that, for every $1 \leq i \leq n$ and every ordered pair $(s_1, s_2)$ of nonnegative integers, the coefficient of $(X^{s_1}Y^{s_2})^{p^{d_i}}$ in $g$ lies in $P(k, \dots, 0, \dots, k)$, where the $0$ is in the $i$th coordinate (and the $d_i$ are as in Lemma \ref{divunivppoly}). Because $g$ represents a class in $C$, the polynomials
\[
g(X, Y) - g(Y, X) \in k[X, Y]
\]
\begin{equation}
\label{C^1(F)findimpfeqn1}
g(X, Y+Z) + g(Y, Z) - g(X+Y, Z) - g(X, Y) \in k[X, Y, Z]
\end{equation}
represent trivial cohomology classes in ${\rm{H}}^1(\Ga^2, U)$ and ${\rm{H}}^1(\Ga^3, U)$, respectively. On the other hand, for every $1 \leq i \leq n$ and every ordered pair $(s_1, s_2)$ of nonnegative integers, the coefficient of $(X^{s_1}Y^{s_2})^{p^{d_i}}$ in $g$ lies in $P(k, \dots, 0, \dots, k)$, where the $0$ is in the $i$th coordinate. It follows that the same holds for the two polynomials appearing above (with the analogous condition in $3$ variables in the case of the second polynomial in (\ref{C^1(F)findimpfeqn1})). Since they represent trivial cohomology classes, they lie in $F(k[X, Y]^n)$ and $F(k[X, Y, Z]^n)$, respectively. By the uniqueness assertion of Lemma \ref{divunivppoly}, it follows that the two polynomials in (\ref{C^1(F)findimpfeqn1}) vanish. By \cite[Lem.\,2.8.2]{rostateduality}, $g$ may be written in the form $$g(X, Y) = f(X+Y) - f(X) - f(Y) + \sum_{i=0}^N r_i\cdot S(X^{p^i}, Y^{p^i}),$$ for some $f \in k[X]$ and some $r_i \in k$, where $S(X, Y) \in \F_p[X, Y]$ is the mod $p$ reduction of the Witt polynomial $[(X+Y)^p - X^p - Y^p]/p \in \Z[X, Y]$. Thus, by modifying $g$ by $d_1$ of the element of ${\rm{H}}^1(\Ga, U)$ represented by $f$, we obtain that our element $\alpha \in C$ is represented by a polynomial of the form $\sum_{i=0}^N r_i\cdot S(X^{p^i}, Y^{p^i})$ with $r_N \neq 0$ if $N \geq 0$. We may apply Lemma \ref{divunivppoly} again to conclude that $\alpha$ is represented by a polynomial $h(X, Y) \in k[X, Y]$ such that each monomial appearing with nonzero coefficient in $h$ is of the form $(X^aY^{p-a})^{p^s}$ with $s < M := \max_i d_i$. In particular, every element of $C$ is represented by a sum of monomials of degree a power of $p$ that is $< pM$. The action of $\lambda \in k$ on the class in $C$ of a homogeneous polynomial $H$ of degree $p^t$ is represented by $\lambda^{p^t}H$. (Note that this $k$-action is identified with the $k$-action on $\Ext^2(\Ga, U)$ thanks to the functoriality aspect of Lemma \ref{Ext^2explicit}.) Because $[k: k^{p^t}] < \infty$ for all $t$ (since $k$ has finite degree of imperfection), it follows that $C$ is spanned by finitely many vectors.
\end{proof}

The definition of permawoundness for a group $U$ is rather strong in that it imposes a condition on extensions of $\Ga$ by any quotient group of $U$, and this condition on quotients will play an essential role in proving the fundamental fact that permawound groups admit filtrations as in Theorem \ref{rigidityintro}. In order to show that groups are permawound, therefore, we need a mechanism for dealing with quotients. The following lemma is that mechanism.

\begin{lemma}
\label{swquotfdH^1prim}
Let $U$ be a commutative unipotent $k$-group scheme, and suppose that $0 = U_0 \subset U_1 \subset \dots \subset U_m = U$ is such that, for each $0 \leq i < m$, $U_{i+1}/U_i \simeq \{F_i = 0\} \subset \Ga^{n_i}$ for a reduced $p$-polynomial $F_i \in k[X_1, \dots, X_{n_i}]$ with principal part $P_i$. If $\Ext^1(\Ga, U)$ is a finite-dimensional $k$-vector space, then each $P_i$ is universal and $\Ext^2(\Ga, U)$ is a finite-dimensional $k$-vector space.
\end{lemma}

\begin{proof}
We proceed by induction on the length of the filtration. If the filtration has length $0$, then the assertion is immediate. Now suppose that $m > 0$, and let $W := U/U_{m-1}$. Note that each group in the filtration is semiwound by \cite[Lem.\,B.1.7,$(1)\Longrightarrow(2)$]{cgp}. In particular, $\Hom(\Ga, W) = 0$, so the finite-dimensionality of $\Ext^1(\Ga, U)$ implies the same for $\Ext^1(\Ga, U_{m-1})$. By induction, it follows that $P_i$ is universal for $0 \leq i < m-1$ and that $\Ext^2(\Ga, U_{m-1})$ is finite-dimensional . We deduce that $\Ext^1(\Ga, W)$ is finite-dimensional, so $P_{m-1}$ is universal by Proposition \ref{univiffH^1primfd}. Hence $\Ext^2(\Ga, W)$ is finite-dimensional by Lemma \ref{C^1(F)findim}, so $\Ext^2(\Ga, U)$ is also finite-dimensional.
\end{proof}

\begin{remark}
It is possible to prove the above results without making use of Breen's spectral sequences, by defining in an ad hoc manner the cohomology group appearing in Lemma \ref{Ext^2explicit}, replacing the group $\Ext^1(\Ga, U)$ with the group ${\rm{H}}^1(\Ga, U)_{\prim}$ (which one may show to be isomorphic to $\Ext^1(\Ga, U)$ when $U$ is semiwound), and then showing that these groups satisfy the same formal properties as the first and second Ext groups. We will never use this.
\end{remark}

\section{(Quasi-)weak permawoundness}
\label{qwasunivsection}

In this section, as a tool for studying permawound groups, we introduce (and prove some basic properties of) the notion of (quasi-)weak permawoundness, which is more restrictive than permawoundness in that it only applies to commutative groups, but less so in that it does not require smoothness. This combination of restrictiveness and looseness makes (quasi-)weak permawoundness a technically convenient tool.

\begin{definition}
\label{weakasunivdef}
We say that a commutative unipotent $k$-group scheme $U$ is {\em $k$-(quasi-)weakly permawound}, or just (quasi-)weakly permawound when $k$ is clear from context, when the following conditions holds. For every exact sequence of {\em commutative} $k$-group schemes
\[
(0 \longrightarrow) U \longrightarrow E \longrightarrow \Ga \longrightarrow 1,
\]
$E$ contains a $k$-subgroup scheme $k$-isomorphic to $\Ga$.
\end{definition}

Note that, unlike in the definition of permawoundness, we do not require $U$ to be smooth in the definition of (quasi-)weak permawoundness. Clearly, weak permawoundness implies quasi-weak permawoundness, and for smooth commutative $U$, permawoundness implies weak permawoundness.

\begin{proposition}
\label{asunivperfect}
Let $k$ be a perfect field.
\begin{itemize}
\item[(i)] Every smooth unipotent $k$-group scheme is permawound.
\item[(ii)] Every commutative unipotent $k$-group scheme is (quasi-)weakly permawound.
\end{itemize}
\end{proposition}

\begin{proof}
All assertions follow from the claim that, for any extension of unipotent $k$-groups
\[
1 \longrightarrow U \longrightarrow E \longrightarrow \Ga \longrightarrow 1,
\]
$E$ contains a copy of $\Ga$. To see this, consider $E_{\red}^0 \subset E$. Because $k$ is perfect, this is a smooth unipotent $k$-subgroup scheme of $E$. It is also connected and nontrivial (the latter because $E$ admits a positive-dimensional quotient). Therefore, again because $k$ is perfect, it is split and nontrivial, hence $E$ admits a nontrivial split unipotent $k$-subgroup.
\end{proof}

Proposition \ref{asunivperfect} tells us that permawoundness is not an interesting notion over perfect fields.

\begin{proposition}
\label{splitunipasuniv}
Every split unipotent $k$-group scheme $U$ is permawound. If $U$ is commutative, then it is also (quasi-)weakly permawound.
\end{proposition}

\begin{proof}
Because quotient groups of split unipotent groups are also split, the proposition follows from the following claim. Given an exact sequence
\[
1 \longrightarrow U \longrightarrow E \longrightarrow \Ga \longrightarrow 1,
\]
with $U$ split, then $E$ contains a copy of $\Ga$. Indeed, given such a sequence, $E$ is necessarily split unipotent and nontrivial, hence contains a copy of $\Ga$.
\end{proof}

The following proposition is immediate from the definition of (weak) permawoundness.

\begin{proposition}
\label{quotofasunivisasuniv}
Any quotient group of a (weakly) permawound $k$-group scheme is also (weakly) permawound.
\end{proposition}

(Quasi-)weak permawoundness is inherited by subgroups with semiwound quotient.

\begin{proposition}
\label{qwasunivsubgps}
Let $U' \subset U$ be unipotent $k$-group schemes with $U'$ central in $U$, and assume that $U'' := U/U'$ is semiwound. Assume either that $U$ is permawound or that $U$ is commutative and (quasi-)weakly permawound. Then $U'$ is (quasi-)weakly permawound.
\end{proposition}

\begin{proof}
Suppose given an exact sequence
\[
(0 \longrightarrow) U' \longrightarrow E \longrightarrow \Ga \longrightarrow 0.
\]
Consider the following commutative diagram of exact sequences in which the top left square is a pushout diagram (and where we use centrality of $U' \subset U$ in the permawound case to ensure that this pushout is a $k$-group scheme):
\[
\begin{tikzcd}
(0 \arrow{r} & )U' \arrow{d} \arrow{r} & E \arrow{d} \arrow{r} & \Ga \arrow[d, equals] \arrow{r} & 0 \\
(0 \arrow{r} & )U \arrow{d} \arrow{r} & G \arrow{r} \arrow{d}{g} & \Ga \arrow{r} & 0 \\
& U'' \arrow[r, equals] & U'' &&
\end{tikzcd}
\]
Permawoundity, or (quasi-)weak permawoundness, of $U$ implies that there is an inclusion $f\colon \Ga \hookrightarrow G$. Because $U''$ is semiwound, $g\circ f = 1$, so $f$ factors through an inclusion $\Ga \hookrightarrow E$.
\end{proof}

The following simple proposition is useful.

\begin{proposition}
\label{qwasunivclundext}
Suppose given an exact sequence of commutative unipotent $k$-group schemes
\[
0 \longrightarrow U' \longrightarrow U \longrightarrow U'' \longrightarrow 0.
\]
If $U'$ and $U''$ are (quasi-)weakly permawound, then so is $U$.
\end{proposition}

\begin{proof}
Suppose given an exact sequence of commutative $k$-groups
\begin{equation}
\label{qwasunivclundextpfeqn1}
(0 \longrightarrow) U \xlongrightarrow{f} E \longrightarrow \Ga \longrightarrow 0.
\end{equation}
Then $E/f(U')$ sits in a similar exact sequence with $U''$ on the left instead of $U$, so by (quasi-)weak permawoundness of $U''$, $E/f(U')$ contains a $k$-subgroup $\mathbf{G} \simeq \Ga$. Letting $\pi\colon E \rightarrow E/f(U')$ denote the quotient map, $\pi^{-1}(\mathbf{G})$ then sits in an exact sequence analogous to (\ref{qwasunivclundextpfeqn1}) but with $U'$ on the left. (Quasi-)weak permawoundness of $U'$ then implies that $\pi^{-1}(\mathbf{G}) \subset E$ contains a copy of $\Ga$.
\end{proof}

\section{Permawound subgroups of vector groups}
\label{asunivsubgpsvecgpssection}

In this section we will give a criterion characterizing those $p$-polynomials which define permawound groups (Theorem \ref{asunivsmoothcrit}) and use it to prove that permawound groups are ``ubiquitous'' in a suitable sense (Theorem \ref{ubiquity}). We first need a criterion for quasi-weak permawoundness (which is easier to analyze than weak permawoundness because there is no need to deal with quotients).

\begin{proposition}
\label{qweakasunivcrit}
Let $F \in k[X_1, \dots, X_n]$ $($n > 0$)$ be a reduced $p$-polynomial with principal part $P$, and let $U := \{F = 0\} \subset \Ga^n$. Then $U$ is quasi-weakly permawound if and only if $P$ is universal.
\end{proposition}

\begin{proof}
First suppose that $P$ is not universal, and we will show that $U$ is not quasi-weakly permawound. We may multiply $F$ by a nonzero constant and thereby assume that $1 \notin P(k, \dots, k)$. Let $p^{d_i} := {\rm{deg}}_{X_i}(F)$ and $N := \max_i d_i$. Consider the $p$-polynomial $$G(T, X_1, \dots, X_n) := T^{p^N} - F(X_1, \dots, X_n) \in k[X_1, \dots, X_n, T].$$ Let $E := \{G = 0\} \subset \Ga^{n+1}$. The map $E \rightarrow \Ga$ which projects onto the $T$ coordinate is surjective with kernel $U$, so in order to prove that $U$ is not quasi-weakly permawound, it suffices to show that $E$ is semiwound.

The principal part of $G$ is $T^{p^N} - P(X_1, \dots, X_n)$. We claim that this has no nontrivial zeroes over $k$ -- that is, $G$ is reduced. Indeed, suppose that $t^{p^N} = P(x_1, \dots, x_n)$ with $t, x_i \in k$. If $t \neq 0$, then one may divide by $t^{p^N}$ to obtain $$1 = P(x_1/t^{p^{N-d_1}}, \dots, x_n/t^{p^{N-d_n}}),$$ in violation of the assumption that $P$ does not represent $1$. It follows that $t = 0$. Because $P$ is reduced by assumption, it then follows that $x_i = 0$ for all $i$. Thus $G$ is reduced, hence $E$ is semiwound by \cite[Lem.\,B.1.7,$(1)\Longrightarrow(2)$]{cgp}.

Next suppose that $P$ is universal, and we will show that $U$ is quasi-weakly permawound. Suppose given an exact sequence of commutative $k$-groups
\[
0 \longrightarrow U \longrightarrow E \longrightarrow \Ga \longrightarrow 0.
\]
We must show that $E$ contains a copy of $\Ga$. By Lemma \ref{extGavertriv}, $E$ has vanishing Verschiebung. Now consider the following commutative diagram of exact sequences in which the top left square is a pushout diagram:
\[
\begin{tikzcd}
0 \arrow{r} & U \arrow[d, hookrightarrow] \arrow{r} & E \arrow[d, hookrightarrow] \arrow{r} & \Ga \arrow{r} \arrow[d, equals] & 0 \\
0 \arrow{r} & \Ga^n \arrow{r} \arrow{d}{F} & W \arrow{d}{G} \arrow{r} & \Ga \arrow{r} & 0 \\
& \Ga \arrow[r, equals] & \Ga &&
\end{tikzcd}
\]
Then $W$ also has vanishing Verschiebung, so by Proposition \ref{splitvecsubgps}, the middle row splits. Choosing a splitting $W \simeq \Ga^n \times \Ga$, the map $G$ is given by $F$ on the $\Ga^n$ coordinate, and by some $p$-polynomial $H \in k[X_{n+1}]$ on the $\Ga$ coordinate. The group $E$ is therefore isomorphic to the vanishing locus of $G := F(X_1, \dots, X_n) + H(X_{n+1}) \in k[X_1, \dots, X_{n+1}]$. Let $P_H(X_{n+1}) = aX_{n+1}^{p^d}$ denote the principal part of $H$. (Possibly $H = 0$, in which case $a = 0$.) Note that the principal part $P_G := P(X_1, \dots, X_n) + P_H(X_{n+1})$ of $G$ is universal, because $P$ is, and that it has a nontrivial zero: Choosing $x_1, \dots, x_n \in k$ such that $P(x_1, \dots, x_n) = -a$, the tuple $(x_1, \dots, x_n, 1)$ yields a nontrivial zero of $P_G$. Among those non-reduced $p$-polynomials $R \in k[X_1, \dots, X_{n+1}]$ with universal principal part defining $E$, choose one $R_0$ such that the quantity $\sum_{i=1}^{n+1} {\rm{deg}}_{X_i}(R_0)$ is minimized, where we make the convention that ${\rm{deg}}_{X_i}(R) = 0$ if $X_i$ does not appear in $R$. Then we claim that at least one of the $X_i$ -- say $X_{i_0}$ -- does not appear in $R_0$. In that case, the $k$-subgroup of $E$ defined by the condition that $X_i = 0$ for all $i \neq i_0$ is a copy of $\Ga$ sitting inside $E$, so the proof will be complete.

To prove the claim, suppose for the sake of contradiction that all of the $X_i$ appear in $R_0$. Let $p^{d_i} := {\rm{deg}}_{X_i}(R_0)$. Choose a nontrivial zero $(x_1, \dots, x_{n+1}) \in k^{n+1}$ of the principal part $P_0$ of $R_0$. Let $i_0$ be such that $x_{i_0} \neq 0$, and such that $d_i \leq d_{i_0}$ for all $i$ such that $x_i \neq 0$. Consider the invertible change of variables on $\Ga^{n+1}$ defined by
\[
X_i \mapsto \begin{cases}
X_i, & x_i = 0 \\
x_{i_0}X_{i_0}, & i = i_0 \\
X_i + x_iX_{i_0}^{p^{d_{i_0}-d_i}}, & \mbox{else}.
\end{cases}
\]
Then the new $p$-polynomial $R'$ obtained from $R_0$ after making this change of variables satisfies ${\rm{deg}}_{X_i}(R') = {\rm{deg}}_{X_i}(R_0)$ for $i \neq i_0$ and ${\rm{deg}}_{X_{i_0}}(R') < {\rm{deg}}_{X_{i_0}}(R_0)$, so the sum of the degrees in the $X_i$ is strictly smaller for $R'$ than for $R_0$. On the other hand, we claim that the principal part $P'$ of $R'$ is universal and has a nontrivial zero over $k$, which will contradict the minimality of $R_0$ and prove the claim.

We first prove that $P'(X_1, \dots, 0, \dots, X_{n+1})$, where the $0$ is in the $i_0$ spot, is universal. Indeed, we know that $P'$ and $P_0$ agree when one sets $X_{i_0} = 0$, so it suffices to show that $P_0$, which is universal, remains so upon throwing out $X_{i_0}$. To see this, let $aX_{i_0}^{p^{d_{i_0}}}$ be the monomial in $X_{i_0}$ appearing in $P_0$. Then one has that $$a(-x_{i_0})^{p^{d_{i_0}}} = P_0(x_1, \dots, 0, \dots, x_{n+1}),$$ where the $0$ is in the $i_0$ spot. We therefore obtain for all $y \in k$ $$ay^{p^{d_{i_0}}} = P(-x_1(y/x_{i_0})^{p^{d_{i_0}- d_1}}, \dots, 0, \dots, -x_{n+1}(y/x_{i_0})^{p^{d_{i_0}- d_{n+1}}}),$$ where we place a $0$ in the $i_0$ spot and, by convention, the term $-x_i(y/x_{i_0})^{p^{d_{i_0}- d_i}}$ is just taken to be $0$ if $x_i = 0$. In particular, the terms all make sense because $i_0$ was chosen so that $d_{i_0} \geq d_i$ if $x_i \neq 0$. We therefore see that every element of $k$ represented by $aX_{i_0}^{p^{d_{i_0}}}$ is also represented by $P_0$ with $X_{i_0}$ thrown away. It follows that $P_0$ with $X_{i_0}$ thrown away is universal, hence so is $P'$ with $X_{i_0}$ thrown away. In particular, $P'$ is universal. Now let $bX_{i_0}^{p^d}$ be the monomial in $P'$ involving $X_{i_0}$. (Possibly $b = 0$.) Because $P'$ with $X_{i_0}$ thrown away is universal, one has $-b = P'(y_1, \dots, 0, \dots, y_{n+1})$ for some $y_i \in k$, where the $0$ is in the $i_0$ spot. Then $(y_1, \dots, 1, \dots, y_{n+1})$, where the $1$ is in the $i_0$ spot, is a nontrivial zero of $P'$, as required.
\end{proof}

A handy consequence of (weak) permawoundness is connectedness.

\begin{proposition}
\label{asunivconn}
Let $k$ be an imperfect field. Then every (weakly) permawound $k$-group scheme is connected.
\end{proposition}

\begin{proof}
If $U$ is (weakly) permawound, then so is its finite \'etale component group $U/U^0$ by Proposition \ref{quotofasunivisasuniv}. It therefore suffices to show that a nontrivial finite \'etale unipotent group $W$ over an imperfect field $k$ is not (weakly) permawound. We may filter $W$ by its derived series and thereby reduce to the case in which $W$ is commutative (again using Proposition \ref{quotofasunivisasuniv}), and we may reduce to the case in which $W$ is $p$-torsion by replacing $W$ by its quotient $[p]W$ if $W$ is not already killed by $[p]$. Since permawound implies weakly so in the commutative setting, it therefore suffices to prove that a nontrivial finite \'etale commutative $p$-torsion group is not weakly permawound, for which it is enough to show that it is not quasi-weakly permawound.

The group $W$ is the zero locus of some $p$-polynomial $F \in k[X]$ \cite[Prop.\,B.1.13]{cgp}. We may multiply $F$ by some element of $k^{\times}$ and thereby assume that it is monic. Note that $p^d := {\rm{deg}}(F) > 1$ because $W$ is nontrivial. Because $k$ is imperfect, $X^{p^d}$ is not universal, so $W$ is not quasi-weakly permawound by Proposition \ref{qweakasunivcrit}.
\end{proof}

We can also now show that there are no interesting (quasi-)weakly permawound groups over fields of infinite degree of imperfection, and therefore also no interesting permawound ones.

\begin{proposition}
\label{noasunivinfdegimp}
Let $k$ be a field of infinite degree of imperfection. Then there is no nontrivial semiwound unipotent $k$-group $U$ that is either (quasi-)weakly permawound or permawound.
\end{proposition}

\begin{proof}
Suppose that $U$ is permawound. Then $U$ is smooth and connected (Proposition \ref{asunivconn}), hence wound. (This appeal to connectedness is not really necessary; it is simply to put us in the context of \cite{cgp} to avoid having to reprove beyond the connected case some of the results state there for wound unipotent groups.) Therefore, if $U$ is nontrivial, then it admits a nontrivial commutative wound quotient \cite[Cor.\,B.3.3]{cgp}, which is also permawound by Proposition \ref{quotofasunivisasuniv} and therefore weakly permawound. It therefore suffices to treat the (quasi-)weakly permawound assertion. 

So let $U$ be a nontrivial commutative semiwound unipotent $k$-group scheme. By Proposition \ref{swdfilthyper}, $U$ admits a filtration $0 = U_0 \subset \dots \subset U_m = U$ such that each $U_i/U_{i-1}$ is a nontrivial group which is the vanishing locus of a reduced $p$-polynomial $F_i$, and in particular semiwound \cite[Lem.\,B.1.7,$(1)\Longrightarrow(2)$]{cgp}. If $U$ is (quasi-)weakly permawound, then $\{F_1 = 0\}$ is (quasi-weakly) permawound by Proposition \ref{qwasunivsubgps}. We are therefore reduced to showing that the vanishing locus of a reduced $p$-polynomial $F$, if nontrivial, is not quasi-weakly permawound. If it is, then the principal part $P$ of $F$ is universal by Proposition \ref{qweakasunivcrit}. Proposition \ref{monppolynuniv} then implies that $F \in k[X]$ is of the form $aX$ for some $a \in k^{\times}$, so that $\{F = 0\} \subset \Ga$ is the trivial group, in violation of our assumption. This completes the proof.
\end{proof}

We may now prove a criterion for weak permawoundness.

\begin{proposition}
\label{asunivcrit}
For each $i = 1, \dots, m$, let $F_i \in k[X_1, \dots, X_{n_i}]$ be a reduced $p$-polynomial with principal part $P_i$, and let $W_i := \{F_i = 0\} \subset \Ga^{n_i}$. Suppose that $U$ is a commutative $k$-group scheme admitting a filtration $0 = U_0 \subset \dots \subset U_m = U$ such that $U_i/U_{i-1} \simeq W_i$ for $1 \leq i \leq m$. Then $U$ is weakly permawound if and only $P_i$ is universal for every $1 \leq i \leq m$.
\end{proposition}

\begin{proof}
If $U$ is weakly permawound, then so are $U_1$ and $U/U_1$ by Propositions \ref{qwasunivsubgps} and \ref{quotofasunivisasuniv}. By induction, we conclude that $P_i$ is universal for $1 < i \leq m$. Because $U_1$ is weakly permawound, it is quasi-weakly so, hence $P_1$ is also universal by Proposition \ref{qweakasunivcrit}.

Conversely, suppose that the $P_i$ are all universal. Then each $\Ext^1(\Ga, W_i)$ is finite-dimensional by Proposition \ref{univiffH^1primfd}, hence $\Ext^1(\Ga, U)$ is as well. Now suppose given a quotient group $\overline{U}$ of $U$. We must show that $\overline{U}$ is quasi-weakly permawound. This is trivial if $\overline{U}$ contains a copy of $\Ga$, so assume that $\overline{U}$ is semiwound. Let $V := \ker(U \rightarrow \overline{U})$. By Proposition \ref{filterhypersurf}, both $V$ and $\overline{U}$ admit filtrations by hypersurface groups defined by reduced $p$-polynomials, and splicing these together yields such a filtration on $U$. Because $\Ext^1(\Ga, U)$ is finite-dimensional, Lemma \ref{swquotfdH^1prim} ensures that the $p$-polynomials in this filtration, and in particular those appearing in the filtration on $\overline{U}$, have universal principal part. Propositions \ref{qweakasunivcrit} and \ref{qwasunivclundext} then imply that $\overline{U}$ is quasi-weakly permawound.
\end{proof}

\begin{remark}
The assumption in Proposition \ref{asunivcrit} that the $F_i$ be reduced is necessary. Indeed, one may begin with a reduced $p$-polynomial with universal principal part, and by a change of variables transform it into a (necessarily non-reduced) $p$-polynomial with non-universal principal part. For instance, suppose that $k$ is imperfect with finite degree of imperfection and let $\lambda_1, \dots, \lambda_r$ be a $p$-basis for $k$. Then $$F := X_0 + \sum_{f \in I} \left(\prod_{i=1}^r \lambda_i^{f(i)}\right)X_f^p \in k[X_f\mid f \in I],$$ where $0$ denotes the $0$ function, is reduced with universal principal part; see example \ref{univbasicexample}. But if one makes the change of variables which fixes $X_f$ for all $f \neq 0$ but sends $X_0$ to $X_0 + \sum_{f \neq 0} X_f^p$, then $F$ is transformed into a $p$-polynomial with principal part $(X_0 + \sum_{f \neq 0} X_f^p)^p$. This is not universal because it only takes $p$th power values.
\end{remark}

\begin{corollary}
\label{univppolyconn}
Let $F \in k[X_1, \dots, X_n]$ be a reduced $p$-polynomial with universal principal part, and let $U := \{F = 0\} \subset \Ga^n$. Then $U$ is connected.
\end{corollary}

\begin{proof}
This follows from Propositions \ref{asunivcrit} and \ref{asunivconn}.
\end{proof}

\begin{proposition}
\label{weakasunivinhbyexts}
Let $K/k$ be a separable extension of fields of characteristic $p$ and the same finite degree of imperfection $($for instance, an algebraic separable extension$)$, and let $U$ be a commutative unipotent $k$-group scheme. Then $U$ is weakly permawound over $k$ if and only if it is so over $K$.
\end{proposition}

\begin{proof}
If $k$ is perfect, then so is $K$ so we are done by Proposition \ref{asunivperfect}. We may therefore assume that $k$ is imperfect and in particular infinite. First assume that $U$ is semiwound. Then by Proposition \ref{swdfilthyper}, there is a filtration $0 = U_0 \subset \dots \subset U_m = U$ such that each $W_i := U_i/U_{i-1}$ is of the form $\{F_i = 0\} \subset \Ga^{n_i}$ for some reduced $p$-polynomial $F_i \in k[X_1, \dots, X_{n_i}]$. The assertion therefore follows from Propositions \ref{asunivcrit} and \ref{universaloverextensions}.

Now consider the general case. Let $U_s \subset U$ be the maximal split $k$-subgroup, and let $U_w := U/U_s$. Then $U_w$ is $k$-semiwound. Further, $U$ is $k$- or $K$-weakly permawound if and only if $U_w$ is, by Propositions \ref{quotofasunivisasuniv}, \ref{splitunipasuniv}, and \ref{qwasunivclundext}. The proposition therefore reduces to the semiwound case.
\end{proof}

\begin{lemma}
\label{extGasemiwdsm}
Every extension of $\Ga$ by a semiwound smooth commutative unipotent $k$-group $U$ is commutative.
\end{lemma}

\begin{proof}
Let $E$ be such an extension. The group $U$ remains semiwound over $k_s$ (Proposition \ref{semiwoundsepble}), so we are free to assume that $k = k_s$. Let $u \in U(k)$. The quotient $\Ga$ of $E$ acts on $U$ by conjugation, so we obtain a $k$-scheme map $\Ga \rightarrow U$, $x \mapsto xux^{-1}$. Because $U$ is semiwound, this map is constant, hence $u$ is central in $E$. Since $u \in U(k)$ was arbitrary and $U(k)$ is dense in $U$ (because $U$ is smooth and $k = k_s$), it follows that $U$ is central in $E$. Therefore, the commutator map $E \times E \rightarrow U \subset E$ descends to a map $\Ga \times \Ga \rightarrow U$. Once again, this map must be constant because $U$ is semiwound. Thus, $E$ is commutative.
\end{proof}

\begin{proposition}
\label{asuniv=weakasuniv}
A smooth commutative unipotent $k$-group scheme $U$ is permawound if and only if it is weakly permawound.
\end{proposition}

\begin{proof}
The only if direction is clear. Conversely, suppose that $U$ is weakly permawound. Let $\overline{U}$ be a quotient group of $U$. We must show that any extension $E$ of $\Ga$ by $\overline{U}$ contains a copy of $\Ga$. If $\overline{U}$ contains such a copy, then this is immediate, so assume that $\overline{U}$ is semiwound. Then Lemma \ref{extGasemiwdsm} implies that $E$ is commutative, hence contains a copy of $\Ga$ by the weak permawoundness of $U$.
\end{proof}

We now finally obtain our criterion for permawoundness of smooth commutative $p$-torsion groups.

\begin{theorem}
\label{asunivsmoothcrit}
Let $F \in k[X_1, \dots, X_n]$ be a separable, reduced $p$-polynomial with principal part $P$, and let $U := \{F = 0\} \subset \Ga^n$. Then $U$ is permawound if and only if $P$ is universal.
\end{theorem}

\begin{proof}
Combine Propositions \ref{asuniv=weakasuniv} and \ref{asunivcrit}.
\end{proof}

Theorem \ref{asunivsmoothcrit} provides us with a wealth of examples of permawound groups, and in fact provides all examples of such groups that are commutative and $p$-torsion, thanks to \cite[Prop.\,B.1.13, Lem.\,B.1.7]{cgp}. We will show that it provides us with enough examples to obtain the all-important ubiquity property of permawoundness (Theorem \ref{ubiquityintro}). Before doing so, we require a lemma.

\begin{lemma}
\label{redexttouniv}
Let $k$ be a field of finite degree of imperfection, and let $P \in k[X_1, \dots, X_n]$ be a monogeneous reduced $p$-polynomial. Then there exist $m \geq 0$ and a monogeneous $p$-polynomial $Q \in k[X_{n+1}, \dots, X_{n+m}]$ such that $P + Q \in k[X_1, \dots, X_{n+m}]$ is reduced and universal.
\end{lemma}

\begin{proof}
We first prove the following claim. Suppose given a reduced monogeneous $p$-polynomial $P \in k[X_1, \dots, X_n]$, and let $p^N := \max_i {\rm{deg}}_{X_i}(P)$. Assume that $a \notin P(k, \dots, k)$. Then $F(X_1, \dots, X_n, Y) := P(X_1, \dots, X_n) - aY^{p^N}$ is still reduced. Indeed, suppose that $(x_1, \dots, x_n, y) \in k^{n+1}$ is a zero of $F$. Let $p^d_i := {\rm{deg}}_{X_i}(P)$. If $y \neq 0$, then one obtains $$a = P(x_1/y^{p^{N-d_1}}, \dots, x_n/y^{p^{N-d_n}}),$$ in violation of the assumption that $a \notin P(k, \dots, k)$. Therefore $y = 0$, hence $x_i = 0$ for all $i$ by reducedness of $P$, proving the claim.

Now we prove the lemma. Let $r$ be the degree of imperfection of $k$. Let $p^N := \max_i {\rm{deg}}_{X_i}(P)$ as above, and choose a reduced monogeneous $p$-polynomial of the form $F := P(X_1, \dots, X_n) + Q(X_{n+1}, \dots, X_{n+m})$ such that $Q$ is homogeneous of degree $p^N$ and such that the quantity $\phi(F) := \sum_{i=1}^{n+m} {\rm{deg}}_{X_i}(F)^{-r}$ is maximized. Proposition \ref{monppolynuniv} implies that there exists such an $F$. We claim that $F$ is universal. Indeed, if not, then there is some $a \in k$ not represented by $F$. Then $G := F - aY^{p^N}$ is still reduced and of the above form, but $\phi(G) > \phi(F)$, in violation of the maximality of $F$.
\end{proof}

We may now prove the ubiquity of permawoundness.

\begin{theorem}$($Theorem $\ref{ubiquityintro}$$)$
\label{ubiquity}
Let $k$ be a field of finite degree of imperfection. Then for any smooth commutative $p$-torsion semiwound unipotent $k$-group $U$, there is an exact sequence
\[
0 \longrightarrow U \longrightarrow W \longrightarrow V \longrightarrow 0
\]
with $W$ wound, commutative, $p$-torsion, and permawound, and $V$ a vector group.
\end{theorem}

\begin{proof}
If $U = 0$ then the assertion is trivial, so assume that $U \neq 0$. By \cite[Prop.\,B.1.13, Lem.\,B.1.7]{cgp}, $U \simeq \{F = 0\} \subset \Ga^n$ for some nonzero reduced $p$-polynomial $F \in k[X_1, \dots, X_n]$. Let $P$ be the principal part of $F$. By Lemma \ref{redexttouniv}, there exist $m \geq 0$ and a monogeneous $p$-polynomial $Q \in k[X_{n+1}, \dots, X_{n+m}]$ such that $P + Q \in k[X_1, \dots, X_{n+m}]$ is reduced and universal. Let $W := \{F + Q = 0\} \subset \Ga^{n+m}$. Then $W$ is permawound by Theorem \ref{asunivsmoothcrit}, therefore connected by Proposition \ref{asunivconn}, and it is wound by \cite[Lem.\,B.1.7,$(1)\Longrightarrow(2)$]{cgp}. Furthermore, the map $W \rightarrow \Ga^m$, $(X_1, \dots, X_{n+m}) \mapsto (X_{n+1}, \dots, X_{n+m})$ is surjective with kernel $U$.
\end{proof}

\section{The groups $\R_{k^{1/p}/k}(\alpha_p)$ and $\mathscr{V}$}
\label{specgpsection}

In this section we study certain unipotent groups which will play a fundamental role in the study of permawound groups. In particular, we study the groups $\R_{k^{1/p}/k}(\alpha_p)$ and $\mathscr{V}$ (see Definition \ref{V_PLdef}) which appear in the rigidity theorem \ref{rigidityintro}.

\begin{lemma}
\label{univpluslinchangeofvars}
Let $k$ be a separably closed field of finite degree of imperfection $r$. For $P \in k[X_1, \dots, X_{p^{nr}}]$ a reduced $p$-polynomial that is homogeneous of degree $p^n$, and $0 \neq L \in k[X_1, \dots, X_{p^{nr}}]$ that is homogeneous of degree $1$, consider the $p$-polynomial $F_{P, L} := P + L$. Then for any two such pairs $(P_1, L_1)$, $(P_2, L_2)$, there exists $c \in k^{\times}$ such that $cF_{P_2, L_2}$ is obtainable from $F_{P_1, L_1}$ by a linear change of coordinates in the $X_i$.
\end{lemma}

Note that a homogeneous $p$-polynomial $P \in k[X_1, \dots, X_{p^{nr}}]$ of degree $p^n$ over a field of degree of imperfection $r$ is reduced precisely when it is universal (Proposition \ref{monppolynuniv}).

\begin{proof}
Let $\lambda_1, \dots, \lambda_r \in k$ form a $p$-basis for $k$. Because $P$ is universal, its coefficients form a $k^{p^n}$-basis of $k$. Any two such bases are related by an invertible $k^{p^n}$-linear change of variables. It follows that any such $P$ may be transformed by a linear change of variables into $Q := \sum_{f \in I_n} (\prod_{i=1}^r \lambda_i^{f(i)})X_f^{p^n} \in k[X_f\mid f \in I_n]$. It therefore suffices to show that the $F_{Q, L}$ as $L$ varies over the nonzero linear forms in $k[X_f\mid f \in I_n]$ are all related by linear changes of variables as in the lemma, and one may simply show that they are all so related to $F := F_{Q, L_0}$, where $L_0 := X_0$, with $0 \in I_n$ denoting the $0$ function.

So let $L = \sum_{g \in I} \beta_g X_g$ where the $\beta_g \in k$ are not all $0$. For each $g \in I_n$, let $g' \in I_n$ denote the function
\[
g'(i) := \begin{cases}
0, & g(i) = 0 \\
p^n-g(i), & g(i) \neq 0,
\end{cases}
\]
and let $\epsilon_g \colon \{1, \dots, r\}\rightarrow \{0, 1\}$ denote the function
\[
\epsilon_g(i) := \begin{cases}
0, & g(i) = 0 \\
1, & g(i) \neq 0.
\end{cases}
\]
Consider the following system of equations in the $Y_f$ as $g$ varies over $I_n$:
\begin{equation}
\label{univpluslinsamepfeqn1}
\left( \prod_{i=1}^r \lambda_i^{\epsilon_g(i)} \right) Y_{g} = \beta_{g'} \sum_{f \in I_n} \left( \prod_{i=1}^r \lambda_i^{f(i)} \right)Y_f^{p^n}.
\end{equation}
We claim that this system has a nonzero solution $(Y_f) \in k^{I_n}$. Indeed, fix $g_0$ such that $\beta_{g_0'} \neq 0$, and consider the following equation in $Y_{g_0}$.
\[
\left( \prod_{i=1}^r \lambda_i^{\epsilon_{g_0}(i)} \right) Y_{g_0} = \beta_{g'_0}\sum_{f \in I_n} \left( \prod_{i=1}^r\lambda_i^{f(i)} \right) \left( \left(\prod_{i=1}^r \lambda_i^{\epsilon_{g_0}(i) - \epsilon_f(i)}\right) \beta_{f'}\beta_{g_0'}^{-1} \right)^{p^n} Y_{g_0}^{p^n}.
\]
Because the $\lambda_i$ are $p$-independent, the coefficient of $Y_{g_0}^{p^n}$ on the right side is nonzero (because the summand with $f = g_0$ is nonzero). Because $k = k_s$, it follows that the above equation has a nonzero solution $Y_{g_0} \in k$. Then, as one readily checks, setting
\[
Y_g := \left( \prod_{i=1}^r \lambda_i^{\epsilon_{g_0}(i) - \epsilon_g(i)} \right)\beta_{g'}\beta_{g'_0}^{-1}Y_{g_0}
\]
for all $g \in I_n$ yields a nonzero solution to the system (\ref{univpluslinsamepfeqn1}), proving the claim.

Now fix a nonzero solution $(a_{g, 0})_{g \in I_n}$ to (\ref{univpluslinsamepfeqn1}), and let 
\begin{equation}
\label{univpluslinsamepfeqn3}
c := \sum_{f \in I_n} \left( \prod_{i=1}^r\lambda_i^{f(i)} \right) a_{f, 0}^{p^n}.
\end{equation}
Note that $c \neq 0$ because the $\lambda_i$ are $p$-independent. Because the $\lambda_i$ form a $p$-basis, for each $g \in I_n$, one has unique $a_{f, g} \in k$ such that
\begin{equation}
\label{univpluslinsamepfeqn2}
\sum_{f \in I} \left( \prod_{i=1}^r\lambda_i^{f(i)} \right)a_{f, g}^{p^n} = c\prod_{i=1}^r \lambda_i^{g(i)}.
\end{equation}
Note that the notation $a_{f,g}$ is consistent with the previously defined quantities $a_{f,0}$ in the case $g = 0$. Consider the following linear change of variables on $\Ga^{I_n}$:
\begin{equation}
\label{univpluslinsamepfeqn4}
X_f \mapsto \sum_{g \in I} a_{f, g}X_g.
\end{equation}
We claim that this is invertible. Indeed, using (\ref{univpluslinsamepfeqn2}), one sees that it transforms $Q$ into $cQ$. That is, the $p^n$th power of the linear change of variables $(a_{f, g})$ transforms one $k^{p^n}$-basis for $k$ into another, hence this change of variables is invertible. We have also seen that it transforms $Q$ into $cQ$. We claim that it transforms $L_0$ into $cL$. Assuming this, we have obtained a linear change of variables sending $F_{Q, L_0}$ into $cF_{Q, L}$.

To obtain the claim, we first compute $a_{0, g}$. Using the expression (\ref{univpluslinsamepfeqn3}) for $c$, and equating the coefficients of $\prod_i \lambda_i^0$ on both sides of (\ref{univpluslinsamepfeqn2}), one obtains that
\begin{equation}
\label{univpluslinsamepfeqn5}
a_{0, g} = \left(\prod_{i=1}^r \lambda_i^{\epsilon_g(i)}\right)a_{g', 0},
\end{equation}
where $\epsilon_g$ and $g'$ are as defined previously. The new coefficient of $X_g$ after applying the change of variables (\ref{univpluslinsamepfeqn4}) to $V_{Q, L_0}$ is $a_{0, g}$, and (\ref{univpluslinsamepfeqn5}) and (\ref{univpluslinsamepfeqn3}), together with the facts that $(a_{f, 0})_{f \in I_n}$ is a solution to (\ref{univpluslinsamepfeqn1}), $g \mapsto g'$ is an involution of $I_n$, and $\epsilon_g = \epsilon_{g'}$, imply that $a_{0, g} = c\beta_g$.
\end{proof}

\begin{corollary}
\label{univpluslinsame}
Let $k$ be a separably closed field of finite degree of imperfection $r$. Then the $k$-groups $\{P + L = 0\} \subset \Ga^{p^r}$ are all $k$-isomorphic as $P$ varies over all reduced homogeneous $p$-polynomials of degree $p^n$ in $k[X_1, \dots, X_{p^{nr}}]$, and $L$ varies over the nonzero homogeneous degree $1$ polynomials in the same ring.
\end{corollary}

\begin{definition}
\label{V_PLdef}
Let $k$ be a field of finite degree of imperfection $r$. For any homogeneous $p$-polynomials $P, L \in k[X_1, \dots, X_{p^r}]$, $P$ of degree $p$ and $L$ of degree $1$, such that $P$ is universal (equivalently, reduced) and $L \neq 0$, let $\mathscr{V}_{P, L} := \{P + L = 0\} \subset \Ga^{p^r}$. When $k$ is separably closed, the $\mathscr{V}_{P,L}$ are all $k$-isomorphic by Corollary \ref{univpluslinsame}, and we denote the common $k$-group they define by $\mathscr{V}/k$, or just $\mathscr{V}$ when $k$ is clear from context.
\end{definition}

Note that $\mathscr{V}_{P,L}$ is smooth and connected. Indeed, one readily checks that it becomes isomorphic to $\Ga^{p^r-1}$ over $k^{1/p}$. Note also that $(\mathscr{V}_{P,L}/k)_K \simeq \mathscr{V}/K$ for any separable extension $K/k$ that is separably closed and has the same degree of imperfection as $k$. Indeed, this follows from Proposition \ref{universaloverextensions}(ii).

The other group which appears in the filtration of Theorem \ref{rigidityintro} is $\R_{k^{1/p}/k}(\alpha_p)$. Note that this is a totally nonsmooth group: $\R_{k^{1/p}/k}(\alpha_p)(k_s) = 0$. It is useful to write down an equation for it. Choose a $p$-basis $\lambda_1, \dots, \lambda_r$ for $k$. Then the elements $\prod_{i=1}^r \lambda_i^{f(i)/p}$ as $f$ varies over $I$ form a $k$-basis of $k^{1/p}$. It follows that $\R_{k^{1/p}/k}(\Ga)$ may be described as $\Ga^I$ with coordinates $X_I$,with elements written in the form $\sum_{f \in I} (\prod_{i=1}^r \lambda_i^{f(i)/p})X_I$. Then $\R_{k^{1/p}/k}(\alpha_p)$ is the subgroup consisting of those elements whose $p$th powers vanish, so it is described the equation
\begin{equation}
\label{restalphapeqn}
\sum_{f \in I} \left(\prod_{i=1}^r \lambda_i^{f(i)}\right)X_I^p = 0.
\end{equation}
By applying a linear change of variables, this is the same as the group defined by the equation $P = 0$ for any homogeneous degree $p$ universal $p$-polynomial $P \in k[X_1, \dots, X_{p^r}]$. We record this fact for later use.

\begin{proposition}
\label{eqnforweilrestalphap}
For a field $k$ of characteristic $p$ and finite degree of imperfection $r$, one has $\R_{k^{1/p}/k}(\alpha_p) \simeq \{P = 0\}$ for any $P \in k[X_1, \dots, X_{p^{r}}]$ which is universal and homogeneous of degree $p$.
\end{proposition}

More generally, we shall require the following description of the groups $\R_{k^{1/p^n}/k}(\alpha_p)$.

\begin{proposition}
\label{eqngenweilrestalphap}
For a field $k$ of characteristic $p$ and finite degree of imperfection $r$, one has for any $n > 0$ that $\R_{k^{1/p^n}/k}(\alpha_p) \simeq \left(\R_{k^{1/p}/k}(\alpha_p)\right)^{p^{r(n-1)}}$.
\end{proposition}

\begin{proof}
Choose a $p$-basis $\lambda_1, \dots, \lambda_r$ of $k$. Then $\R_{k^{1/p^n}/k}(\Ga)$ may be described as $\A_k^{I_n}$ with each coordinate $X_f$ corresponding to the coefficient of $\prod_{i=1}^r\lambda_i^{f(i)/p^n}$. The equation for $\R_{k^{1/p^n}/k}(\alpha_p) \subset \R_{k^{1/p^n}/k}(\Ga)$ is then
\[
\sum_{f \in I_n} \left(\prod_{i=1}^r \lambda_i^{f(i)/p^{n-1}} \right)X_f^p = 0,
\]
or
\[
\sum_{g \in I_{n-1}} \left(\prod_{i=1}^r \lambda_i^{g(i)/p^{n-1}} \right) \sum_{\substack{f \in I_n \\ f \equiv g \pmod{p^{n-1}}}} \left(\prod_{i=1}^r \lambda_i^{(f(i)-g(i))/p^{n-1}} \right) X_f^p = 0.
\]
Each summand in the internal sum above is defined over $k$, so that the above equation is equivalent to the group given by the product over $g \in I_{n-1}$ of the group $U_g \subset \Ga^{f \equiv g \pmod{p^{n-1}}}$ described by the vanishing of the internal sum. But the map $f \mapsto (f(i) - g(i))/p^{n-1}$ is a bijection between those $f \in I_n$ that are $\equiv g\pmod{p^{n-1}}$ and $I_1$. Thus we find that $\R_{k^{1/p^n}/k}(\alpha_p) \simeq U^{I_{n-1}}$, where $U \subset \Ga^{I_1}$ is the group
\[
\sum_{h \in I_1} \left(\prod_{i=1}^r \lambda_i^{h(i)} \right)Y_h^p = 0.
\]
By Proposition \ref{eqnforweilrestalphap}, one has $U \simeq \R_{k^{1/p}/k}(\alpha_p)$. Since $\#I_{n-1} = p^{r(n-1)}$, the proposition follows.
\end{proof}

The equations for $\mathscr{V}_{P, L}$ and for $\R_{k^{1/p}/k}(\alpha_p)$, in conjunction with Theorem \ref{asunivsmoothcrit} and Proposition \ref{asunivcrit}, imply the following result.

\begin{proposition}
\label{VR(alphap)asuniv}
Let $k$ be a field of characteristic $p$ of finite degree of imperfection. Then $\mathscr{V}_{P, L}$ is permawound, and $\R_{k^{1/p}/k}(\alpha_p)$ is weakly permawound. 
\end{proposition}

Here is a useful property of $\mathscr{V}_{P, L}$ and $\R_{k^{1/p}/k}(\alpha_p)$.

\begin{proposition}
\label{quotvralphsplit}
Let $k$ be a field of characteristic $p$ of finite degree of imperfection $r$, and let $U$ be either $\mathscr{V}_{P, L}$ or $\R_{k^{1/p}/k}(\alpha_p)$. Then for any nonzero closed $k$-subgroup scheme $0 \neq H \subset U$, the $k$-group $U/H$ is a vector group.
\end{proposition}

\begin{proof}
The argument given here is inspired by one given in the proof of \cite[Ex.\,9.10]{totaro}. Because $U/H$ is commutative and $p$-torsion, it is a vector group if and only if it is split. \cite[Th.\,B.2.5]{cgp}. Since this is insensitive to passing to $k_s$, we are free to assume that $k$ is separably closed. Similarly, any quotient of a power of $\Ga$ is still a power of $\Ga$ (because it is still commutative, $p$-torsion, and split unipotent). To prove the lemma for $H$, therefore, it suffices to prove it for a nonzero $k$-subgroup of $H$. If $\ker(F_{H/k}) = 0$, then $H$ is finite \'etale and since $k = k_s$, $H$ contains a $k$-subgroup isomorphic to $\Z/p\Z$. Otherwise, $\ker(F_{H/k})$ is nonzero and is killed by Frobenius and Verschiebung, hence is a power of $\alpha_p$, hence contains a copy of $\alpha_p$. We may therefore assume that $H \simeq \alpha_p$ or $\Z/p\Z$.

By definition, $\mathscr{V} \subset \Ga^{p^r}$, and we may consider $\R_{k^{1/p}/k}(\alpha_p)$ as lying in $\Ga^{p^r}$ via the equation (\ref{restalphapeqn}). We claim that for any line $L \subset \Ga^{p^r}$ through the origin, the intersection $U \cap L$ is a finite $k$-group scheme of order $\leq p$. (In fact, it is of order $p$.) We first note that $U \cap L$ is finite. Indeed, if not, then $L \subset U$, an impossibility as $U$ is semiwound by \cite[Lem.\,B.1.7,$(1)\Longrightarrow(2)$]{cgp}. Now $U \subset \Ga^{p^r}$ is defined by an equation of the form
\[
F = 0
\]
for some $F \in k[X_1, \dots, X_{p^r}]$ of degree $p$. An invertible linear change of variables in the $X_i$ yields an equation of the same form. After making such a change of variables, we may assume that $L$ is the line $X_2 = \dots = X_{p^r} = 0$. Then $U \cap L$ is the subgroup of $\Ga = {\rm{Spec}}(k[X_1])$ defined by a (nonzero, because $U \cap L \neq L$) polynomial of degree $\leq p$. That is, $U \cap L$ is a finite group scheme of order $\leq p$, as claimed.

Now choose a line $L$ such that $H \subset L$. For $H = \alpha_p$, we simply choose the line whose Lie algebra coincides with $H$, while for $H = \Z/p\Z$, we choose an element $0 \neq x \in H(k)$ and then choose $L$ containing $x$. Then $U \cap L \supset H$ is of order $\leq p$, hence coincides with $H$. Thus we have an inclusion of $k$-groups
\[
U/H = U/(U \cap L) \hookrightarrow \Ga^{p^r}/L \simeq \Ga^{p^r-1}.
\]
Since ${\rm{dim}}(U/H) = p^r-1 = {\rm{dim}}(\Ga^{p^r-1})$, this inclusion must actually be an isomorphism: $U/H \simeq \Ga^{p^r-1}$. This completes the proof of the proposition.
\end{proof}

\begin{corollary}
\label{VRsubgpnotasuniv}
Let $k$ be a field of characteristic $p$ and finite degree of imperfection $r$, and let $U = \mathscr{V}_{P, L}$ or $\R_{k^{1/p}/k}(\alpha_p)$. Then the only quasi-weakly permawound $k$-subgroups of $U$ are $0$ and $U$.
\end{corollary}

\begin{proof}
If $0 \subsetneq H \subsetneq U$, then $U$ is a semiwound extension of the nontrivial $k$-group $U/H$ -- which is a vector group by Proposition \ref{quotvralphsplit} -- by $H$. The preimage in $U$ of any copy of $\Ga$ inside $U/H$ then yields a semiwound extension of $\Ga$ by $H$.
\end{proof}

\begin{lemma}
\label{nononzhomrestalphaptoV}
Let $k$ be a field of characteristic $p$ of finite degree of imperfection. Then there is no nonzero $k$-homomorphism $\R_{k^{1/p^n}/k}(\alpha_p) \rightarrow \mathscr{V}_{P, L}$.
\end{lemma}

\begin{proof}
Due to Proposition \ref{eqngenweilrestalphap}, it suffices to treat the case $n = 1$. Let $\phi\colon \R_{k^{1/p}/k}(\alpha_p) \rightarrow \mathscr{V}_{P, L}$ be a $k$-homomorphism. Because $\mathscr{V}_{P, L}$ is semiwound \cite[Lem.\,B.1.7,$(1)\Longrightarrow(2)$]{cgp}, Proposition \ref{quotvralphsplit} implies that $\phi$ is either $0$ or injective. If $\phi$ is injective, then for dimension reasons it is an isomorphism. But $\R_{k^{1/p}/k}(\alpha_p)$ is totally nonsmooth while $\mathscr{V}_{P, L}$ is smooth and nontrivial, so the two groups cannot be isomorphic. Hence $\phi = 0$.
\end{proof}

\section{Preparations for rigidity}
\label{prepforrigiditysection}

In this section we lay the groundwork for the proof of the rigidity property \ref{rigidityintro} of permawound groups. The general strategy for proving rigidity is as follows. The key is to show that for any nontrivial smooth commutative semiwound unipotent $k$-group (not necessarily permawound), there is a nonzero $k_s$-homomorphism to $\mathscr{V}$. We may assume that $k$ is finitely generated over $\F_p$, and in that case we note that (for any $P, L$ as in Definition \ref{V_PLdef}) $\calHom(U, \mathscr{V}_{P,L})$ is the automorphism functor of the trivial extension of $U$ by $\mathscr{V}_{P,L}$. In particular, one obtains an exact sequence
\[
0 \longrightarrow {\rm{H}}^1(\Gal(k_s/k), \Hom_{k_s}(U, \mathscr{V}_{P,L})) \longrightarrow \Ext^1_k(U, \mathscr{V}_{P,L}) \longrightarrow \Ext^1_{k_s}(U, \mathscr{V}_{P,L}).
\]
In order to show that $\Hom_{k_s}(U, \mathscr{V}) \neq 0$, therefore, it suffices to construct a nontrivial extension of $U$ by $\mathscr{V}_{P,L}$ over $k$ that splits over $k_s$. In this section we prove two fundamental ingredients that will allow us to accomplish this in \S\ref{rigiditysection}. First, we consider the exact sequence defining $\mathscr{V}_{P,L}$
\[
0 \longrightarrow \mathscr{V}_{P,L} \longrightarrow \Ga^{p^r} \xlongrightarrow{P+L} \Ga \longrightarrow 0,
\]
and show that the induced map $\Hom(U, \Ga^{p^r}) \rightarrow \Hom(U, \Ga)$ is surjective over $k_s$ (Lemma \ref{mapofhomssurj}). This will allow us to show that certain extensions of $U$ by $\mathscr{V}$ over $k_s$ are trivial. Second, we will show that, over $k$ and for suitably chosen $P$ and $L$, there are nontrivial extensions of $\Z/p\Z$ by $\mathscr{V}_{P,L}$ that trivialize over $k_s$ (Proposition \ref{ext^1Z/pZVnotinj}). This will allow us to construct suitable nontrivial extensions of $U$ by $\mathscr{V}_{P,L}$ over $k$.

We begin with the following lemma, which we will require in the proof of the first ingredient mentioned above.

\begin{lemma}
\label{modnocommppart}
Let $k$ be a field with $p$-basis $\{\lambda_s\}_{s \in S}$, and let $Q \in k[X_1, \dots, X_n]$ be a homogeneous $p$-polynomial of degree $p^d$. Then there is an invertible linear change of variables $L$ in the $X_i$ such that $Q \circ L = \sum_{i=1}^n c_iX_i^{p^d}$ has the following property. Let $I$ be the set of finitely-supported functions $S \rightarrow \{0, \dots, p^d-1\}$. Then for any $1 \leq j < i \leq n$ with $c_j \neq 0$, one has $c_ic_j^{-1} = \sum_{f \in I} \left( \prod_{s\in S} \lambda_s^{f(s)} \right) a_f^{p^d}$ with $a_f \in k$ such that $a_0 = 0$.
\end{lemma}

\begin{proof}
We prove by induction on $m \leq n$ the existence of a change of variables such that the $f = 0$ coefficient in the expansion of $c_ic_j^{-1}$ vanishes whenever $j \leq m$, $j < i$, and $c_j \neq 0$. The $m = 0$ case is trivial, so assume that $m > 0$. We may assume by induction that the $f = 0$ term in the expansion of $c_ic_j^{-1}$ vanishes when $j < m$, $j < i$, and $c_j \neq 0$. For each $1 \leq j < i \leq n$ with $c_j \neq 0$, write 
\begin{equation}
\label{modnocommppartpfeqn1}
c_ic_j^{-1} = \sum_{f \in I} \left( \prod_{s\in S} \lambda_s^{f(s)} \right) a_{f, i, j}^{p^d}
\end{equation}
with $a_{f,i,j} \in k$. We know that $a_{0, i, j} = 0$ if $j < m$ and $j < i$. Make the change of variables which fixes $X_i$ for $i \neq m$, but which sends $X_m$ to $X_m - \sum_{i > m} a_{0, i, m}X_i$, and let $Q'$ be the resulting homogeneous $p$-polynomial of degree $p^d$. Write $Q' = \sum_{i=1}^n b_iX_i^{p^d}$. Then we claim that the $f = 0$ term in the expansion of $b_ib_j^{-1}$ vanishes whenever $j \leq m$, $j < i$, and $b_j \neq 0$, which will complete the induction.

To prove the claim, we may suppose that $b_j \neq 0$, as otherwise there is nothing to prove. We first compute that
\begin{equation}
\label{modnocommppartpfeqn4}
b_i = \begin{cases}
c_i, & i \leq m \\
c_i - c_ma_{0,i,m}^{p^d}, & i > m.
\end{cases}
\end{equation}
First consider the case $j < m$. If $i \leq m$, then $b_i = c_i$ (and $b_j = c_j$), so $b_ib_j^{-1} = c_ic_j^{-1}$ has vanishing $f = 0$ term in its expansion. If $i > m$, on the other hand, one has, using (\ref{modnocommppartpfeqn1}),
\[
b_ib_j^{-1} = (c_i - c_ma_{0,i,m}^{p^d})c_j^{-1} = \sum_{f \in I} \left( \prod_{s\in S} \lambda_s^{f(s)} \right) (a_{f,i,j} - a_{f,m,j}a_{0,i,m})^{p^d}.
\]
The coefficient of $\prod_{s \in S} \lambda_s^0$ in the above expansion is $a_{0,i,j} - a_{0,m,j}a_{0,i,m}$, and $a_{0,t,j} = 0$ for all $t > j$ by induction (since $j < m$), so this vanishes, as required.

It remains to treat the case $j = m$. For all $m < i \leq n$, one has, using (\ref{modnocommppartpfeqn4}) and (\ref{modnocommppartpfeqn1}),
\[
b_ib_m^{-1} = (c_i - c_ma_{0,i,m}^{p^d})c_m^{-1} = \sum_{f \in I} \left( \prod_{s\in S} \lambda_s^{f(s)} \right) (a_{f,i,m} - \delta_{f, 0}a_{0,i,m})^{p^d},
\]
where 
$$\delta_{f, 0} = \begin{cases}
1, & f = 0\\
0, & f \neq 0.
\end{cases}$$ The coefficient of $\prod_{s \in S} \lambda_s^0$ in the above expansion vanishes, so the induction is complete.
\end{proof}

The following lemma and Corollary \ref{ptorextbyvsplits} -- which are the key properties that make the group $\mathscr{V}$ fundamental -- will allow us to show that certain extensions split over $k_s$.

\begin{lemma}
\label{mapofhomssurj}
Let $k$ be a separably closed field of finite degree of imperfection $r$, and let $\F\colon \Ga^{p^r} \rightarrow \Ga$, $F := P + L$ for any $P, L \in k[X_1, \dots, X_{p^r}]$ with $P$ and $L$ homogeneous, $P$ of degree $p$ and $L$ of degree $1$, such that $P$ is universal (equivalently, reduced) and $L \neq 0$. Finally, let $U$ be a commutative $p$-torsion semiwound smooth unipotent $k$-group scheme. Then the map $\phi\colon \Hom(U, \Ga^{p^r}) \rightarrow \Hom(U, \Ga)$ induced by $F$ is surjective.
\end{lemma}

\begin{proof}
By Lemma \ref{univpluslinchangeofvars}, it suffices to prove the lemma for a particular choice of $P$ and $L$. We choose a $p$-basis $\{\lambda_1, \dots, \lambda_r\}$ of $k$. Take
\[
P := \sum_{f \in I_1} \left( \prod_{i=1}^r \lambda_i^{f(i)} \right)X_f^p \in k[X_f \mid f \in I_1]
\]
and $L := -X_0$. The group $U$ is isomorphic to $\{G = 0\} \subset \Ga^m$ for some reduced (nonzero) $p$-polynomial $G \in k[Y_1, \dots, Y_m]$ \cite[Prop.\,B.1.13, Lem.\,B.1.7]{cgp}. Let us first prove the lemma under the assumption that the principal part $Q$ of $G$ is homogeneous, say of degree $p^d$. If $d = 0$, then reducedness of $G$ implies that $U = 0$, and there is nothing to prove, so assume that $d > 0$. Note that a linear change of variables in the $Y_j$ preserves the property that the degree $p^d$ part of $G$ is reduced, so in particular all of the variables still appear in it. Lemma \ref{modnocommppart} therefore allows us to make a linear change of variables in the $Y_j$ and thereby assume that
\begin{equation}
\label{mapofhomssurjpfeqn2}
Q = \sum_{i=1}^m c_iY_i^{p^d}
\end{equation}
where all $c_j \neq 0$, and for all $1 \leq j < i \leq m$, one has
\begin{equation}
\label{mapofhomssurjpfeqn3}
c_ic_j^{-1} = \sum_{f \in I_d} \left( \prod_{i=1}^r \lambda_i^{f(i)}\right) a_{f,i,j}^{p^d} \mbox{ with $a_{0,i,j} = 0$}.
\end{equation}

We first prove the following claim. For any $s \geq 0$ and any $(\gamma_f)_f \in \Hom(U, \Ga)^{I_s}$, one has
\begin{equation}
\label{mapofhomssurjpfeqn1}
\sum_{f \in I_s} \left( \prod_{i=1}^r \lambda_i^{f(i)}\right) \gamma_f^{p^s} \equiv \gamma_0 \pmod{\im(\phi)},
\end{equation}
where $0 \colon \{1, \dots, r\} \rightarrow \{0, \dots, p^s-1\}$ denotes the constant function with value $0$. We proceed by induction on $s$, the $s = 0$ case being trivial. Suppose that $s > 0$. For $g \in I_1$, let
\[
\psi_g := \sum_{\substack{f \in I_s \\ f \equiv g \pmod{p}}} \left( \prod_{i=1}^r \lambda_i^{(f(i)-g(i))/p} \right) \gamma_f^{p^{s-1}}.
\]
Then
\[
\sum_{f \in I_s} \left( \prod_{i=1}^r \lambda_i^{f(i)}\right) \gamma_f^{p^s} = P(\psi_g \mid g \in I_1) \equiv \psi_0 = \sum_{f \in I_{s-1}} \left( \prod_{i=1}^r \lambda_i^{f(i)}\right) \gamma_f^{p^{s-1}} \pmod{\im(\phi)}.
\]
By induction, this last expression is equivalent modulo $\im(\phi)$ to $\gamma_0$, which proves the claim.

Every $k$-homomorphism $U \rightarrow \Ga$ is a $p$-polynomial in the $Y_j$, as follows from the anti-equivalence of categories between groups killed by Verschiebung and $k[F]$-modules \cite[Ch.\,IV, \S3, Cor.\,6.7]{demazuregabriel}, so (\ref{mapofhomssurjpfeqn1}) implies that every such homomorphism is equivalent modulo $\im(\phi)$ to a $k$-linear combination of the $Y_j$. It therefore suffices to show that every such combination lies in $\im(\phi)$. Fix $r \in k$ and $1 \leq j \leq m$. Write
$$G = Q + \sum_{\substack{1 \leq i \leq m \\ 0 \leq s < d}} c_{is}Y_i^{p^s}.$$ Using (\ref{mapofhomssurjpfeqn2}), one obtains an equality of homomorphisms $U \rightarrow \Ga$
\begin{equation}
\label{mapofhomssurjpfeqn4}
0 = c_j^{-1}r^{p^d}G = \sum_{i=1}^m c_ic_j^{-1}(rY_i)^{p^d} + r^{p^d}\sum_{\substack{1 \leq i \leq m \\ 0 \leq s < d}} c_j^{-1}c_{is}Y_i^{p^s}.
\end{equation}
Writing each term $c_j^{-1}c_{is}$ out as a sum $\sum_{f \in I_s} (\prod_{i=1}^r \lambda_{j,i,s}^{f(i)})b_f^{p^s}$ and using (\ref{mapofhomssurjpfeqn1}), one obtains
\begin{equation}
\label{mapofhomssurjpfeqn5}
r^{p^d}\sum_{\substack{1 \leq i \leq m \\ 0 \leq s < d}} c_j^{-1}c_{is}Y_i^{p^s} \equiv \sum_{i=1}^m F_{i, j}(r^p)Y_i \pmod{\im(\phi)}
\end{equation}
for some $p$-polynomials $F_{i, j} \in k[X]$ which depend only on $i, j$ (and of course $G$), but not on $r$. To analyze the first sum in (\ref{mapofhomssurjpfeqn4}), we combine (\ref{mapofhomssurjpfeqn3}) and (\ref{mapofhomssurjpfeqn1}) (with $s = d$) to conclude that
\begin{equation}
\label{mapofhomssurjpfeqn6}
\sum_{i=1}^m c_ic_j^{-1}(rY_i)^{p^d} \equiv rY_j + \sum_{i < j} a_{0,i,j}rY_i \pmod{\im(\phi)}.
\end{equation}
Combining (\ref{mapofhomssurjpfeqn4}), (\ref{mapofhomssurjpfeqn5}), and (\ref{mapofhomssurjpfeqn6}), we obtain that, for all $1 \leq j \leq m$, there exist $b_{i,j} \in k$ and $p$-polynomials $F_{i, j} \in k[X]$ ($1 \leq i \leq m$) such that, for all $r \in k$, one has
\begin{equation}
\label{mapofhomssurjpfeqn7}
r\left(Y_j + \sum_{1 \leq i < j} b_{i,j}Y_i\right) - \sum_{i=1}^m F_{i,j}(r^p)Y_i \in \im(\phi).
\end{equation}
Let $B$ be the upper-triangular $m \times m$ matrix with entry $b_{i,j}$ in the $(i, j)$ position for $i < j$ and with $1$'s along the diagonal (in particular, $B$ is invertible), let $\vec{Y} := (Y_1, \dots, Y_m)^T$, let $\mathbf{F}$ be the $m\times m$ matrix over $k[X_j]$ with $F_{i,j}$ in the $(i, j)$ position, and let $\vec{\mathbf{1}} := (1, \dots, 1)^T \in k^m$. For any $\vec{r} := (r_1, \dots, r_m)^T \in k^m$, let $\vec{r}^p$ denote the vector obtained by raising each entry of $\vec{r}$ to the $p$th power, and let
\[
\alpha(\vec{r}) := B\vec{r} - \mathbf{F}(\vec{r}^p)\cdot \vec{\mathbf{1}} \in k^m.
\]
Add (\ref{mapofhomssurjpfeqn7}) with $r = r_j$ over $j = 1, \dots, m$ to obtain
\[
\alpha(\vec{r}) := \alpha(\vec{r})\cdot \vec{Y} \in \im(\phi).
\]
Consider the $k$-group homomorphism $\alpha \colon \Ga^m \rightarrow \Ga^m$ defined by the formula $\vec{r} \mapsto \alpha(\vec{r})$. Because the $\mathbf{F}$ term involves polynomials in the $r_j^p$, this map has differential $B$ which is invertible. It is therefore an \'etale homomorphism, hence surjective. Because $k$ is separably closed and $\alpha$ is \'etale surjective, $\alpha$ is surjective on $k$-points. It follows that $\sum_{i=1}^m t_iY_i \in \im(\phi)$ for all $t_1, \dots, t_m \in k$, which completes the proof of the lemma for $U$ which we recall was assumed to be defined by a reduced $G$ whose principal part is homogeneous.

We now prove the general case in which the principal part $Q$ of $G$ may fail to be homogeneous. For each $s \geq 0$ let $H_s \in k[X_f \mid f \in I_s]$ be the $p$-polynomial
\[
H_s := \sum_{f \in I_s} \left( \prod_{i=1}^r \lambda_i^{f(i)} \right)X_f^{p^s}.
\]
Then $H_s$ is reduced, and one has a surjective homomorphism $h_s\colon \{H_s =0\} \rightarrow \{H_{s-1} = 0\}$ for $s \geq 1$ defined by replacing each $Y_f$ with $f \in I_{s-1}$ by $H_1(X_{f,1}, \dots, X_{f,p^r})$. By Proposition \ref{eqnforweilrestalphap}, $\ker(h_s)$ is a power of $\R_{k^{1/p}/k}(\alpha_p)$, hence, by induction, $H_s$ admits a filtration with this group as the successive quotients. For each $1 \leq j \leq m$, let $p^{d_j} := {\rm{deg}}_{Y_j}(G)$. Let $N := \max_j d_j$, and for each $1 \leq j \leq m$, let $R_j \in k[Z_{j,1}, \dots, Z_{j,p^{r(N-d_j)}}]$ be the $p$-polynomial $H_{N-d_j}(Z_{j,1}, \dots, Z_{j, p^{r(N-d_j)}})$. Consider the $k$-group $W$ defined by the equation
\[
R := G(R_1, \dots, R_m) = 0.
\]
The map $\pi\colon W \rightarrow U$, $(Z_{i,j})_{i,j} \mapsto (R_1, \dots, R_m)$ is surjective with kernel $\prod_{d_i < N} \R_{k^{1/p^{N-d_i}}/k}(\alpha_p)$. Let $\beta \colon U \rightarrow \Ga$ be a $k$-homomorphism. We wish to show that $\beta$ lifts to a homomorphism $U \rightarrow \Ga^{p^r}$ via the map $\Hom(U, \Ga^{p^r}) \rightarrow \Hom(U, \Ga)$ induced by $F\colon \Ga^{p^r} \rightarrow \Ga$. By the homogeneous case, $\beta \circ \pi$ lifts to a homomorphism $f\colon W \rightarrow \Ga^{p^r}$. Now $F \circ f|_{\ker(\pi)} = 0$, so $f|{\ker(\pi})$ factors through a homomorphism $\ker(\pi) \rightarrow \mathscr{V}$. The only such map is $0$ by Lemma \ref{nononzhomrestalphaptoV}, hence $f$ kills $\ker(\pi)$, and therefore factors through a homomorphism $U \rightarrow \Ga^{p^r}$ lifting $\beta$.
\end{proof}

\begin{corollary}
\label{ptorextbyvsplits}
Let $k$ be a separably closed field of finite degree of imperfection, and let $U$ be a smooth commutative $p$-torsion semiwound unipotent $k$-group. Then every $p$-torsion extension of $U$ by $\mathscr{V}$ splits.
\end{corollary}

\begin{proof}
Write $\mathscr{V}$ as $\mathscr{V}_{P,L}$ for some $P$ and $L$ so that we have an exact sequence
\begin{equation}
\label{ptorextbyvsplitseqn1}
0 \longrightarrow \mathscr{V} \xlongrightarrow{i} \Ga^{p^r} \xlongrightarrow{P+L} \Ga \longrightarrow 0.
\end{equation}
Let $E$ be a $p$-torsion extension of $U$ by $\mathscr{V}$. Then $i_*E \in \Ext^1(U, \Ga^{p^r})$ is smooth and $p$-torsion, hence has vanishing Verschiebung map. It follows from Proposition \ref{splitvecsubgps} that $i_*(E)$ splits, hence $E$ lies in the image of the connecting map $\Hom(U, \Ga) \rightarrow \Ext^1(U, \mathscr{V})$ associated to (\ref{ptorextbyvsplitseqn1}). Then Lemma \ref{mapofhomssurj} implies that $E$ is trivial as an extension.
\end{proof}

Recall that the $k$-linear action on $\Ga$ induces a $k$-linear action on $\Ext^1(\Ga, \mathscr{F})$ for any fppf abelian sheaf $\mathscr{F}$ over $k$. Our next task is to compute $\Ext^1(\Ga, \mathscr{V}_{P, L})$. Let $r$ be the degree of imperfection of $k$. We have the exact sequence
\begin{equation}
\label{ext^1(Ga,V)computepfeqn1}
0 \longrightarrow \mathscr{V}_{P,L} \xlongrightarrow{i} \Ga^{p^r} \xlongrightarrow{F := P+L} \Ga \longrightarrow 0,
\end{equation}
which induces a connecting map $\delta\colon\Hom(\Ga, \Ga) \rightarrow \Ext^1(\Ga, \mathscr{V}_{P,L})$.

\begin{proposition}
\label{ext^1(Ga,V)compute}
Let $k$ be a field of finite degree of imperfection. The above map $k \rightarrow \Ext^1(\Ga, \mathscr{V}_{P,L})$, $\alpha \mapsto \delta(\alpha\cdot {\rm{Id}})$ is a $k$-linear isomorphism that is functorial in separable extensions $K/k$ of the same degree of imperfection.
\end{proposition}

\begin{proof}
$P$ remains universal over $K$ as in the lemma thanks to Proposition \ref{universaloverextensions}, and functoriality follows from the functoriality of the connecting map, so it only remains to show that the map is a $k$-linear isomorphism. The $k$-linearity follows from the definition of the $k$-action on $\Ext^1(\Ga, \mathscr{V}_{P,L})$. To prove that the map is an isomorphism, we first check that it is injective, which is equivalent to the assertion that it is nonzero. So it suffices to show that $\delta({\rm{Id}}) \neq 0$ -- that is, the sequence (\ref{ext^1(Ga,V)computepfeqn1}) does not split. If it did, then we would obtain an isomorphism $\Ga^{p^r} \simeq \mathscr{V}_{P,L} \times \Ga$, hence a surjection $\Ga^{p^r} \twoheadrightarrow \mathscr{V}_{P,L}$, in violation of the woundness of the latter group. Thus the map is injective.

It remains to prove surjectivity. So suppose given $E \in \Ext^1(\Ga, \mathscr{V}_{P,L})$. Then $E$ has vanishing Verschiebung by Lemma \ref{extGavertriv}, hence (referring to the exact sequence (\ref{ext^1(Ga,V)computepfeqn1})) so does $i_*(E) \in \Ext^1(\Ga, \Ga^{p^r})$, hence splits by Proposition \ref{splitvecsubgps}. It follows that $E$ lies in the image of the connecting map $\Hom(\Ga, \Ga) \rightarrow \Ext^1(\Ga, \mathscr{V}_{P,L})$, so $E = \delta(\phi)$ for some $\phi \in \Hom(\Ga, \Ga)$. 

We are free to modify $\phi$ by $F_*(\chi)$ for any $\chi \in \Hom(\Ga, \Ga^{p^r})$, and we claim that any element $\phi \in \Hom(\Ga, \Ga)$ is equivalent modulo $F_*(\Hom(\Ga, \Ga^{p^r}))$ to a $k$-multiple of the identity. We have that $\phi \in k[T]$ is a $p$-polynomial, which we may assume to be homogeneous. We then prove the claim by induction on ${\rm{deg}}(\phi)$, the degree $1$ case (or the case $\phi = 0$) being trivial. So assume that ${\rm{deg}}(\phi) = p^n > 1$. Because $P$ is universal and homogeneous of degree $p$, we may write $\phi = P(\psi_1, \dots, \psi_{p^r})$ for some $p$-polynomials $\psi_i \in k[T]$ either vanishing or of degree $p^{n-1}$. Then $$\phi \equiv -L(\psi_1, \dots, \psi_{p^r}) \pmod{F_*(\Hom(\Ga, \Ga^{p^r}))},$$ and the latter homomorphism either vanishes or has degree $p^{n-1}$, so we are done by induction.
\end{proof}

An important step in the proof of Theorem \ref{rigidityintro} entails proving that the groups $\mathscr{V}_{P, L}$ for suitable choice of $P$ and $L$ have nonvanishing cohomology when $K$ is infinite and finitely generated over $\F_p$. This will allow us to construct suitable nontrivial extensions of smooth unipotent groups by $\mathscr{V}_{P, L}$. We first prove the following analogous result over fraction fields of complete regular local rings.

\begin{proposition}
\label{H^1nonvanishing}
Let $R$ be a complete $r$-dimensional regular local ring of characteristic $p$ with maximal ideal $\mathfrak{m}$, fraction field $K$, and perfect residue field $R/\mathfrak{m}$. By the Cohen structure theorem, we may choose a subfield $k \subset R$ which maps isomorphically onto the residue field. Choose elements $\lambda_1, \dots, \lambda_r \in R$, and for each $i$ let $c_i \in k$ be the unique element such that $\lambda_i - c_i \in \mathfrak{m}$. Assume that the elements $\lambda_i - c_i$ are generators for $\mathfrak{m}$. Then the $\lambda_i$ form a $p$-basis for $K$. Let $P := \sum_{f \in I} \left( \prod_{i=1}^r \lambda_i^{f(i)} \right) X_f^p$ and $L := -X_{c_{p-1}}$, where $c_{p-1}$ denotes the constant function with value $p-1$. Assume that ${\rm{H}}^1(k, \Z/p\Z) \neq 0$. Then ${\rm{H}}^1(K, \mathscr{V}_{P, L}) \neq 0$.
\end{proposition}

\begin{proof}
The continuous $k$-algebra map $k\llbracket X_1, \dots, X_r\rrbracket \rightarrow R$, $X_i \mapsto \lambda_i-c_i$, is an isomorphism. It follows that the $\lambda_i - c_i$ form a $p$-basis for $k$, hence -- because $c \in k = k^p$ -- so too do the $\lambda_i$. We identify $R$ with $k\llbracket X_1, \dots, X_r\rrbracket$ and $\lambda_i$ with $X_i + c_i$. Order the group $G := \Z^r$ lexicographically. This makes $\Z^r$ into an ordered abelian group. Consider the ring $k\llbracket T^G\rrbracket$ of Hahn series with value group $G$. Recall that this is the ring of formal power series $f = \sum_{g \in G} c_gT^g$ with $c_g \in k$ such that the set $\{g \in G\mid c_g \neq 0\}$ is well-ordered. Then $R \subset k\llbracket T^G \rrbracket$ via $X_1^{e_1}\dots X_r^{e_r} \mapsto T^{(e_1, \dots, e_r)}$, hence $K \subset k\llbracket T^G\rrbracket$. For $f \in k\llbracket T^G\rrbracket$, we define the {\em residue} of $f$ to be the coefficient of $X_1^{-1}\dots X_r^{-1}$ in the expansion of $f$.

Using the exact sequence
\[
0 \longrightarrow \mathscr{V}_{P, L} \longrightarrow \Ga^I \xlongrightarrow{F := P+L} \Ga \longrightarrow 0,
\]
we see that the nonvanishing of ${\rm{H}}^1(K, \mathscr{V}_{P, L})$ is equivalent to the non-surjectivity of the map $K^I \rightarrow K$ induced by $F$. To show that this map is not surjective, we compute the residue of $F((\alpha_f)_f)$ for any $(\alpha_f)_f \in K^I$. More generally, we compute this residue for any $(\alpha_f)_f \in k\llbracket T^G\rrbracket^I$. The only term in $P((\alpha_f)_f) = \sum_{f \in I} \left( \prod_{i=1}^r \lambda_i^{f(i)} \right) X_f^p = \sum_{f \in I} \left( \prod_{i=1}^r (X_i+c_i)^{f(i)}\right)\alpha_f^p$ that can contribute to the residue is the $f = c_{p-1}$ term. If write $\alpha_{c_{p-1}} = \sum_{g \in G} a_gT^g$, then the residue of $P((\alpha_f)_f)$ is $a_{g_{-1}}^p$, where $g_{-1} := (-1, \dots, -1) \in G$. The residue of $F((\alpha_f)_f) = (P+L)((\alpha_f)_f)$ is therefore $a_{g_{-1}}^p - a_{g_{-1}} \in k$. By assumption, ${\rm{H}}^1(k, \Z/p\Z) \neq 0$, so the Artin-Schreier map $k \rightarrow k$, $x \mapsto x^p-x$ is not surjective. For any element $\beta \in k$ not in the image of Artin-Schreier, the element $\beta (\lambda_1-c_1)^{-1}\dots (\lambda_r-c_r)^{-1} \in K$ does not lie in $F(K^I)$, so we conclude that ${\rm{H}}^1(K, \mathscr{V}_{P, L}) \neq 0$.
\end{proof}

As a corollary, we obtain the following result. Later we will prove the much stronger Theorem \ref{infofcohomintro}.

\begin{corollary}
\label{H^1(V)nonzero}
Let $K$ be a field finitely generated over $\F_p$, and let $\lambda_1, \dots, \lambda_r \in K$ be a $p$-basis. Also let $$P := \sum_{f \in I} \left( \prod_{i=1}^r \lambda_i^{f(i)} \right)X_f^p$$ and $L := -X_{c_{p-1}} \in k[X_f \mid f \in I]$, where $c_{p-1} \in I$ is the constant function with value $p-1$. Then ${\rm{H}}^1(K, \mathscr{V}_{P,L}) \neq 0$.
\end{corollary}

\begin{proof}
Write $K$ as the function field of a smooth integral $\F_p$-scheme $X$ of finite type. Shrinking $X$ if necessary, we may assume that the $d\lambda_i$ freely generate the sheaf $\Omega^1_{X/\F_p}$. Let $x \in X$ be a closed point, and let $R$ be the completion of $\calO_{X, x}$. Let $\kappa \subset R$ be the subfield mapping isomorphically onto the residue field (it is unique because the residue field is finite), and for each $i$, let $c_i \in \kappa$ be the element such that $\lambda_i - c_i$ lies in the maximal ideal $\mathfrak{m}$ of $R$. Because the $d\lambda_i$ generate $\Omega^1_{X/\F_p}$, the elements $\lambda_i-c_i$ generate $\mathfrak{m}$. Let $K_x := {\rm{Frac}}(R)$. Proposition \ref{H^1nonvanishing} implies that ${\rm{H}}^1(K_x, \mathscr{V}_{P, L}) \neq 0$. Therefore the map $F_{K_x} = P + L\colon K_x^I \rightarrow K_x$ is not surjective, and we need to prove the same with $K_x$ replaced by $K$. Let $\beta \in K_x$ be an element not lying in the image of $F_{K_x}$. Because $F$ is smooth, $\im(F_{K_x})$ contains an open neighborhood of $0 \in K_x$, so an open neighborhood of $\beta$ is disjoint from $\im(F_{K_x})$. Because the map $K \rightarrow K_x$ has dense image, it follows that there is an element of $K$ not in the image of $F_{K_x}$, hence not in the image of $F_K$.
\end{proof}

\begin{proposition}
\label{ext^1Z/pZVnotinj}
Let $k$ be a finitely generated extension field of $\F_p$, $\lambda_1, \dots, \lambda_r \in k$ a $p$-basis. Let $$P := \sum_{f \in I} \left( \prod_{i=1}^r \lambda_i^{f(i)}\right)X_f^p$$ and $L := -X_{c_{p-1}} \in k[X_f\mid f \in I]$, where $c_{p-1} \in I$ is the constant function with value $p-1$. Then the map $\Ext^1_k(\Z/p\Z, \mathscr{V}_{P, L}) \rightarrow \Ext^1_{k_s}(\Z/p\Z, \mathscr{V}_{P, L})$ is not injective.
\end{proposition}

\begin{proof}
Consider the exact sequence
\[
0 \longrightarrow \Z/p\Z \longrightarrow \Ga \xlongrightarrow{\wp} \Ga \longrightarrow 0,
\]
where $\wp$ is the Artin-Schreier map $x \mapsto x^p-x$. This induces for any extension field $K$ of $k$ an inclusion $\Ext^1_K(\Ga, \mathscr{V}_{P,L})/\wp^*(\Ext^1_K(\Ga, \mathscr{V}_{P,L})) \hookrightarrow \Ext^1_K(\Z/p\Z, \mathscr{V}_{P,L})$. We claim that the former group is nontrivial when $K = k$, but vanishes when $K = k_s$, which will prove the proposition. Via Proposition \ref{ext^1(Ga,V)compute}, when $K/k$ is separable algebraic, we may identify $\Ext^1_K(\Ga, \mathscr{V}_{P,L})$ $K$-linearly with $K$ via $\phi\colon \alpha \mapsto \delta(\alpha\cdot {\rm{Id}})$ in the notation of Proposition \ref{ext^1(Ga,V)compute}. Write $\wp = F - I$, where $F$ is the Frobenius isogeny and $I$ is the identity. Then $F^*(\delta(\alpha\cdot{\rm{Id}})) = \delta(\alpha F)$. We seek to write this in the form $\delta(\beta\cdot{\rm{Id}})$ for some $\beta \in K$. 

Write $\alpha = P((\gamma_f)_{f\in I})$ with $\gamma_f \in K$. Then $$\alpha F = P((\gamma_f\cdot {\rm{Id}})_{f \in I}) \equiv \gamma_{c_{p-1}}\cdot {\rm{Id}} \pmod{(P+L)_*(\Hom(\Ga, \Ga^I))},$$ so $\delta(\alpha F) = \delta(\gamma_{c_{p-1}}\cdot {\rm{Id}})$. In the map $\wp_K^*\colon \Ext^1_K(\Ga, \mathscr{V}_{P,L}) \rightarrow \Ext^1_K(\Ga, \mathscr{V}_{P,L})$, we identify the source with $K^I$ via $\phi \circ P$ and the target with $K$ via $\phi$. Then the induced map $\wp^*\colon K^I \rightarrow K$ is $-L-P$, hence the cokernel is isomorphic to ${\rm{H}}^1(K, \mathscr{V}_{P,L})$. This is nonzero when $K = k$ by Corollary \ref{H^1(V)nonzero} and vanishes when $K = k_s$ because $\mathscr{V}_{P,L}$ is smooth.
\end{proof}

\section{Rigidity}
\label{rigiditysection}

In this section we apply the results of \S\ref{prepforrigiditysection} to prove the crucial rigidity property of permawound groups (Theorem \ref{filtrationasuniv}), and we also give a version of the rigidity property over general, not necessarily separably closed fields of finite degree of imperfection (Theorem \ref{filtasunivgenfields}).

\begin{lemma}
\label{ext^2GaVinj}
Let $k$ be a field of finite degree of imperfection $r$. For any $P, L$, the map $\Ext^2_k(\Ga, \mathscr{V}_{P,L}) \rightarrow \Ext^2_{k_s}(\Ga, \mathscr{V}_{P,L})$ is injective.
\end{lemma}

\begin{proof}
We have the Hochschild-Serre spectral sequence
\[
E_2^{i,j} = {\rm{H}}^i(\Gal(k_s/k), \Ext^j_{k_s}(\Ga, \mathscr{V}_{P,L})) \Longrightarrow \Ext^{i+j}_k(\Ga, \mathscr{V}_{P,L}).
\]
It suffices to prove that $E_2^{2, 0} = E_2^{1,1} = 0$. The group $E_2^{2,0}$ vanishes because $\mathscr{V}_{P,L}$ is semiwound, so $\Hom_{k_s}(\Ga, \mathscr{V}_{P,L}) = 0$. For $E_2^{1,1}$, Proposition \ref{ext^1(Ga,V)compute} provides a functorial -- hence Galois-equivariant -- isomorphism $\Ext^1_{k_s}(\Ga, \mathscr{V}_{P,L}) \simeq k_s$. It follows that \[
{\rm{H}}^1(\Gal(k_s/k), \Ext^1_{k_s}(\Ga, \mathscr{V}_{P,L})) = 0.\qedhere
\]
\end{proof}

\begin{lemma}
\label{nonsplitcommptorquot}
Let $k$ be a field of characteristic $p$. Then any smooth non-split unipotent $k$-group $U$ admits a nontrivial smooth semiwound commutative $p$-torsion quotient.
\end{lemma}

\begin{proof}
If $U$ is not connected, then it admits a nontrivial finite \'etale quotient, which admits a nontrivial finite \'etale commutative $p$-torsion quotient. If, on the other hand, $U$ is connected, then by \cite[Th.\,B.3.4]{cgp}, there is a normal split unipotent $k$-subgroup $U_{\rm{split}} \trianglelefteq U$ such that $U/U_{\rm{split}}$ is wound. We may therefore assume that $U$ is a nontrivial wound group. Then \cite[Cor.\,B.3.3]{cgp} implies that $U$ admits a nontrivial wound commutative $p$-torsion quotient.
\end{proof}

We are now ready to prove the following fundamental result, which is of significant interest in its own right and is also the key to proving Theorem \ref{rigidityintro}.

\begin{theorem}
\label{nonzerohomtoV}
Let $k$ be a separably closed field of finite degree of imperfection $r > 0$, and let $U$ be a smooth nonsplit unipotent $k$-group scheme. Then $\Hom(U, \mathscr{V}) \neq 0$.
\end{theorem}

\begin{proof}
By Lemma \ref{nonsplitcommptorquot}, we may assume that $U$ is semiwound, commutative, and $p$-torsion. Choose a finitely generated extension field $K/\F_p$ of degree of imperfection $r$ such that $U$ is defined over $K$ and $k/K$ is separable. We may assume that $k = K_s$. By \cite[Lem.\,B.1.10]{cgp}, there is a finite \'etale $K$-group scheme $E \subset U$ such that $U/E$ is a vector group. Replacing $K$ by a finite separable extension, therefore, we may assume that there exists for some $m > 0, n \geq 0$ an exact sequence of $K$-group schemes
\begin{equation}
\label{nonzerohomtoVpfeqn1}
0 \longrightarrow (\Z/p\Z)^m \longrightarrow U \xlongrightarrow{f} \Ga^n \longrightarrow 0.
\end{equation}

Choose a $p$-basis $\lambda_1, \dots, \lambda_r$ of $K$, let $P := \sum_{f \in I} \left( \prod_{i=1}^r \lambda_i^{f(i)}\right)X_f^p$ and $L := -X_{c_{p-1}} \in k[X_f\mid f \in I]$, where $c_{p-1} \in I$ is the constant function with value $p-1$. Because the automorphism functor of the trivial extension of $U$ by $\mathscr{V}_{P,L}$ is $\calHom(U, \mathscr{V}_{P,L})$, one has an exact sequence
\begin{equation}
\label{nonzerohomtoVpfeqn2}
0 \longrightarrow {\rm{H}}^1(\Gal(K_s/K), \Hom_{K_s}(U, \mathscr{V}_{P,L})) \longrightarrow \Ext^1_K(U, \mathscr{V}_{P,L}) \longrightarrow \Ext^1_{K_s}(U, \mathscr{V}_{P,L})
\end{equation}
In order to show that $\Hom_{K_s}(U, \mathscr{V}) = \Hom_{K_s}(U, \mathscr{V}_{P,L})$ is nonzero, therefore, it suffices to show that the final map above is not injective.

Using Proposition \ref{ext^1Z/pZVnotinj}, choose a nonzero element $\alpha \in \Ext^1_K((\Z/p\Z)^m, \mathscr{V}_{P,L})$ that dies over $K_s$. The image of $\alpha$ under the connecting map $\Ext^1_K((\Z/p\Z)^m, \mathscr{V}_{P,L}) \rightarrow \Ext^2(\Ga^n, \mathscr{V})$ then also dies over $K_s$, so by Lemma \ref{ext^2GaVinj}, $\alpha$ lifts to an element -- necessarily nonzero -- $\beta \in \Ext^1_K(U, \mathscr{V}_{P,L})$, and we claim that $\beta$ dies over $K_s$, which will complete the proof that the final map in (\ref{nonzerohomtoVpfeqn2}) is not injective, hence of the theorem.

Since $\alpha_{K_s} = 0$, $\beta_{K_s}$ lifts to an element $\gamma \in \Ext^1_{K_s}(\Ga^n, \mathscr{V}_{P,L})$. The extension $\gamma$ is represented by a scheme by the effectivity of fpqc descent for relatively affine schemes. This group scheme has vanishing Verschiebung by Lemma \ref{extGavertriv}. It follows that the group scheme corresponding to $\beta_{K_s}$ has vanishing Verschiebung, hence is $p$-torsion, so $\beta_{K_s}$ vanishes by Corollary \ref{ptorextbyvsplits}.
\end{proof}

Our next task is to use Theorem \ref{nonzerohomtoV} to construct a filtration on permawound groups as in Theorem \ref{rigidityintro}. We first require the following proposition.

\begin{proposition}
\label{nonsmmapstorestalphap}
Let $k$ be a field of characteristic $p$ and of finite degree of imperfection, and let $U$ be a commutative unipotent $k$-group scheme that is not smooth. Then $\Hom(U, \R_{k^{1/p}/k}(\alpha_p)) \neq 0$.
\end{proposition}

\begin{proof}
Let $V := (U_{k_{\perf}})_{\red}$, a smooth $k_{\perf}$-subgroup scheme of $U_{k_{\perf}}$. Then $V \neq U_{k_{\perf}}$ because $U$ is not smooth. Therefore, $U_{k_{\perf}}/V$ is a nontrivial infinitesimal unipotent $k_{\perf}$-group scheme. Such a group admits a filtration by nontrivial groups with vanishing Frobenius and Verschiebung. By \cite[Ch.\,IV, \S3, Cor.\,6.7]{demazuregabriel}, therefore, $U_{k_{\perf}}/V$ admits a surjective homomorphism onto $(\alpha_p)_{k_{\perf}}$. We thus obtain a nonzero $k_{\perf}$-homomorphism $U_{k_{\perf}} \rightarrow (\alpha_p)_{k_{\perf}}$. This descends to a homomorphism $U_{k^{1/p^n}} \rightarrow (\alpha_p)_{k_{1/p^n}}$ for some $n > 0$. We therefore obtain a nonzero $k$-homomorphism $U \rightarrow \R_{k^{1/p^n}/k}(\alpha_p)$. (The assumption that $k$ has finite degree of imperfection is used here to ensure that $k^{1/p^n}/k$ is a finite extension, and hence that the Weil restriction is (represented by) a scheme.) We are then done by Proposition \ref{eqngenweilrestalphap}.
\end{proof}

Before proving Theorem \ref{rigidityintro}, we prove the analogous statement for weakly permawound groups.

\begin{theorem}
\label{filtrationweakasuniv}
Let $k$ be a separably closed field of finite degree of imperfection $r > 0$, and let $U$ be a semiwound weakly permawound unipotent $k$-group scheme. Then $U$ admits a filtration $0 = U_0 \subset U_1 \subset \dots \subset U_m = U$ such that, for each $1 \leq j \leq m$, either $U_i/U_{i-1} \simeq \mathscr{V}$ or $U_i/U_{i-1} \simeq \R_{k^{1/p}/k}(\alpha_p)$.
\end{theorem}

\begin{proof}
We proceed by induction on ${\rm{dim}}(U)$. If $U = 0$, then we are done, so assume that $U \neq 0$. If $U$ is not smooth, then Proposition \ref{nonsmmapstorestalphap} provides a nonzero $k$-homomorphism $U \rightarrow \R_{k^{1/p}/k}(\alpha_p)$. If, on the other hand, $U$ is smooth, then Theorem \ref{nonzerohomtoV} provides a nonzero $k$-homomorphism $U \rightarrow \mathscr{V}$. Letting $W$ denote either $\mathscr{V}$ or $\R_{k^{1/p}/k}(\alpha_p)$, depending on whether $U$ is smooth or not, we therefore have a nonzero $k$-homomorphism $f\colon U \rightarrow W$. Proposition \ref{quotofasunivisasuniv} implies that $\im(f)$ is weakly permawound, hence quasi-weakly permawound, so by Corollary \ref{VRsubgpnotasuniv}, $f$ is surjective. Further, $\ker(f)$ is weakly permawound by Proposition \ref{qwasunivsubgps}, so we are done by induction.
\end{proof}

\begin{theorem}$($Theorem $\ref{rigidityintro}$$)$
\label{filtrationasuniv}
Let $k$ be a separably closed field of finite degree of imperfection $r > 0$, and let $U$ be a wound permawound unipotent $k$-group scheme. Then $U$ admits a filtration $1 = U_0 \trianglelefteq U_1 \trianglelefteq \dots \trianglelefteq U_m = U$ such that, for each $1 \leq i \leq m$, either $U_i/U_{i-1} \simeq \mathscr{\mathscr{V}}$ or $U_i/U_{i-1} \simeq \R_{k^{1/p}/k}(\alpha_p)$.
\end{theorem}

\begin{proof}
We proceed by induction on ${\rm{dim}}(U)$, the $0$-dimensional case being trivial. So assume that $U \neq 1$. By \cite[Cor.\,B.3.3]{cgp}, there is a nontrivial smooth central (and even $p$-torsion) $k$-subgroup $U' \subset U$ such that $U/U'$ is wound. Then $U/U'$ is permawound by Proposition \ref{quotofasunivisasuniv}, and $U'$ is weakly permawound by Proposition \ref{qwasunivsubgps}, therefore permawound by Proposition \ref{asuniv=weakasuniv}. By induction, therefore, we may assume that $U$ is commutative, and we are then done by Theorem \ref{filtrationweakasuniv}.
\end{proof}

We now prove a version of Theorems \ref{filtrationweakasuniv} and \ref{filtrationasuniv} over a field that is not separably closed. In this case $U$ may not admit a filtration by the groups $\R_{k^{1/p}/k}(\alpha_p)$ and $\mathscr{V}_{P,L}$ for varying $P$ and $L$. Instead, the best we can do is to obtain a filtration by $\R_{k^{1/p}/k}(\alpha_p)$ and $k_s/k$-forms of powers of a fixed $\mathscr{V}_{P,L}$. (Which $P$ and $L$ we choose does not matter, thanks to Corollary \ref{univpluslinsame}.) As we will see, such powers are classified by Galois representations of $k$ over $\F_p$. The key fact that we require is the following.

\begin{proposition}
\label{End(V)=F_p}
Let $k$ be a field of characteristic $p$ and of finite degree of imperfection $r > 0$. Then for any $P, L$ as in Definition $\ref{V_PLdef}$, $\End(\mathscr{V}_{P,L}) = \F_p$.
\end{proposition}

\begin{proof}
We may extend scalars to $k_s$ and thereby assume that $k$ is separably closed. It then suffices by Corollary \ref{univpluslinsame} to prove the proposition for a single choice of $P$ and $L$. Let us take the usual $P$: choose a $p$-basis $\lambda_1, \dots, \lambda_r$ of $k$, and take $P := \sum_{f \in I}\left(\prod_{i=1}^r \lambda_i^{f(i)}\right)X_f^p$ and $L := -X_0$, and let $F := P+L$. An endomorphism $\psi\colon \mathscr{V} \rightarrow \mathscr{V}$ is a collection of homomorphisms $G_f\colon \mathscr{V} \rightarrow \Ga$ for $f \in I$ such that 
\begin{equation}
\label{End(V)=F_ppfeqn1}
F((G_f)_{f \in I}) = 0.
\end{equation}
By \cite[Prop.\,6.4]{rosmodulispaces}, each $G_f$ is of the form $Q_f((X_g)_{g \in I})$ for some $p$-polynomial $Q_f \in k[Y_f \mid f \in I]$ with ${\rm{deg}}_{Y_0}(Q_f) \leq 1$. We claim that each $Q_f$ is linear. Indeed, suppose to the contrary that this is not the case, and fix $g_0, f_0 \in I$ such that $1 < p^d := {\rm{deg}}_{Y_{g_0}}(Q_{f_0}) = \max_{f \in I} {\rm{deg}}_{Y_{g_0}}(Q_f)$. The coefficient of $X_{g_0}^{p^{d+1}}$ on the left side of (\ref{End(V)=F_ppfeqn1}) is $P((a_f)_{f \in I})$ for some $a_f \in k$, not all $0$, and is in particular nonzero because $P$ is reduced. Since each $Q_f$ has degree $\leq 1$ in $Y_0$, the left side of (\ref{End(V)=F_ppfeqn1}) has degree $\leq p$ in $X_0$. Using the equation $F = 0$ on $\mathscr{V}$ to eliminate the $X_0^p$ term leaves the coefficient of $X_{g_0}^{p^{d+1}}$ unchanged because $d > 0$. This then yields an equation on $\mathscr{V}$ of $p$-polynomials in the $X_f$ with degree $\leq 1$ in $X_0$. By the uniqueness aspect of \cite[Prop.\,6.4]{rosmodulispaces}, it follows that this is an identity of polynomials (and not just of functions on $\mathscr{V}$). Since the coefficient of $X_{g_0}^{p^{d+1}}$ is nonzero, this is a contradiction. We conclude that each $Q_f$ is linear, as claimed.

Write $Q_f = \sum_{g \in I} c_{fg}Y_g$ with $c_{fg} \in k$. Substituting this into (\ref{End(V)=F_ppfeqn1}) and eliminating the $X_0^p$ term using the equation $F = 0$ on $\mathscr{V}$, one obtains an equation of functions on $\mathscr{V}$
\begin{equation}
\label{End(V)=F_ppfeqn5}
\sum_{f \in I} \left( \prod_{i=1}^r \lambda_i^{f(i)}\right) \left[ \left(\sum_{0 \neq g \in I} c_{fg}^pX_g^p\right) + c_{f0}^p\left(X_0 - \sum_{0 \neq g \in I} \left( \prod_{i=1}^r \lambda_i^{g(i)} \right)X_g^p\right) \right] = \sum_{g \in I}c_{0g}X_g.
\end{equation}
Since both sides have degree $\leq 1$ in $X_0$, this is an identity of polynomials, again by the uniqueness aspect of \cite[Prop.\,6.4]{rosmodulispaces}. Comparing coefficients of $X_g$ for $g \neq 0$, we obtain
\begin{equation}
\label{End(V)=F_ppfeqn4}
c_{0g} = 0 \mbox{ for } g \neq 0.
\end{equation}
Comparing coefficients of $X_g^p$ for $g \neq 0$, we obtain
\begin{equation}
\label{End(V)=F_ppfeqn2}
\sum_{f \in I} \left( \prod_{i=1}^r \lambda_i^{f(i)}\right)\left[ c_{fg}^p - c_{f0}^p \left( \prod_{i=1}^r \lambda_i^{g(i)} \right) \right] = 0 \mbox{ for } g \neq 0.
\end{equation}
The left side may be rearranged in the form $P((a_t)_{t \in I})$ for some functions $a_t$ of the $c_{fg}$. Comparing the coefficients of $\prod_i \lambda_i^{f(i)}$ on both sides of the resulting expression in (\ref{End(V)=F_ppfeqn2}), we obtain
\begin{equation}
\label{End(V)=F_ppfeqn3}
c_{fg} = \prod_{i=1}^r \lambda_i^{\epsilon_{f, g}(i)}c_{h(f, g)0} \mbox{ for } g \neq 0,
\end{equation}
where $h(f, g) \in I$ is the unique function satisfying $h(f, g) \equiv f-g \pmod{p}$ and $\epsilon_{f, g}(i) := (h(f, g)(i) + g(i) - f(i))/p \in \{0, 1\}$. Applying (\ref{End(V)=F_ppfeqn3}) with $f = 0$ and (\ref{End(V)=F_ppfeqn4}), we obtain $c_{g0} = 0$ for $g \neq 0$. Applying (\ref{End(V)=F_ppfeqn3}) again with $f \neq g$, we obtain $c_{fg} = 0$ whenever $f \neq g$. That is, $c_{fg}$ is supported along the diagonal. Comparing coefficients of $X_0$ in (\ref{End(V)=F_ppfeqn5}), one obtains $c_{00}^p = c_{00}$, so $c_{00} \in \F_p$. Applying (\ref{End(V)=F_ppfeqn3}) with $f = g$, one obtains $c_{gg} = c_{00}$ for all $g$. Call this common value $c \in \F_p$ It follows that $Q_f = cY_f$ for all $f \in I$, hence $\psi\colon \mathscr{V} \rightarrow \mathscr{V}$ is just multiplication by some element of $\F_p$, as desired.
\end{proof}

Let $k$ be a field of characteristic $p$ and finite degree of imperfection $r$. For any finite-dimensional $\F_p$-vector space $\mathscr{W}$, one may define $\mathscr{W} \otimes \mathscr{V}_{P,L}$ to be the tensor product of the constant $\F_p$-vector space scheme $\mathscr{W}$ with the $\F_p$-vector space scheme $\mathscr{V}_{P,L}$. When $\mathscr{W} = \F_p^n$, then $\mathscr{W} \otimes \mathscr{V} = \mathscr{V}^n$ canonically. In general, a non-canonical choice of isomorphism $\mathscr{W} \simeq \F_p^n$ yields an isomorphism $\mathscr{W} \otimes \mathscr{V}_{P,L} \simeq \mathscr{V}_{P,L}^n$. The construction $\mathscr{W} \mapsto \mathscr{W} \otimes \mathscr{V}_{P,L}$ is covariant in $\mathscr{W}$. One has the following proposition, which is essentially a formal consequence of results we have already proven.

\begin{proposition}
\label{asunivsubgpV^n}
Let $k$ be a field of characteristic $p$ and finite degree of imperfection $r > 0$, and let $\mathscr{W}$ be a finite-dimensional $\F_p$-vector space. The quasi-weakly permawound $k$-subgroup schemes of $\mathscr{W} \otimes \mathscr{V}_{P,L}$ are exactly those of the form $\mathscr{U} \otimes \mathscr{V}_{P,L}$ for $\mathscr{U} \subset \mathscr{W}$ an $\F_p$-vector subspace.
\end{proposition}

\begin{proof}
Subgroups of the given form are isomorphic to powers of $\mathscr{V}_{P,L}$, hence permawound. To see that these are the only such subgroups, we proceed by induction on ${\rm{dim}}(\mathscr{W})$, the $0$-dimensional case being trivial. So assume that $\mathscr{W} \neq 0$. We may assume that $\mathscr{W} = \F_p^n$. Let $U \subset \mathscr{V}_{P,L}^n$ be quasi-weakly permawound. If $\mathscr{V}_i := \{0\}^{i-1} \times \mathscr{V}_{P,L} \times \{0\}^{n-i} \subset U$ for all $1 \leq i \leq n$, then $U = \mathscr{V}_{P,L}^n$, so assume that there is some $i$ such that $\mathscr{V}_i \not\subset U$.

Let $\pi^i\colon \mathscr{V}_{P,L}^n \rightarrow \mathscr{V}_{P,L}^{n-1}$ denote the projection away from the $i$th factor. The exact sequence
\[
0 \longrightarrow U \cap \mathscr{V}_i \longrightarrow U \xlongrightarrow{\pi^i} \mathscr{V}_{P,L}^{n-1},
\]
together with Proposition \ref{qwasunivsubgps}, shows that $U \cap \mathscr{V}_i$ is quasi-weakly permawound. Because $U \cap \mathscr{V}_i \neq \mathscr{V}_i$ by assumption, Corollary \ref{VRsubgpnotasuniv} implies that $U \cap \mathscr{V}_i = 0$. The map $\pi^i\colon U \rightarrow \mathscr{V}_{P,L}^{n-1}$ is therefore injective. By induction, $\pi^i(U) = \mathscr{W} \otimes \mathscr{V}_{P,L}$ for some $\F_p$-subspace $\mathscr{W} \subset \F_p^{n-1}$. Let $\pi_i\colon \mathscr{V}_{P,L}^n \rightarrow \mathscr{V}_{P,L}$ denote projection onto the $i$th factor. The composition $\mathscr{W} \otimes \mathscr{V}_{P,L} \xrightarrow{\sim} U \xrightarrow{\pi_i} \mathscr{V}_{P,L}$ is, thanks to Proposition \ref{End(V)=F_p}, induced by an $\F_p$-linear map $f\colon \mathscr{W} \rightarrow \F_p$. If $\mathscr{U} \subset \mathscr{W} \times \F_p$ is the graph of $f$, then $U = \mathscr{U} \otimes \mathscr{V}_{P,L}$.
\end{proof}

Let $k$ be a field of characteristic $p$ and finite degree of imperfection $r$, and fix a choice of $P$ and $L$ as in Definition \ref{V_PLdef}. Thanks to Proposition \ref{End(V)=F_p}, we have an isomorphism -- canonical up to a universal choice of sign -- between the group $\Hom_{{\rm{cts}}}(\Gal(k_s/k), \GL_n(\F_p))$ of continuous $n$-dimensional Galois representations of $k$ over $\F_p$ and the group ${\rm{H}}^1(\Gal(k_s/k), {\rm{Aut}}_{k_s}(\mathscr{V}_{P, L}^n))$ of $k_s/k$-forms of $\mathscr{V}_{P, L}^n$. Fix such a choice of sign, and, for such a Galois representation $\rho$, let $(\mathscr{V}_{P,L})_{\rho}$ denote the associated form of $\mathscr{V}_{P,L}^n$. The following two results are the versions of Theorems \ref{filtrationweakasuniv} and \ref{filtrationasuniv} over fields that are not necessarily separably closed.

\begin{theorem}
\label{filtrationweakasunivgenfields}
Let $k$ be a field of characteristic $p$ and finite degree of imperfection $r > 0$. Fix $P, L$ as in Definition $\ref{V_PLdef}$. Then for every semiwound weakly permawound $k$-group scheme $U$, there is a filtration $0 = U_0 \subset \dots \subset U_m = U$ by $k$-group schemes such that, for each $1 \leq i \leq m$, the $k$-group $U_i/U_{i-1}$ is isomorphic either to $\R_{k^{1/p}/k}(\alpha_p)$, or to $(\mathscr{V}_{P,L})_{\rho}$ for some irreducible continuous finite-dimensional $\F_p$-representation $\rho$ of $\Gal(k_s/k)$.
\end{theorem}

\begin{theorem}
\label{filtasunivgenfields}
Let $k$ be a field of characteristic $p$ and finite degree of imperfection $r > 0$. Fix $P, L$ as in Definition $\ref{V_PLdef}$. Then for every wound permawound $k$-group scheme $U$, there is a filtration $1 = U_0 \trianglelefteq \dots \trianglelefteq U_m = U$ by $k$-group schemes such that, for each $1 \leq i \leq m$, the $k$-group $U_i/U_{i-1}$ is isomorphic either to $\R_{k^{1/p}/k}(\alpha_p)$, or to $(\mathscr{V}_{P,L})_{\rho}$ for some irreducible continuous finite-dimensional $\F_p$-representation $\rho$ of $\Gal(k_s/k)$.
\end{theorem}

\begin{proof}[Proof of Theorems \ref{filtrationweakasunivgenfields} and \ref{filtasunivgenfields}]
The reduction of Theorem \ref{filtrationasuniv} to Theorem \ref{filtrationweakasuniv} also works here to reduce Theorem \ref{filtasunivgenfields} to Theorem \ref{filtrationweakasunivgenfields}, so we prove the latter. We show that, if $U \neq 0$, then it admits a surjection onto either $\R_{k^{1/p}/k}(\alpha_p)$ or $(\mathscr{V}_{P,L})_{\rho}$ for some irreducible $\rho$, and the proof then proceeds exactly as in the proof of Theorem \ref{filtrationweakasuniv}. If $U$ is not smooth, then Proposition \ref{nonsmmapstorestalphap} furnishes a nonzero homomorphism $f\colon U \rightarrow \R_{k^{1/p}/k}(\alpha_p)$. Since $U$ is weakly permawound, $\im(f)$ is quasi-weakly permawound, so $f$ is surjective by Corollary \ref{VRsubgpnotasuniv}.

Assume, on the other hand, that $U$ is smooth. By Theorem \ref{nonzerohomtoV}, there is a finite Galois extension $K/k$ and a nonzero $K$-homomorphism $\phi_1\colon U_K \rightarrow (\mathscr{V}_{P,L})_K$. Let $\phi_1, \dots, \phi_n\colon U_K \rightarrow (\mathscr{V}_{P,L})_K$ denote the $\Gal(K/k)$-conjugates of $\phi_1$, and let $H_i := \ker(\phi_i)$. Then $\cap_{i=1}^n H_i \subset U_K$ is $\Gal(K/k)$-invariant, hence descends uniquely to a $k$-subgroup $H \subset U$. Further, $H_K$ is the kernel of the map $\phi := \prod_i \phi_i\colon U_K \rightarrow \prod_i (\mathscr{V}_{P,L})_K$. Because $U$ is weakly permawound over $k$, it is so over $K$ by Proposition \ref{weakasunivinhbyexts}. Thus $\im(\phi)$ is quasi-weakly permawound, hence it is $K$-isomorphic to $\mathscr{V}_{P,L}^m$ for some $0 < m \leq n$ by Proposition \ref{asunivsubgpV^n}. (We also implicitly use Proposition \ref{universaloverextensions} to ensure that $P$ is still reduced and universal over $K$.) Therefore, $U/H$ is a $K/k$-form of $\mathscr{V}_{P,L}^m$, hence isomorphic to $(\mathscr{V}_{P,L})_{\beta}$ for some continuous $m$-dimensional Galois representation over $\F_p$. Choose a nonzero irreducible quotient $\rho$ of $\beta$. Then the composition $U/H \rightarrow (\mathscr{V}_{P,L})_{\beta} \rightarrow (\mathscr{V}_{P,L})_{\rho}$ is surjective.
\end{proof}

\section{Cohomology of unipotent groups}
\label{cohunipgpsection}

In this section we illustrate the utility of permawound unipotent groups by applying them to the study of the cohomology of unipotent groups. In particular, we prove that for a non-split smooth unipotent group $U$ over an infinite field $K$ finitely generated over $\F_p$, ${\rm{H}}^1(K, U)$ is infinite (Theorem \ref{infofcohomintro}).
The idea underlying the proof is quite simple. We reduce to the semiwound commutative $p$-torsion case, and ubiquity (Theorem \ref{ubiquity}) then allows us to assume that $U$ is permawound. If $U$ admits a surjection onto $\mathscr{V}_{P, L}$ with $P$ and $L$ as in Corollary \ref{H^1(V)nonzero}, then \cite[Lem.\,2.4]{nguyensros} reduces us to the case $U = \mathscr{V}_{P, L}$. This case is then handled by using Proposition \ref{H^1nonvanishing} to show that $\mathscr{V}_{P,L}$ has nonvanishing ${\rm{H}}^1$ over infinitely many completions of $K$, and the desired infinitude is then easily obtained from a version of weak approximation in this setting.

In general, rigidity (Theorem \ref{filtrationasuniv}) implies that $U$ surjects onto $\mathscr{V}_{P,L}$ not over $K$, but over a finite separable extension $E/K$. The trick is to show that the extension $E/K$ trivializes at infinitely many completions of $K$. That is, for infinitely many completions $K_w$ of $K$, the map ${\rm{Spec}}(E \otimes_K K_w) \rightarrow {\rm{Spec}}(K_w)$ admits a section. It then follows that over such $K_w$ there is a surjection $U \twoheadrightarrow \mathscr{V}_{P,L}$ and one then proceeds as above.

The key first step, therefore, is to obtain this trivialization at infinitely many places of $K$. This is a geometric analogue of the fact that, for a finite extension $F/K$ of number fields, infinitely many primes of $K$ split completely in $F$. The proof of that fact relies upon noting that the zeta functions of the two fields both have their rightmost pole at the same point $1$. (What one really needs is that it is at the same $x$-coordinate.) We will therefore proceed similarly to prove the desired local trivialization, by using properties of the zeta functions of varieties corresponding to the function fields $E$ and $K$. We begin with the following proposition -- interesting in its own right -- from which the local trivialization follows easily. In what follows, for a scheme $W$, $W_{\rm{cl}}$ denotes the set of closed points of $W$.

\begin{proposition}
\label{mapvarstrivinfloc}
Let $f\colon X \rightarrow Y$ be a generically finite morphism between positive-dimensional $\F_p$-schemes of finite type, and let $S$ be the set of closed points $y \in Y$ such that $f^{-1}(y) \rightarrow y$ admits a section. Then $S$ is infinite.
\end{proposition}

\begin{proof}
For an $\F_q$-scheme $W$ of finite type, recall that the zeta function of $W$ is the meromorphic function of $s$ defined by the Euler product
\[
\zeta_W(s) := \prod_{w \in W_{\rm{cl}}}(1 - |\kappa(w)|)^{-s})^{-1}.
\]
If $d := {\rm{dim}}(W)$, then this product converges absolutely and uniformly on ${\rm{Re}}(s) > d + \epsilon$ for any $\epsilon > 0$. Furthermore, by the celebrated Weil conjectures proved by Dwork, Grothendieck, and Deligne, it is a rational function of $T := q^{-s}$, and has a simple pole at $s = d$ when $W$ is irreducible.

In particular, these remarks apply to $X$ and $Y$, which have the same dimension -- $d$, say -- because of the generic finiteness of $f$. We may shrink $X$ and $Y$ and thereby assume them to be irreducible. If the set $S$ referred to in the proposition were finite, then one would have $$|\kappa(x)| \geq |\kappa(f(x))|^2$$ for all but finitely many $x \in X_{\rm{cl}}$. It would then follow that the Euler product expansion of $\zeta_X$ was absolutely and uniformly convergent on ${\rm{Re}}(s) > d/2 + \epsilon$, which is impossible as $\zeta_X$ has a pole at $s = d$.
\end{proof}

We may now prove the desired trivialization result.

\begin{proposition}
\label{infmanytrivpts}
Let $f\colon X \rightarrow Y$ be a generically \'etale morphism between positive-dimensional integral $\F_p$-schemes of finite type with function fields $E$ and $K$, respectively. For a closed point $y \in Y$, let $K_y$ denote the fraction field of the completion $\widehat{\calO}_{Y, y}$ of the local ring $\calO_{Y,y}$. Then the set $T$ of $y \in Y_{\rm{cl}}$ such that the map ${\rm{Spec}}(E \otimes_K K_y) \rightarrow {\rm{Spec}}(K_y)$ has a section is infinite.
\end{proposition}

\begin{proof}
We may shrink $X$ and $Y$ and thereby assume that $f$ is finite. Then let $S$ be as in Proposition \ref{mapvarstrivinfloc}. We will show that $S \subset T$. For $y \in S$, $f^{-1}(\widehat{\calO}_{Y, y})$ is the spectrum of a finite \'etale $\widehat{\calO}_{Y, y}$-algebra $A$. Because $\widehat{\calO}_{Y, y}$ is complete, and therefore Henselian, $A$ is a finite product of local finite \'etale $\widehat{\calO}_{Y, y}$-algebras, and by definition of $S$, at least one of these algebras has residue field isomorphic to that of $\widehat{\calO}_{Y, y}$. Because $\widehat{\calO}_{Y, y}$ is Henselian, it follows that the map ${\rm{Spec}}(A) \rightarrow {\rm{Spec}}(\widehat{\calO}_{Y, y})$ admits a section, and the desired conclusion then follows by passing to fraction fields.
\end{proof}

Before proving the main result, we require a version of weak approximation.

\begin{proposition}
\label{weakapprox}
Let $X$ be a nonempty integral quasi-projective $k$-scheme with function field $K$. For $x \in X_{\rm{cl}}$, let $K_x$ denote the fraction field of the completion $\widehat{\calO}_{X, x}$ of the local ring $\calO_{X, x}$. Then the image of the embedding $K \hookrightarrow \prod_{x \in X_{\rm{cl}}} K_x$ is dense, where the latter set is given the product topology arising from the $\mathfrak{m}_x$-adic topologies on the $K_x$, with $\mathfrak{m}_x$ the maximal ideal of $\calO_{X, x}$.
\end{proposition}

\begin{proof}
Let $S \subset X_{\rm{cl}}$ be a finite subset, and for each $x \in S$ let $\alpha_x \in K_x$. We must show that there exists $\beta \in K$ as close as we desire to each of the $\alpha_x$ simultaneously. Because $X$ is quasi-projective, any finite subset of $X$ is contained in an affine open, so we may assume that $X = {\rm{Spec}}(A)$ is affine. Because $K$ is dense in each $K_x$, we may assume that each $\alpha_x \in K$. Then multiplying through by a common denominator, we may even assume that each $\alpha_x \in A$, because choosing $\beta \in A$ suitably close to the $\alpha_x$ after clearing denominators will yield the desired approximation by redividing by the common denominator. Thus we have $\alpha_x \in A$ for $x \in S$ and we seek $\beta \in A$ as $x$-adically close as we may desire to each of the $\alpha_x$ simultaneously. But each $x$ corresponds to a maximal ideal $I_x$ of $A$. In particular, the $I_x$ are pairwise comaximal, hence the desired conclusion follows from the Chinese Remainder Theorem.
\end{proof}

We are now prepared to prove the main result of this section.

\begin{theorem}$($Theorem $\ref{infofcohomintro}$$)$
\label{infofcohom}
Let $K$ be an infinite finitely generated extension field of $\F_p$, and let $U$ be a smooth unipotent $K$-group scheme that is not split. Then ${\rm{H}}^1(K, U)$ is infinite.
\end{theorem}

\begin{proof}
Combining Lemma \ref{nonsplitcommptorquot} and \cite[Lem.\,2.4]{nguyensros}, we may assume that $U$ is semiwound, commutative, and $p$-torsion. By ubiquity (Theorem \ref{ubiquity}), we reduce to the case in which $U$ is permawound. Let $P$ and $L$ be as in Corollary \ref{H^1(V)nonzero}. By rigidity (Theorem \ref{filtrationasuniv}), for some finite separable extension $E/K$, $U_E$ admits a surjection onto either $\R_{E^{1/p}/E}(\alpha_p)$ or $(\mathscr{V}_{P,L})_E$. Because the former is not smooth, there must be a surjection onto the latter. Now write $K$, respectively $E$, as the function fields of smooth quasi-projective $\F_p$-schemes $Y$, respectively $X$, of finite type such that the extension $E/K$ arises via an $\F_p$-morphism $f\colon X \rightarrow Y$. Shrinking $Y$ (hence also $X$) if necessary, we may assume that the $d\lambda_i$ freely generate the sheaf $\Omega^1_{Y/\F_p}$. It then follows that, for each $y \in Y_{\rm{cl}}$, in the notation of Proposition \ref{H^1nonvanishing}, the elements $\lambda_i - c_i$ generate the maximal ideal of the local ring at $y$. Because $E/K$ is separable, $f$ is generically \'etale, so Proposition \ref{infmanytrivpts} ensures that, for infinitely many $y \in Y_{\rm{cl}}$, the map ${\rm{Spec}}(E \otimes_K K_y) \rightarrow {\rm{Spec}}(K_y)$ admits a section. In particular, there is a surjective $K_y$-homomorphism $U_y \twoheadrightarrow (\mathscr{V}_{P,L})_{K_y}$, so by Proposition \ref{H^1nonvanishing} and \cite[Lem.\,2.4]{nguyensros} again, we find that ${\rm{H}}^1(K_y, U) \neq 0$ for infinitely many $y \in Y_{\rm{cl}}$.

Because $U$ is smooth, commutative, and $p$-torsion, it is isomorphic to $\{F = 0\} \subset \Ga^n$ for some $p$-polynomial $0 \neq F \in k[X_1, \dots, X_n]$ \cite[Prop.\,B.1.13]{cgp}. Therefore, for any field extension $M/K$, one has a natural (once one fixes an isomorphism between $U$ and $\{F = 0\}$) isomorphism ${\rm{H}}^1(M, U) \simeq M/F(M^n)$. Let $m > 0$ be given. Choose a subset $S \subset Y_{\rm{cl}}$ of size $m$ such that ${\rm{H}}^1(K_y, U) \neq 0$ for all $y \in S$. This means that, for each $y \in S$, one has $\alpha_y \in K_y$ such that $\alpha_y \not\in F(K_y^n)$. For each subset $T \subset S$, let $\alpha_T \in \prod_{y \in S} K_y$ be the element that is $\alpha_y$ on the $y$ factors in $T$ and $0$ for those $y$ not in $T$. Using Proposition \ref{weakapprox}, choose for each $T$ an element $\beta_T \in K$ that is very close to $\alpha_T$. Then we claim that the $\beta_T$ yield distinct elements in $K/F(K^n) \simeq {\rm{H}}^1(K, U)$ for distinct $T$. It will then follow that ${\rm{H}}^1(K, U)$ has size at least $2^m$, and since $m > 0$ was arbitrary, this will complete the proof of the theorem.

It only remains to check that the $\beta_T$ do indeed yield distinct cohomology classes. Given $T \neq T'$, let $s \in S$ be a place where they differ -- say, $s \in T - T'$. Then $\beta_T - \beta_{T'}$ is close to $\alpha_s \in K_s$. Because $F$ is a smooth morphism, it induces an open map $K_s^n \rightarrow K_s$. It follows that every element of $K_s$ close to $\alpha_s$ lies in the same nontrivial cohomology class as $\alpha_s$. Therefore, $\beta_T - \beta_{T'} \not\in F(K_s^n)$, so also $\not\in F(K^n)$. It follows that $\beta_T - \beta_{T'}$ yields a nontrivial class in ${\rm{H}}^1(K, U)$, as desired.
\end{proof}

\begin{remark}
The smoothness assumption is crucial in Theorem \ref{infofcohom}. Indeed, the unipotent $K$-group $U := \R_{K^{1/p}/K}(\alpha_p)$ satisfies ${\rm{H}}^1(K, U) = 0$ because it is defined by a universal monogeneous $p$-polynomial by Proposition \ref{eqnforweilrestalphap}. One can say more about the cohomology of unipotent groups beyond the smooth case, but doing so here would take us too far afield, so we postpone such a discussion for now.
\end{remark}

\section{Unipotent quotients of commutative pseudo-reductive groups}
\label{unipquotcpredsec}

In this section we apply permawound unipotent groups to give a partial answer to Totaro's question \ref{totaroquestion} about commutative pseudo-reductive groups over fields of degree of imperfection $1$. More precisely, we will use permawound groups to provide a remarkably simple proof of the affirmative answer to Question \ref{totaroquestion} in the case in which $U$ is $p$-torsion. A treatment of general $U$ is significantly more involved, so we postpone a discussion of this more general case for now.

We begin by identifying the group $\mathscr{V}$ of Definition \ref{V_PLdef} when $k$ has degree of imperfection $1$. The following result is \cite[Ch.\,VI, Prop.\,5.3]{oesterle}.

\begin{proposition}
\label{eqnforweilrestgm}
Let $k$ be a field of degree of imperfection $1$, and let $\lambda \in k - k^p$. Then $\R_{k^{1/p}/k}(\Gm)/\Gm \simeq \mathscr{V}_{P, L}$, where $P := \sum_{i=0}^{p-1} \lambda^iX_i^p \in k[X_0, \dots, X_{p-1}]$ and $L := -X_{p-1}$.
\end{proposition}

The following theorem answers Question \ref{totaroquestion} in the affirmative when $U$ is $p$-torsion.

\begin{theorem}$($Theorem $\ref{totquesyesintro}$$)$
\label{totquesyes}
Let $k$ be a field of degree of imperfection $1$. Then every commutative $p$-torsion wound unipotent $k$-group is the maximal unipotent quotient of a commutative pseudo-reductive $k$-group.
\end{theorem}

\begin{proof}
By \cite[Lem.\,9.1]{totaro}, we may assume that $k$ is separably closed. Note that the conclusion of the theorem is preserved upon passage to subgroups: If $U' \subset U$ is an inclusion of commutative wound unipotent $k$-groups, and $U$ is the maximal unipotent quotient of the commutative pseudo-reductive group $C$, then $U'$ is the maximal unipotent quotient of the smooth connected commutative $k$-subgroup $C' := C \times_U U'$ of $C$, which is pseudo-reductive because $C$ is. Thanks to the ubiquity of permawound unipotent groups (Theorem \ref{ubiquity}), therefore, in proving the theorem we may assume that $U$ is permawound.

We proceed by induction on ${\rm{dim}}(U)$, the $0$-dimensional case being trivial. So assume that $U \neq 0$. By rigidity (Theorem \ref{filtrationasuniv}), there is an exact sequence
\[
0 \longrightarrow U' \longrightarrow U \xlongrightarrow{\pi} \overline{U} \longrightarrow 1
\]
with $\overline{U}$ wound (and permawound by Proposition \ref{quotofasunivisasuniv}) and $U'$ isomorphic to either $\R_{k^{1/p}/k}(\alpha_p)$ or $\mathscr{V}$. By induction, we may assume that $\overline{U}$ is the maximal unipotent quotient of a commutative pseudo-reductive group $\overline{C}$. If $U' \simeq \mathscr{V}$, then $U \simeq \overline{U} \times \mathscr{V}$ by Corollary \ref{ptorextbyvsplits}. Since $\mathscr{V}$ is the maximal unipotent quotient of the pseudo-reductive group $\R_{k^{1/p}/k}(\Gm)$ (Proposition \ref{eqnforweilrestgm}), it follows that $U$ is the maximal unipotent quotient of the commutative pseudo-reductive group $C := \overline{C} \times \R_{k^{1/p}/k}(\Gm)$.

Suppose, on the other hand, that $U' \simeq \R_{k^{1/p}/k}(\alpha_p)$. Then $U$ is the maximal unipotent quotient of the smooth connected commutative $k$-group $C := \overline{C} \times_{\overline{U}} U$ (because $C$ is an extension of $U$ by a torus, which also explains why it is smooth and connected), and we claim that $C$ is pseudo-reductive. We have an exact sequence
\[
0 \longrightarrow U' \longrightarrow C \xlongrightarrow{\pi} \overline{C} \longrightarrow 0.
\]
To see that $C$ is pseudo-reductive, suppose that $W \subset C$ is smooth connected unipotent. We must show that $W = 0$. We have $\pi(W) = 0$ because $\overline{C}$ is pseudo-reductive, so $W \subset U'$. But $U'$ is totally nonsmooth, so this implies that $W = 0$.
\end{proof}

\appendix

\section{Semiwound groups}
\label{semiwoundsection}

The notion of woundness plays a central role in the theory of smooth connected unipotent groups (and therefore of more general algebraic groups) over imperfect fields. It is useful to extend this notion to arbitrary unipotent groups, which we do in this brief appendix.

\begin{proposition}
\label{semiwoundconds}
Let $U$ be a unipotent group scheme over a field $k$. The following are equivalent:
\begin{itemize}
\item[(i)] There is no $k$-group inclusion $\Ga \hookrightarrow U$.
\item[(ii)] There is no nonzero $k$-group homomorphism $\Ga \rightarrow U$.
\item[(iii)] Every $k$-morphism $\A^1_k \rightarrow U$ is the constant map to some $u \in U(k)$.
\end{itemize}
\end{proposition}

\begin{proof}
We first note a trivial but useful reformulation of condition (iii) as the condition that the only $k$-morphism $\A^1_k \rightarrow G$ with $0 \mapsto 1 \in G(k)$ is the constant map to $1$. Each of the conditions for $U$ is equivalent to the same condition for the maximal smooth $k$-subgroup scheme of $U$ (see \cite[Lem.\,C.4.1, Rem.\,C.4.2]{cgp}), so we are free to assume that $U$ is smooth. Further, each condition for $U$ is equivalent to the same condition for the identity component of $U$. Thus we may further assume that $U$ is connected. The desired equivalence is then \cite[Prop.\,B.3.2]{cgp}.
\end{proof}

\begin{definition}
\label{semiwounddef}
We say that a unipotent group scheme $U$ over a field $k$ is {\em $k$-semiwound}, or just semiwound when $k$ is clear from context, when the equivalent conditions of Proposition \ref{semiwoundconds} hold for $U$.
\end{definition}

As with woundness, semiwoundness is insensitive to separable field extension.

\begin{proposition}
\label{semiwoundsepble}
Let $K/k$ be a separable extension of fields, and let $U$ be a unipotent $k$-group scheme. Then $U$ is $k$-semiwound if and only if $U_K$ is $K$-semiwound.
\end{proposition}

\begin{proof}
The if direction is clear. For the only if direction, we first note that semiwoundness is insensitive to replacing a group with its maximal smooth subgroup. Since the formation of the maximal smooth subgroup is insensitive to separable field extension \cite[Lem.\,C.4.1, Rem.\,C.4.2]{cgp}, we are free to assume that $U$ is smooth. Since semiwoundness is insensitive to replacing a group by its identity component, and since the formation of the identity component is insensitive to arbitrary field extension, we may also assume that $U$ is connected. The proposition then follows from \cite[Prop.\,B.3.2]{cgp}.
\end{proof}

\section{Breen's spectral sequences}
\label{breenseqapp}

In this appendix, we discuss a few spectral sequences constructed by Breen that we utilize in the proof of Lemma \ref{Ext^2explicit}. For more details, see \cite[\S1]{breen} and especially \cite[page 1250]{breen2}. Let $S$ be a scheme. Associated to any commutative $S$-group scheme $G$, there is a descending complex $A(G) = A(G)_{\bullet}$ of fppf abelian sheaves concentrated in nonnegative degrees such that each term of $A(G)$ is a product of sheaves of the form $\Z[G^n]$ (the sheaf freely generated by $G^n$). Further, $A(G)_0 = \Z[G]$, the canonical map $G \rightarrow A(G)_0$ induces an isomorphism 
\begin{equation}
\label{G=H_0(A)}
G \simeq {\rm{H}}_0(A(G)),
\end{equation}
and we have 
\begin{equation}
\label{H_1(A)=0}
{\rm{H}}_1(A(G)) = 0 \hspace{.3 in}
{\rm{H}}_2(A(G)) = G/2G.
\end{equation}
We also have
\[
A(G)_1 = \Z[G^2],
\]
\[
A(G)_2 = \Z[G^2] \oplus \Z[G^3],
\]
and the differential $A(G)_1 = \Z[G^2] \rightarrow \Z[G] = A(G)_0$ is the map induced by $m - \pi^2_1 - \pi^2_2:  G \times G \rightarrow G$, where $m, \pi^2_i:  G \times G \rightarrow G$ are the multiplication and projection maps, respectively, while the differential $A(G)_2 = \Z[G^2] \oplus \Z[G^3] \rightarrow \Z[G^2] = A(G)_1$ is induced on the first component by the map ${\rm{Id}} - \sigma\colon G^2 \rightarrow G^2$, where $\sigma\colon G^2 \rightarrow G^2$ is the switching map $(g_1, g_2) \mapsto (g_2, g_1)$, and on the second component the differential is induced by the map $$(\pi^3_1, m\circ \pi_{23}) + \pi_{23} - (m\circ \pi_{12}, \pi^3_3) - \pi_{12} \colon G^3 \rightarrow G^2,$$ where $\pi^3_i:  G^3 \rightarrow G$ is projection onto the $i$th factor, $\pi_{ij}:  G^3 \rightarrow G^2$ is projection onto the $i$th and $j$th factors, and $m:  G^2 \rightarrow G$ is once again multiplication.

Breen obtains first a spectral sequence
\[
F_1^{i, j} = \Ext^j(A(G)_i, H) \Longrightarrow \Ext^{i+j}(A(G), H).
\]
We have a canonical isomorphism $\Ext^j(\Z[G], \cdot) \simeq {{\rm{H}}}^j(G, \cdot)$ of functors
on fppf abelian sheaves, since both are the derived functors of $\Gamma(G, \cdot)$, thanks to Yoneda's Lemma. Thus, the sequence above becomes
\begin{equation}
\label{F_1specseq22}
F_1^{i,j} = {{\rm{H}}}^j(X_i, H) \Longrightarrow \Ext^{i+j}(A(G), H),
\end{equation}
where $X_i$ is some explicit disjoint union of products of copies of $G$, and in particular, 
\begin{equation}
\label{X_idescription}
X_0 = G \hspace{.3 in}  X_1 = G^2 \hspace{.3 in} X_2 = G^2 \coprod G^3
\end{equation}
with the differentials $F_1^{0, j} \rightarrow F_1^{1, j}$ being 
\begin{equation}
\label{breenX_0to1differential}
m^* - (\pi^1_1)^* - (\pi^2_2)^*,
\end{equation}
and the differential $F_1^{1, j} \rightarrow F_1^{2,j}$ being
\begin{equation}
\label{breenX_1to2differential}
(1^* - \sigma^*) \times (\pi^3_1, m\circ \pi_{23})^* + \pi_{23}^* - (m\circ \pi_{12}, \pi^3_3)^* - \pi_{12}^*.
\end{equation}

In fact, Breen shows that we may replace the above sequence with another one that is somewhat more convenient, involving ``reduced'' cohomology groups $\widetilde{\rm{H}}^j(X_i, H)$ defined as follows. Let $Y_i$ be the analogue of $X_i$ for the $0$ group; that is, $Y_i$ is a corepresenting object for $\Hom(A(0)_i, \cdot)$ (so $Y_i$ is a disjoint union of copies of $S$). Then via the identity section $S \rightarrow G$, we obtain maps $Y_i \rightarrow X_i$, and we define $\widetilde{\rm{H}}^j(X_i, H) : = \ker({{\rm{H}}}^j(X_i, H) \rightarrow {{\rm{H}}}^j(Y_i, H))$ to be the kernel of the induced map on cohomology. Breen proved that these reduced cohomology groups provide a spectral sequence analogous to (\ref{F_1specseq22}) and with the same abutment. That is, we have a spectral sequence
\begin{equation}
\label{E_1specseqbreen20}
E_1^{i,j} = \widetilde{{{\rm{H}}}}^j(X_i, H) \Longrightarrow \Ext^{i+j}(A(G), H).
\end{equation}

The second spectral sequence constructed by Breen takes the following form: 
\begin{equation}
\label{E_2spectralseqbreen21}
{}^{\prime}E_2^{i,j} = \Ext^i_S({\rm{H}}_j(A(G)), H) \Longrightarrow \Ext^{i+j}_S(A(G), H).
\end{equation}

\noindent \address
\vspace{.3 in}

\noindent \email


\begin{thebibliography}{ram}

\bibitem[Bor]{borel} Armand Borel, {\em Linear Algebraic Groups}, 2nd edition, Springer GTM 126, 1991.

\bibitem[Bre1]{breen2} Lawrence Breen, {\em On A Nontrivial Higher Extension of Representable Abelian Sheaves}, Bulletin of the American Mathematical Society, vol. 75, no. 6 (1969), pp. \, 1249-1253.

\bibitem[Bre2]{breen} Lawrence Breen, {\em Un th\'eor\`eme d'annulation pour certain $\Ext^i$ de faisceaux ab\'eliens},
Annales scientifiques de l'E.N.S.\,8(3) (1975), pp.\,339-352.

\bibitem[Con]{conradsolvable} Brian Conrad, ``The structure of solvable groups'' in {\em Autour des sch\'emas en groupes} (vol. II), Panoramas et Synth\`eses no. 46, Soc. Math. de France, 2015.

\bibitem[CGP]{cgp} Brian Conrad, Ofer Gabber, Gopal Prasad, {\em Pseudo-reductive Groups}, Cambridge Univ.\,Press (2nd edition), 2015.

\bibitem[DG]{demazuregabriel} Michel Demazure, Pierre Gabriel, {\em Groupes Alg\'ebriques}, North-Holland Publishing Company, 1970.

\bibitem[NNR]{nguyensros} {\fontencoding{T5}\selectfont Nguy\~ \ecircumflex n Duy T\^an}, {\fontencoding{T5}\selectfont  Nguy\~ \ecircumflex n Qu\tabb'\ocircumflex c Th\' \abreve ng}, Zev Rosengarten, {\em On The Galois And Flat Cohomology Of Unipotent Algebraic Groups Over Local And Global Function Fields II}, to appear in The Michigan Mathematical Journal.

\bibitem[Oes]{oesterle} Joseph Oesterl\'e, {\em Nombres de Tamagawa et Groupes Unipotents
en Caract\'eristique $p$}, Inv.\:Math.\,{\bf 78}(1984), 13--88.

\bibitem[Rosengarten1]{rostateduality} Zev Rosengarten, {\em Tate Duality In Positive Dimension Over Function Fields}, to appear in Memoirs of The American Mathematical Society, available at {\tt{https://arxiv.org/pdf/1805.00522.pdf}}.

\bibitem[Rosengarten2]{rosmodulispaces} Zev Rosengarten, {\em Moduli Spaces of Morphisms Into Solvable Algebraic Groups}, to appear in Algebra And Number Theory, available at {\tt{https://arxiv.org/pdf/2107.08320.pdf}}.

\bibitem[Rosenlicht]{rosenlicht} Maxwell Rosenlicht, {\em Some Rationality Questions On Algebraic Groups}, Annali di Matematica Pura ed Applicata volume 43, pp.\,25--50 (1957).

\bibitem[Tit]{tits} Jacques Tits, {\em Lectures On Algebraic Groups}, Yale University Press, New Haven, CT, 1967.

\bibitem[TV]{tossici} Dajano Tossici, Angelo Vistoli, {\em On The Essential Dimension Of Infinitesimal Group Schemes}, American Journal of Mathematics Vol.\,135, No.\,1 (February 2013), pp.\,103-114.

\bibitem[Tot]{totaro} Burt Totaro, {\em Pseudo-abelian Varieties}, {\em Annales Scientifiques de l'École Normale Supérieure} 46(5) (2013), pp.\,693-721.

\end{thebibliography}
\end{document}